\documentclass[11pt]{article}
\usepackage{amsbsy}
\usepackage{amsmath,amsthm,amsfonts,amssymb}
\numberwithin{equation}{section}

\usepackage{bm}
\usepackage{hyperref}
\date{\today}
\usepackage{cite}
\usepackage{array}
\usepackage{xcolor}

\hypersetup{
	colorlinks=true,
	linkcolor=blue,
	filecolor=dviolet,
	urlcolor=blue,}

    \newtheorem{theorem}{Theorem}
    \newtheorem{lemma}{Lemma}
    \newtheorem{proposition}{Proposition}
    \newtheorem{corollary}{Corollary}

\theoremstyle{definition} 

    \newtheorem{remark}{Remark}
    \newtheorem{example}[theorem]{Example}
    \newtheorem{exercise}[theorem]{Exercise}
    
    \newtheorem{assumption}[]{Assumption}

\def\newblock{\hskip .11em plus .33em minus .07em}

\def\BB{{\mathcal B}}
\def\AA{{\mathcal A}}

\def\d{\delta}
\def\e{\epsilon}

\def\vp{\varphi}

\def\Z{\mathbb{Z}}

\def\R{\mathbb{R}}

\def\Z{\mathbb{Z}}

\def\sgn{{\mb{sgn}}}

\def\l{\left}
\def\r{\right}
\def\<{\langle}
\def\>{\rangle}

\def\mb{\mbox}
\newcommand{\one}{{\mathbf 1}}
\newcommand{\E}{\mbox{\bf E}}

\def\bar{\overline}

\newcommand\Tr{{\mbox{Tr}}}

\newcommand\mnote[1]{} 
\newcommand\be{\begin{equation*}}

\newcommand\ee{\end{equation*}}

\newcommand\ben{\begin{equation}}
\newcommand\een{\end{equation}}
\newcommand\bes{\begin{eqnarray*}}
\newcommand\ees{\end{eqnarray*}}

\newcommand\bex{\begin{exercise}}
\newcommand\eex{\end{exercise}}
\newcommand\beg{\begin{example}}
\newcommand\eeg{\end{example}}
\newcommand\benu{\begin{enumerate}}
\newcommand\eenu{\end{enumerate}}
\newcommand\beit{\begin{itemize}}
\newcommand\eeit{\end{itemize}}
\newcommand\berk{\begin{remark}}
\newcommand\eerk{\end{remark}}
\newcommand\bdefn{\begin{defintion}}
\newcommand\edefn{\end{definition}}
\newcommand\bthm{\begin{theorem}}
\newcommand\ethm{\end{theorem}}
\newcommand\bprf{\begin{proof}}
\newcommand\eprf{\end{proof}}
\newcommand\blem{\begin{lemma}}
\newcommand\elem{\end{lemma}}

\newcommand{\sm}{{\raise0.3ex\hbox{$\scriptstyle \setminus$}}}

\def\mb{\mbox}
\def\l{\left}
\def\r{\right}

\def\CHI{\mathchoice%
{\raise2pt\hbox{$\chi$}}%
{\raise2pt\hbox{$\chi$}}%
{\raise1.3pt\hbox{$\scriptstyle\chi$}}%
{\raise0.8pt\hbox{$\scriptscriptstyle\chi$}}}
\def\smalloplus{\raise1pt\hbox{$\,\scriptstyle \oplus\;$}}

\numberwithin{equation}{section}

\title{Convergence of high dimensional Toeplitz and related matrices with correlated inputs}

\usepackage{authblk}

\author[]{Kartick Adhikari\thanks{kartick@iiserb.ac.in\ :  Research supported by Inspire Faculty Fellowship,  DST/INSPIRE/04/2020/000579.}}
\author[]{Arup Bose\thanks{bosearu@gmail.com\ : Research supported by J.C. Bose Fellowship JBR/2023/000023 from Anusandhan National Research Foundation (ANRF) Government of India.}}
\author[]{Shambhu Nath Maurya\thanks{shambhumath4@gmail.com\ : Research partially supported by NBHM Post-doctoral Fellowship, 0204/10/(25)/2023/R$\&$D-II/2803.  Work partially done at IISER Bhopal,  funded by DST/INSPIRE/04/2020/000579.}}
	
\affil[]{$^{\ast}$  
Dept.~of Math., Indian Institute of Science Education and Research, Bhopal, 	India}
\affil[]{$^{\dagger}$ $^{\ddagger}$  Statistics and Mathematics Unit, Indian Statistical Institute, Kolkata, India}

\date{\today}

\begin{document}

\maketitle
\vskip-10pt
\begin{abstract} 
  We investigate the joint convergence of independent random Toeplitz matrices with complex input entries that have a pair-correlation structure, along with deterministic Toeplitz matrices and the backward identity permutation matrix. 
Further, we study  the joint convergence of independent  generalized Toeplitz matrices along with other related matrices.
The limits depend only on the correlation structure but are universal otherwise, in that they do not depend on the underlying distributions of the entries. In particular, these results provide the joint convergence of asymmetric Hankel matrices. Earlier results in the literature on the joint convergence of random symmetric Toeplitz and symmetric Hankel matrices with real entries follow as special cases.
\end{abstract}


\vskip5pt


\vskip3pt

\section{Introduction} 

\subsection{Non-random Toeplitz and Hankel matrices} \label{subsec:toe+Han_def}
This article is on sequences of high dimensional Toeplitz and Hankel matrices. 
Let $\Z$ be the set of all integers.
The $n\times n$ Toeplitz matrix $T_n$  with \textit{input sequence} $\{a_{i,n};  i\in \Z\}$,  is defined as
\begin{align*}
T_n=\left(\begin{array}{ccccc}
a_{0,n} & a_{-1,n} & a_{-2,n} & \cdots & a_{1-n,n}\\
a_{1,n} & a_{0,n} & a_{-1,n} & \cdots & a_{2-n,n}\\
a_{2,n} & a_{1,n} & a_{0,n} & \cdots & a_{3-n,n}
\\ \vdots &\vdots &\vdots & \cdots & \vdots\\
a_{n-1,n}& a_{n-2,n}& a_{n-3,n}& \cdots & a_{0,n}
\end{array}  \right).
\end{align*}
We shall write $a_i$ for $a_{i,n}$. 
In short we can write $T_n=((a_{i-j}))_{i,j=1}^n$. Note that $T_n$  is not symmetric. We shall also work with its symmetric version, namely, 
$T_{n,s}=((a_{|i-j|}))_{i,j=1}^n$.

Likewise, the general $n\times n$  Hankel matrix $H_n$  with the input sequence $\{a_{i,n}; i\in \Z\}$ is
\begin{align} \label{eqn:Hankel}
H_n=\left(\begin{array}{cccccc}
a_{2,n} & a_{-3,n} & a_{-4,n} & \cdots & a_{-n,n} & a_{-(n+1),n}\\
a_{3,n} & a_{4,n} & a_{-5,n } & \cdots & a_{-(n+1),n} & a_{-(n+2),n}\\
a_{4,n} & a_{5,n} & a_{6,n} &\cdots & a_{-(n+2),n} & a_{-(n+3),n}\\
\vdots & \vdots & \vdots & \cdots &\vdots&   \vdots\\
a_{n,n} & a_{n+1,n} & a_{n+2,n} & \cdots&a_{2n-2,n} & a_{-(2n-1),n}\\
a_{n+1,n} & a_{n+2,n} & a_{n+3,n} & \cdots & a_{2n-1,n} & a_{2n,n}
\end{array}  \right).
\end{align}
In short we can write $H_n=((a_{(i+j)\sgn(i-j)}))_{i,j=1}^n$, where
\begin{equation*} 
	\sgn(\ell)= \l\{\begin{array}{rll}
		1 & \mbox{ if } & \ell \geq 0,\\
		-1& \mbox{ if } & \ell<0.
	\end{array} \r.
\end{equation*}
 Note that $H_n$  is not symmetric. 
Its symmetric version is $H_{n,s}=((a_{i+j}))_{i,j=1}^n$.

The Hankel and the Toeplitz matrices are related to each other through certain deterministic matrices as given below. Let $P_n$ be the $n\times n$ {\it backward identity} permutation matrix defined as
	$$
P_n=\left[\begin{array}{cccccc}
	0 & 0 & 0 & \ldots & 0 & 1 \\
	0 & 0 & 0 & \ldots & 1 & 0 \\
	\vdots & \vdots & \vdots & \vdots & \vdots & \vdots  \\
	0 & 1 & 0 & \ldots & 0 & 0 \\
	1 & 0 & 0 & \ldots & 0 & 0
\end{array}\right].
$$
Then it is easy to check that $P_nT_n$ is a symmetric Hankel matrix. Conversely, for any $H_{n,s}$, $P_n H_{n,s}$ is a  Toeplitz matrix. In this article, our symmetric Hankel matrices are always considered to be of the form $P_n T_{n}$. Similarly, asymmetric Hankel matrices are also related to some form of Toeplitz matrices and $P_n$. The details of this relation are given in Section \ref{subsec:gen_toe}.

 We shall consider only the high dimensional case, namely where $n\to \infty$. 
 Let us first consider these matrices with non-random entries.
Such Toeplitz matrices appear in many areas, such as numerical analysis, signal processing, time series analysis, etc. 
 They are very well-known in operator theory. 
 For early works on the spectral properties of these matrices, especially when $n\to \infty$, see \cite{grenander_szego_book_84}, \cite{bottcher+Silbermann_book_90},
and \cite{bottcher_analysisofToeop_book_06}.
  The recent survey article \cite{basor_harold_toe+oper_rev_22} presents the background history and development of the Toeplitz matrix. For further information on high dimensional Toeplitz matrices, see \cite{bottcher+silbermann_Trunc+Toe_book_99}, \cite{bottcher_analysisofToeop_book_06}, \cite{nikolski_Toe+mart+ope_book_20}, 
 \cite{basor_harold_toe+oper_rev_22} and the references therein.
Non-random Hankel matrix also  appears crucially in operator theory and related areas. It is used to check the solvability of the \textit{Hamburger moment problem} (see \cite{shohat_moment+problem_43}). In signal processing, for a discrete signal, $X=(x_1, x_2, \ldots, x_n)$, a Hankel matrix $H_n=((x_{i+j}))$ is  created by this signal.
The singular value decomposition (SVD) of $H_n$ is used in the analysis of the signal,  
see  
\cite{xuezhi_signal_procss_hankel_09} and \cite{jain_signal+analysis_15}, for further use of SVD of the Hankel matrix in 
signal processing. 
 
 
 In information theory, the memory-less channel of the form $z= Ax + \epsilon$, is used, where $x$ is an 
input signal, $z$ is the 
output signal, $\epsilon$ is an additive noise and $A$ is a suitable $n \times p$ random matrix. 
The matrix $AA^{\prime}$ plays a crucial role in the analysis of such systems. 
 For more details, see \cite{Tulino_2004_RMT_Wireless_Commu}, \cite{couillet+debbah_RMT_Wireless_book_11}, \cite{geLiang} and references therein.

 Before we discuss random Toeplitz and Hankel matrices, 
 we need the notions of convergence of the spectral distribution and
 the joint algebraic convergence of several matrices when taken together.

\subsection{\texorpdfstring{$*$}{} probability space and algebraic convergence}
A $*$-probability space is a pair $(\mathcal A,\vp)$ where $\mathcal A$ is a unital $*$-algebra (with unity $\one_{\mathcal A}$) over complex numbers,  and  $\vp$ is a linear functional such that $\vp(\one_{\mathcal A})=1$ and $\vp(aa^*)\ge 0$ for all $a\in \mathcal A$. The state is called tracial if $\vp (ab)=\vp(ba)$ for all $a, b \in \mathcal{A}$.

Let $\mathcal{A}_n$ be the matrix algebra of $n\times n$ random matrices (the operation $*$ is one of taking complex conjugate), whose input entries have all moments finite, 
with the state 
\begin{align*}
\vp_n(B_n)&:=\frac{1}{n}\E[\Tr(B_n)], \;\mbox{where $\Tr(B_n)=\sum_{i=1}^{n}b_{ii}$ when $B_n=((b_{ij}))_{n\times n}\in \mathcal A_n$}.
\end{align*}
Suppose $\{A_{i,n}; 1\leq i \leq p\}$ are  $p$ sequences of $n\times n$ random matrices. 
 We say that 
they {\it converge jointly in $*$-distribution} to some elements $\{a_i; 1\leq i \leq p\}\in (\mathcal A,\vp)$  
if for every choice of $k \geq 1$, $\epsilon_1,\epsilon_2,\ldots, \epsilon_k\in \{1, *\}$ and $i_1, \ldots, i_k \in \{1, \ldots , p\}$, we have 
\begin{equation}\label{eq:star_conv}
\lim_{n\to \infty}\varphi_n(A_{i_{1},n}^{\epsilon_1}\cdots A_{i_{k},n}^{\epsilon_k})=\varphi(a_{i_{1}}^{\epsilon_1}\cdots a_{i_{k}}^{\epsilon_k}).
 \end{equation}
Then we write $(A_{i,n}; 1\leq i\leq p)\stackrel{*\mbox{-dist}}{\longrightarrow} (a_i; 1\leq i \leq p)$. 
 
If the matrices are \textit{real symmetric}, then taking $*$ is redundant. The quantities  $\{\varphi_n(\cdot)\}$ and $\{\varphi(\cdot)\}$ are referred to as $*$-moments of the respective variables. A moment is said to be odd if $k$ is odd.
It is important to note that if the limit on the left side of \eqref{eq:star_conv} exists for all choices, then these limits automatically define a $*$-probability space with $p$ indeterminates. 

\subsection{Spectral distribution and its convergence}
Let $A_n$ be an $n\times n$ random matrix with eigenvalues $\lambda_{1,n},\ldots, \lambda_{n,n}$.
Then the {\it empirical spectral distribution} (ESD) of $A_n$ is defined as 
\begin{align*}
F^{A_n}(x,y)=n^{-1}\#\{k\;:\; \Re(\lambda_{k,n})\le x, \Im(\lambda_{k,n})\le y\},\;\;\;\mbox{ for $x,y\in \R$},
\end{align*} 
where $\# A$ denotes the cardinality of $A$, and $\Im(x)$ and $\Re (x)$ denote the imaginary and real parts of $x$ respectively.
While $F^{A_n}$ is a random distribution function, its expectation  $\E[F^{A_n}(x,y)]$ (called the EESD) is a non-random  distribution function. 

If as $n \to \infty$, the ESD and/or the EESD converges weakly (a.s.~or in probability for the ESD)
to a distribution function $F_{\infty}$, then the limit(s) are commonly referred to as the  {\it limiting spectral distribution} (LSD) of $A_n$. 

The two notions of convergence are related as follows: suppose $\{A_n\}$ are hermitian and converge to $a$ in the above sense where the moments $\{\varphi(a^k)\}_{k \geq 1}$ determine a unique probability distribution $\mu$ say. Then the EESD of $A_n$ converges weakly to $\mu$.

\subsection{Random Toeplitz and Hankel matrices}
Random Toeplitz and Hankel matrices were introduced in the review paper \cite{bai_zd_99}.
 First consider the symmetric versions, $T_{n,s}$ and $H_{n,s}$ where   
$\{a_j\}$ are \textit{real} i.i.d.~with finite variance. 
The a.s.~convergence of the ESD of $n^{-1/2} T_{n,s}$  and $n^{-1/2} H_{n,s}$ have been proved by Hammond and Miller \cite{hammond_miller_05}, and  Bryc et al. \cite{bryc_lsd_06}. These limits are universal (do not depend on the underlying distribution of $a_j$).
Sen and Virag \cite{sen+virag_abscont+toe_11} proved that in the Toeplitz case, the limit distribution is absolutely continuous with respect to the Lebesgue measure with a bounded density. 

These results have been extended in two ways. For LSD results on $T_{n,s}$ and $H_{n,s}$, when $\{a_j\}$ are independent but not necessarily identically distributed, see  
\cite{bose+saha+sen_pattern_RMTA_21} and the references therein. 
For the LSD results on these matrices when $\{a_j\}$ is an infinite order moving average, see \cite{bose_sen_LSD_EJP}, and when $\{a_j\}$ is the sum of finitely many independent copies, see \cite{maurya_LED+Toep_dep_24}.  

The problem of convergence of the LSD in the asymmetric case is hard, and is yet unsolved. To gauge this difficulty, consider the  so-called IID matrix $C_n$ 
(this may also be viewed as the non-symmetric version of the Wigner matrix). After a long series of partial results, \cite{mehta_RM+ST_book_67}, \cite{bai_circular+law_annals_97}, \cite{girko_circular+20Y+iii_05} and \cite{gotze_circular+law_annals_10}, it was eventually established in   \cite{tao_circular+IID_10} and \cite{tao_littlewood+circular_09}  that the LSD of $n^{-1/2}C_n$ is the uniform law on the unit disc when the entries are i.i.d.~with unit variance. A survey article \cite{bordenave_around+circular+law_12} is dedicated to this problem. 

There are severe technical difficulties in extending these proofs to the asymmetric Toeplitz and Hankel matrices. First, the number of independent entries in $C_n$ is $O(n^2)$ while  for $T_n$ and $H_n$ there are only $O(n)$ independent entries. Further, while simulations convince us that the LSDs exist, unlike $C_n$ this LSD is unbounded. 


Incidentally, one of the first steps in the proof of the above uniform law result is to show that the empirical distribution of the singular values of $C_n$, that is, ESD of $n^{-1}C_nC_n^\prime$ exists. The parallel result for the Toeplitz and the Hankel matrices is known from the work of 
\cite{bose+gango+sen_XX'_10}, 
when $\{a_j\}$ is \textit{real} i.i.d.~ with some finite moments. 
\cite{bose+priyanka_XX'_22} proved general results where $\{a_j\}$ are real and independent, but not necessarily identically distributed. 

\subsection{Joint convergence of random matrices}
 The work of  Voiculescu \cite{Dan_wigner_JC} was revolutionary. He showed that independent Wigner matrices converge jointly and remarkably, the limit variables are \textit{freely independent}. Since then there have been many works on the joint convergence of other matrices. For an overview, see \cite{rayan_JC_limitnotfree}, 
\cite{basu_bose_JC_patt}, \cite{mingo_freeprob_book} and \cite{bose_freeprob}.

\cite{bose_saha_patter_JC_annals}
presented a unified approach for the joint convergence of the five symmetric random matrices\textemdash Wigner, symmetric Toeplitz, symmetric Hankel, reverse circulant and symmetric circulant. The overriding assumption in these works is that the input variables that are used to populate these matrices are independent and real valued. 

The Wigner matrix with a pair-correlated structure assumes that the pairs $(w_{i,j}, w_{j, i})$ are i.i.d. with a correlation $\rho$ say. It is also known as the \textit{elliptic matrix}. This matrix has an LSD which is uniformly distributed on the interior of an ellipse, centered at the origin with major and minor axes that depend on $\rho$. The special case of $\rho=0$ yields the uniform law. The special case of $\rho=1$ reduces to the standard Wigner matrix. 
See \cite{sommer+etal_spectrum+asymm_RM_PRL_88}, \cite{gotze+etal_RM+correlatedentry_RMTA_15} and \cite{nguyen+orourke_elliptic+law_IMRN_15}. 

The joint convergence of independent elliptic matrices and their asymptotic freeness was established in \cite{bose+adhikari_Brownmeasure_19}. The \textit{cross-covariance matrix} is defined as $n^{-1}XY^{\prime}$. 
Assuming that the pairs $(x_{i,j}, y_{i, j})$ are i.i.d., \cite{monika+bose+dey_JC+samplecov_ALEA_23} proved the joint convergence of independent cross-covariance matrices.  

One may also allow some specific pattern of correlation instead of the constant correlation. See \cite{ziliang_RM+gen+correletedentry_EJP_17} for some of these models. 

\subsection{Joint convergence of Toeplitz and Hankel matrices}
We finally turn to the question of joint convergence of these two matrices. 
Apart from the theoretical curiosity as to what extent the joint convergence of independent Wigner matrices can be extended to the joint convergence of Toeplitz and/or Hankel matrices, there are other reasons why this becomes important. 

First, we have already seen the connection between the convergence of the EESD and $*$-convergence. In the absence of any LSD results for the non-symmetric Toeplitz and Hankel matrices, the question of algebraic convergence gains importance.

In Section \ref{subsec:toe+Han_def}, we have seen the connection between the Toeplitz and Hankel matrices ($H_{n,s}=P_nT_n$). From this relation, it is clear that the problem of joint convergence of symmetric Hankel matrices is connected to the problem of joint convergence of $P_n$ and $T_n$. We have already mentioned that the joint convergence  of i.i.d.~copies of $n^{-1/2} T_{n,s}$ and $n^{-1/2} H_{n,s}$ was 
established in \cite{bose_saha_patter_JC_annals} when $\{a_j\}$ is real i.i.d.~and all its moments are finite. 
The deterministic Toeplitz and the deterministic Hankel matrices are of the form $D_n=((d_{i-j}))_{i,j=1}^n$, and $P_nD_n$ respectively, where $\{d_k; k \in \mathbb{Z}\}$ is a sequence of complex numbers. 
Joint convergence of these matrices and some of their interesting extensions were covered in \cite{bose+sreela+saha_Toe+JC_13}.

Let us now turn to the random Topelitz and Hankel matrices, and look more closely at the $(i,j)$th and $(j,i)$th entries of the Toeplitz matrix. First, note that in the asymmetric case, we assume that they are independent of each other. In the symmetric case we assume that they are identical. Moreover, anticipating the expected universality of the limits, let us for the moment, assume that this pair of entries is Gaussian. Then these two models can be brought under the unified umbrella where we assume they are jointly Gaussian with a correlation $\rho$: values of 
$\rho=1$ and $\rho=0$ yield the above two models. This leads to the consideration of the 
Toeplitz and Hankel matrices with correlated entries. 


Due to the above reasons, we are naturally led to the broad question of joint convergence of deterministic and/or random Toeplitz and Hankel matrices with correlated entries.  
We introduce specific patterns of correlation in these matrices and also allow the entries to 
be complex-valued and study the joint convergence of independent copies of these matrices. 
This gives several different deterministic and random matrix models and we explore their joint convergence. All previous results on the joint convergence of random and deterministic Toeplitz and Hankel matrices are special cases of the results that we shall establish. 

Finally, if we consider any symmetric matrix polynomial $Q_n$  obtained from the classes of matrices that we discuss, then our general joint convergence results, the tracial moments of $Q_n$ converge. Then by appealing to Gaussian bounds for these moments, we can conclude that ESD of $Q_n$ converges. See Corollary \ref{cor:poly_Lsd_T} given later. It may be noted this result is hard to prove using any other method such as the Stieltjes transform.

\section{Main results} 
\label{sec:mainresults}
\subsection{Results under pair correlated structure}
We first introduce a pair-correlation structure on the input sequence
through the following assumptions. 
%
 
 \begin{assumption} \label{assump:toeI}
Suppose $\{a_{j,n}= x_{j,n}+ \mathrm{i} y_{j,n}; j \in \mathbb{Z}\}$ are \textit{complex} random variables with $\E(a_{j,n})=0$, 
and for every $n\geq 1$, $\{(x_{j,n}, x_{-j,n}, y_{j,n}, y_{-j,n})\}$ are independent across $j$. Further, for all $n \geq 1$, 
\begin{enumerate}
	\item [(i)] $\E[x_{j,n}]^2=\sigma_x^2$, $\E[y_{j,n}]^2=\sigma_y^2$ \ \text{for all} $j\in \Z$.
	\item[(ii)] $\E[x_{j,n}y_{j,n}]=\rho_1$, $\E[x_{j,n}x_{-j,n}]=\rho_2$, $\E[x_{j,n} y_{-j,n}]=\rho_3$, \\
	$\E[x_{-j,n} y_{j,n}]= \rho_4, \ \E[y_{j,n} y_{-j,n}]= \rho_5, \ \E[x_{-j,n} y_{-j,n}]= \rho_6$ \ \text{for all} $j \geq 0$.  
	\item[(iii)] $ \displaystyle \sup_{j,n} \{\E(|x_{j,n}|^k),  \E(|y_{j,n}|^k)\} \leq c_k < \infty \ \ \text{for all $j\in \Z$ and } k\geq 1.$ 
\end{enumerate}
\end{assumption}
 For brevity, henceforth we write $a_{j}, x_{j}, y_{j}$ respectively for $a_{j,n}, x_{j,n}, y_{j,n}$.
We must have $0 \leq \rho_i\leq 1$ for all $1\leq i \leq 6$, since we have an infinite sequence of equally correlated variables.  Note that Assumption \ref{assump:toeI} holds if $\{a_j\}$ is i.i.d.~$\E(a_j)=0$, $\E|a_j|^2=1$, $\E|a_j|^k < \infty$ for $k\geq 1$. 

 The \textit{Hermitian} Toeplitz and Hankel matrices, $T_{n,h}$ and $H_{n,h}$ are obtained by imposing $a^*_{i-j}=a_{j-i}$ in $T_n$ and $H_n$. The \textit{real symmetric} Toeplitz and Hankel matrices, $T_{n,s}$ and $H_{n,s}$ are obtained  
by taking $y_j=0$ and $\rho_2=\sigma_x^2$.
The \textit{real asymmetric} Toeplitz and Hankel matrices are obtained by taking $y_j=0$.
 No restriction is required on $\rho_2$.

In this article we work with deterministic Toeplitz and Hankel matrices that satisfy the stronger summability condition than the square summability assumed in \cite{szego_toeplitz_915}. This is helpful for the technical arguments in the proof of our results.
 \vskip5pt
 \begin{assumption} \label{assump:determin}
 	Let $\{d_k; k \in \mathbb{Z}\}$ be a sequence of complex numbers which is absolutely summable, that is, $\sum_{k= -\infty}^{\infty} |d_k| < \infty$.
 \end{assumption} 
 Our first main result provides the joint convergence of $D_n$, $P_n$ and independent copies of $T_n$. Its proof is given in Section \ref{sec:jc_tdp}.
 \begin{theorem}\label{thm:JC_tdp}
 Let $A_{n,1}=n^{-1/2}T_n$ where $T_n$ is a random Toeplitz matrix. For $i=1,2, \ldots, m$, let $\{A_{n,1}^{(i)}\}$ be independent copies of $A_{n,1}$ and  $A_{n,2}^{(i)}=D^{(i)}_n$ be $m$ deterministic Toeplitz matrices. Suppose the input entries of $T_n$ satisfy Assumption \ref{assump:toeI} and the input entries of $D^{(i)}_n$   satisfy Assumption \ref{assump:determin}.
 Then $\{P_n, A_{n,j}^{(i)}; 1\leq j \leq 2, 1\leq i\leq m\}$ converge jointly.
  In particular, $\{P_n, n^{-1/2}T_n, D_{n} \}$ converge jointly.
 \end{theorem}
 
The following corollary follows from  Theorem \ref{thm:JC_tdp}:
 \begin{corollary} \label{cor:jc_dtp_hankel}
 	Suppose $\{T_{n}^{(i)}; 1\leq i\leq m\}$ and $\{H_{n,s}^{(i)}; 1\leq i\leq m\}$ are $m$ independent copies respectively of random Toeplitz and symmetric Hankel matrices, where the input entries satisfy Assumption \ref{assump:toeI}. Also, let $\{D_{n}^{(i)}; 1\leq i\leq m\}$ and $\{\tilde{H}_{n,s}^{(i)}; 1\leq i\leq m\}$ respectively be deterministic Toeplitz and deterministic symmetric Hankel matrices, whose input entries satisfy Assumption \ref{assump:determin}. Let $X^{(i)}_n$ and $Y^{(i)}_n$ be any of the matrices; $n^{-1/2}T_{n}^{(i)}$, $n^{-1/2}H_{n,s}^{(i)}$, $D_{n}^{(i)}$, $\tilde{H}_{n,s}^{(i)}$ and $P_n$. Then $\{X^{(i)}_n, Y^{(i)}_n; 1 \leq i \leq m\}$ converge jointly.
 \end{corollary}
   As a special case of Theorem \ref{thm:JC_tdp}, when the input entries are real, under Assumption \ref{assump:toeI}, independent copies of $n^{-1/2}T_n$ and $n^{-1/2}H_{n,s}$
 converge jointly, and the limit depends only on $\rho_2$. 
 This implies the joint convergence results of \cite{bose_saha_patter_JC_annals} for the special case of symmetric Toeplitz and Hankel matrices with independent entries.
 The LSD results of \cite{hammond_miller_05} and  \cite{bryc_lsd_06} when the input is i.i.d.~and only the finiteness of the second moment is assumed, also follow via a truncation argument.

   Joint convergence results 
   in the following matrix models also follow 
from Theorem \ref{thm:JC_tdp}:
(a) random Toeplitz matrices with scaling $n^{-1/2}$ and complex input entries (see Proposition \ref{thm:toeplitz} in Section \ref{subsec:jc_T}); (b) random Hermitian and symmetric Toeplitz matrices scaled by $n^{-1/2}$ (see Corollaries \ref{cor:A2+3_TJc} and \ref{cor:iid_jc_T} in Section \ref{subsec:jc_T}); (c) matrices $\{D_n, P_n\}$ (see Proposition \ref{thm:jc_D+P} in Section \ref{subsec:dp}); (d) matrices $\{D_n, n^{-1/2}T_n\}$ (see Proposition \ref{thm:jc_T+D} in Section \ref{subsec:td}); (e) matrices $\{n^{-1/2}T_n, P_n\}$ (see Proposition \ref{thm:jc_T+P} in Section \ref{subsec:Tp}), and (f) symmetric Hankel matrices (see Remark \ref{cor:com_HJc} in Section \ref{subsec:Tp}).

Further, 
if we take any \textit{symmetric} matrix polynomial of matrices that converge jointly, then its LSD will also exist.
See Corollaries \ref{cor:poly_Lsd_T} and 
\ref{cor:A2+3_HJc} respectively in Sections \ref{subsec:jc_T} and \ref{subsec:Tp}.
 
   \subsection{Extended model} \label{subsec:gen_toe}
While the characteristic feature of the Toeplitz and the Hankel matrices is the repetitions of the input entries over (parts of) each diagonal or anti-diagonal, let us consider a further generalization as follows: Define  $T_{n,g}=((t_{i,j}))_{n\times n}$, where
 \begin{equation} \label{def:T_gen}
 	t_{i,j}=\l\{\begin{array}{ccl}
 		a_{i-j} & \mbox{ if } & i+j\le n,
 		\\ \\
 		b_{i-j} & \mbox{ if } &  i+j\ge n+1.
 	\end{array}
 	\r.
 \end{equation}
We call it a \textit{generalized Toeplitz matrix}. For example,  $T_{5,g}$ is given by
 \[
 T_{5,g}=\l(\begin{array}{ccccc}
 	a_0 & a_{-1} &a_{-2} & a_{-3} & b_{-4} \\
 	a_1 & a_{0} & a_{-1} & b_{-2} & b_{-3}  \\
 	a_2 & a_{1} & b_0 & b_{-1} & b_{-2}\\
 	a_3 & b_2 & b_{1} & b_0 & b_{-1} \\
 	b_4 & b_3 & b_2 & b_{1} & b_0
 \end{array}
 \r).
 \]
 In other words, the $(i,j)$-th element of $T_{n,g}$ is given by 
 $$t_{i,j}=a_{i-j}\chi_{[2,n]}(i+j)+b_{i-j}\chi_{[n+1,2n]}(i+j), \ \ \text{for} \ \ 1\le i,j\le n,$$ 
 where	$\chi_{[1,n]}(z) =1\;\;\mbox{if $z \in [1,n]$  and zero otherwise}$.
 	
 Note that 
 $P_nT_{n,g}$ is always a (non-symmetric) Hankel matrix.
 For example,  $H_{5}= P_5T_{5,g}$ is given by
 \[
 H_{5}=\l(\begin{array}{ccccc}
 	b_4 & b_3 & b_2 & b_{1} & b_0\\
 	a_3 & b_2 & b_{1} & b_0 & b_{-1} \\
 	a_2 & a_{1} & b_0 & b_{-1} & b_{-2}\\
 	a_1 & a_{0} & a_{-1} & b_{-2} & b_{-3} \\
 	a_0 & a_{-1} &a_{-2} & a_{-3} & b_{-4} \\
 \end{array}
 \r)
 \]
is a Hankel matrix as given in (\ref{eqn:Hankel}).
For the symmetric Hankel matrix ($H_{n,s}$), the relation $H_{n,s}=P_nT_n$ has been used in \cite{liu_wang2011} and \cite{liu2012fluctuations}, to study its LSD and linear spectral statistics. This relation shows that the convergence of any non-symmetric Hankel matrix can be framed in terms of the joint convergence of $P_n$ and $T_n$. 
Since $H_n=P_nT_{n,g}$, the  joint convergence of $P_n$ and independent copies of $n^{-1/2}T_{n,g}$ would imply the joint convergence of independent copies of $n^{-1/2}H_n$.
Observe that for generalized Toeplitz matrices, we have used two labels $a$ and $b$. Analogously, the generalized deterministic Toeplitz matrix, denoted by $D_{n,g}$ is also defined via  (\ref{def:T_gen}) using two sequences of complex numbers $\{d'_i\}_{i\in \Z}$ and $\{d{''}_i\}_{i\in \Z}$. The corresponding deterministic Hankel matrix is defined as $P_nD_{n,g}$.
 
We now state the following assumption on the input sequences. 
\begin{assumption} \label{assump:toe_gII}
	  Let $\{(a_{j,n}= x_{j,n} + \mathrm{i} y_{j,n}, b_{j,n}= x'_{j,n}+ \mathrm{i} y'_{j,n}); j \in \mathbb{Z}\}$  be independent \textit{complex} random variables with mean zero, variance one, and $\{(x_{j,n}, y_{j,n}, x'_{j,n}, y'_{j,n}); j\in \Z \}$ have the following correlation. 
\begin{enumerate}
\item[(i)] $\E[x_{j,n}y_{j,n}]=\rho_1$, $\E[x_{j,n} x'_{j,n}]=\rho_2$, $\E[x_{j,n} y'_{j,n}]=\rho_3$, \\
	$\E[x'_{j,n} y_{j,n}]= \rho_4, \ \E[x'_{j,n} y'_{j,n}]= \rho_5, \ \E[y_{j,n} y'_{j,n}]= \rho_6$.
	\vskip3pt
	\item[(ii)] $ \displaystyle \sup_{j,n} \{\E(|x_{j,n}|^k),  \E(|y_{j,n}|^k), \E(|x'_{j,n}|^k),  \E(|y'_{j,n}|^k)\} \leq c_k < \infty \  \text{for all}\ \ k\geq 1.$
\end{enumerate}
\end{assumption}
Our second main result  yields the joint convergence of $D_{n,g}$, $P_n$ and independent copies of $T_{n,g}$. Its proof is given in Section \ref{sec:Jc_tdp_g}.
\begin{theorem}\label{thm:gen_tdp_com}
	Let $B_{n,1}=n^{-1/2}T_{n,g}$ where $T_{n,g}$ is a random  generalized Toeplitz matrix. For $i=1,2, \ldots, m$, let $\{B_{n,1}^{(i)}\}$ be independent copies of $B_{n,1}$ and $B_{n,2}^{(i)}= D^{(i)}_{n,g}$ be $m$ deterministic generalized Toeplitz matrices. Suppose the input sequences of $T_{n,g}$ satisfy Assumption \ref{assump:toe_gII} and the input sequences of $D_{n,g}$   satisfy Assumption \ref{assump:determin}.
Then $\{P_n, B_{n,j}^{(i)}; 1\leq j \leq 2, 1\leq i\leq m \}$ converge jointly. In particular, $\{P_n, n^{-1/2}T_{n,g}, D_{n,g}\}$ converge jointly.
\end{theorem}

 Note that Assumption \ref{assump:toe_gII} includes the case where $a_{j,n}$ and $a_{-j,n}$ as well as $b_{j,n}$ and $b_{-j,n}$ are not correlated, and moreover all $a_{j,n},b_{j,n}$ have variance one. In the cases where $a_{j,n}$ and $a_{-j,n}$ (and/or $b_{j,n}$ and $b_{-j,n}$) are correlated, and the variance of $a_{j,n},b_{j,n}$ are not equal, we can still derive convergence results using ideas similar to those in the 
   proof of Theorem \ref{thm:gen_tdp_com}.
    However, the  
   combinatorics 
   will be very intricate.
   	
	 The joint convergence of Hankel matrices follows from Theorem \ref{thm:gen_tdp_com} (see Remark \ref{cor:jc_H_g} in Section \ref{subsec:tp_g}).
	Observations as in  Corollary \ref{cor:jc_dtp_hankel} can also be made for generalized Toeplitz and related matrices.

\begin{remark} 
	Suppose $T_{p\times n}$ and $H_{p \times n}$ are Toeplitz and Hankel matrices with \textit{complex} input entries, defined similar to $T_n$ and $H_n$ but are of the order $p\times n$ where $p\to \infty$ and $p/n \to y\in (0, \infty)$. Then it can be shown that, under Assumptions \ref{assump:toeI} and 
 \ref{assump:toe_gII}, the ESDs of $n^{-1}T_{p \times n}T_{p \times n}^{*}$ and $n^{-1}H_{p \times n}H_{p \times n}^{*}$ converge weakly a.s.. There is also convergence in $*$-distribution 
 for any finitely many independent copies of $T_{p \times n}T_{p \times n}^{*}$ and $H_{p \times n}H_{p \times n}^{*}$. 
 The limits 
	depend only on the value of $\rho_i, 1 \leq i \leq 6$ and $y$. This will be clear from the proofs of Propositions \ref{thm:toeplitz}  and  \ref{thm:jc_T+P_g} given in Sections \ref{subsec:jc_T} and \ref{sec:T_n,g}, respectively.
\end{remark}

\section{Proof of Theorem \ref{thm:JC_tdp}: joint  convergence of \texorpdfstring{$T_n$, $D_n$ and $P_n$}{tndnpn}}\label{sec:jc_tdp}
We introduce some notation that will be used throughout. 
We use the convention $0^0=1$. 
Consider the set $[n]=\{ 1,2 ,\ldots , n\}$.
Let $\pi=\{V_{1},\ldots,V_{r}\}$ be a partition of $[n]$. 
The set of all  partitions and pair partitions of $[n]$ are respectively
denoted by  $\mathcal{P}(n)$ and $\mathcal{P}_{2 }(n)$. 


For any partition $\pi$ of $[2k]$, we write the blocks according to the increasing order of their first element. So, a pair partition will always be written as
$$\{(r_1,s_1), (r_2,s_2), \ldots, (r_k,s_k)\}, \ \text{where for each } t, r_t< r_{t+1} \mbox{ and } \ r_t<s_t.$$
Let $\pi=\{V_{1},\ldots,V_{k}\}$ be a pair partition of $[2k]$, where $V_t = (r_t,s_t)$. Then we define a projection map as
\begin{equation}\label{eqn:pi'}
	\mbox{$\pi' (i):=j$ if $i$ belongs to the block $V_{j}$}.
\end{equation}

We also define: 
\begin{align*}
	\d_{i,j} &=1\;\;\mbox{if $i=j$  and zero otherwise},
\end{align*}
and for $\e_i \in \{1,*\}$,
\begin{align}\label{eqn:epsilon'}
	\e'_i :=\l\{\begin{array}{rll}
		1 & \mbox{ if } & \e_i =1,\\
		-1& \mbox{ if } & \e_i = *.
	\end{array} \r.
\end{align}

The combinatorics involved in a direct proof of Theorem \ref{thm:JC_tdp}
would be somewhat intricate.
Hence we break the proof into four steps. 
In Sections \ref{subsec:jc_T}, 
\ref{subsec:dp}, \ref{subsec:td} and  \ref{subsec:Tp},  we establish the joint convergence of independent copies of $T_n$, 
$\{D_n,P_n\}$, $\{D_n,T_n\}$ and $\{T_n,P_n\}$, respectively. Finally, the above steps help us to complete  the 
proof of Theorem \ref{thm:JC_tdp} in Section \ref{subsec:jc_tdp}.

At the heart of all the proofs, we develop appropriate formulae (Lemmas \ref{lem:tracetoeplitz}, \ref{lem:hankel}, \ref{lem:tracegeneralT}, \ref{lem:hankel_g}) for the traces of monomials of our matrices.

\subsection{Joint  convergence of copies of \texorpdfstring{$T_n$}{tn}} \label{subsec:jc_T}
\begin{proposition} \label{thm:toeplitz}
	If $\{T_{n}^{(i)}; 1\leq i\leq m\}$ are $m$ independent copies of random Toeplitz matrices 
	whose input entries satisfy Assumption \ref{assump:toeI}, then $\{n^{-1/2}T_{n}^{(i)}; 1 \leq i \leq m\}$ converge jointly. 
	Moreover the limit $*$-moments are as follows. For $\e_1,\ldots, \e_{p}\in \{1,*\}$ and $\tau_1,\ldots, \tau_{p}\in \{1,\ldots, m\}$, 
	\begin{align} \label{eqn:lim_phi(T*tau)}
	&\lim_{n\to \infty} \vp_n(\dfrac{T_n^{(\tau_1)\e_1}}{n^{1/2}} \cdots \dfrac{T_n^{(\tau_{p})\e_{p}}}{n^{1/2}}) \nonumber \\ 
		&=\l\{\begin{array}{lll}
			\displaystyle \hskip-7pt \sum_{\pi \in {\mathcal P}_2(2k)} \prod_{(r,s)\in \pi} \hskip-5pt \d_{\tau_r,\tau_s} \theta(r,s) \hskip-2pt \int\limits_{0}^1 \hskip-2pt \int\limits_{[-1,1]^k} \prod_{\ell=1}^{2k}\chi_{[0,1]}\big(z_0+\sum_{t=\ell}^{2k}\e_{\pi}(t)z_{\pi'(t)}\big)\prod_{i=0}^kdz_i & \hskip-5pt \mbox{if } p=2k, \\
			0 & \hskip-9pt \mbox{if } p=2k+1,
		\end{array}\r.
	\end{align}
	where $\pi'$,  $\e'$  and $\e_{\pi}$ are as defined in \eqref{eqn:pi'}, \eqref{eqn:epsilon'} and \eqref{eqn:epsilon_pi}, respectively, and 
	\begin{align*} 
		\theta(r,s) = \big[\sigma_x^2 + \sigma_y^2\big]^{1-\d_{\e'_{r}, \e'_{s} }} \big[(\rho_2-\rho_5) + \mathrm{i} \e'_r(\rho_3+\rho_4)\big]^{\d_{\e'_r,\e'_s}}.  
	\end{align*}
	\noindent In particular, for $m=1$,
 \begin{align}  \label{eqn:lim_mome_A1_TJc}
		&\lim_{n\to \infty}\vp_n(\dfrac{T_n^{\e_1}}{n^{1/2}}\cdots \dfrac{T_n^{\e_{p}}}{n^{1/2}}) \nonumber \\
		&=\l\{\begin{array}{lll}
			\displaystyle \sum_{\pi \in {\mathcal P}_2(2k)} \prod_{(r,s)\in \pi} \hspace{-0.1in}\theta(r,s) \hspace{-0.05in}\int_{0}^{1}\hspace{-0.05in}\int_{[-1,1]^{k}}\prod_{\ell=1}^{2k}\chi_{[0,1]}\big(z_0+\sum_{t=\ell}^{2k}\e_{\pi}(t) z_{\pi'(t)}\big)\prod_{i=0}^k dz_i & \mbox{if } p=2k, \\
			0 &  \hskip-9pt \mbox{if } p=2k+1,
		\end{array}\r.
	\end{align}
 where $\theta(r,s)$ is as in (\ref{eqn:lim_phi(T*tau)}).

 \noindent Moreover if $\{a_j\}$ is real-valued,
	then 
	$\theta(r,s)= (\sigma_x^2)^{1-\d_{\e'_{r}, \e'_{s} }} \rho_{2}^{\d_{\e'_r,\e'_s}}$.
\end{proposition}
The following corollaries follow from the above proposition.

\begin{corollary} \label{cor:A2+3_TJc}
	Suppose $\{T_{n,h}^{(i)}; 1\leq i\leq m\}$ are $m$ independent copies of random Hermitian Toeplitz matrices 
	whose input entries satisfy Assumption \ref{assump:toeI}.
	Then $\{n^{-1/2}T^{(i)}_{n,h}; 1 \leq i \leq m\}$ 
	converge in $*$-distribution, with the limit $*$-moments as in (\ref{eqn:lim_phi(T*tau)}) with $\theta(r,s)=(\sigma_x^2 + \sigma_y^2)$. 
\end{corollary}
Specializing further, if the input entries are real-valued, then for $m$ independent copies of random symmetric Toeplitz matrices, $\{n^{-1/2}T^{(i)}_{n,s}; 1 \leq i \leq m\}$, the limit $*$-moments are as in (\ref{eqn:lim_phi(T*tau)}) with $\theta(r,s)=\sigma_x^2$. This is  
Proposition 1 of \cite{bose_saha_patter_JC_annals}.

\begin{corollary} \label{cor:iid_jc_T}
Let $T_{n}^{(1)},\ldots, T_{n}^{(m)}$ be independent Toeplitz matrices whose input  entries are i.i.d complex random variables with mean zero, variance one and all moments finite. Then $\{n^{-1/2}T_{n}^{(1)},\ldots, n^{-1/2}T_{n}^{(m)}\}$ converge in $*$-distribution, with the limit $*$-moments as in (\ref{eqn:lim_phi(T*tau)}) with $\theta(r,s)=\big[(\rho_2-\rho_5) + \mathrm{i} \e'_r(\rho_3+\rho_4)\big]^{\d_{\e'_r,\e'_s}}$.
\end{corollary}

\begin{corollary}\label{cor:poly_Lsd_T} 
	Suppose $T_{n}^{(1)},\ldots, T_{n}^{(m)}$ are independent Toeplitz matrices whose input entries satisfy Assumption \ref{assump:toeI},  
	and $Q_n=Q(n^{-1/2} T_{n}^{(1)},\ldots, n^{-1/2}T_{n}^{(m)})$ is a self-adjoint polynomial in these matrices  and their adjoints. Then, the Expected ESD of $Q_n$ converges weakly, the ESD converges to the same limit a.s., and the moments of the LSD are bounded by the moments of a Gaussian distribution. 
	\vskip3pt
	
	\noindent In particular, 
	if $\{a_{j}\}$ satisfy Assumption \ref{assump:toeI}, then the ESD of $n^{-1/2}T_{n,h}$ and $n^{-1/2}T_{n,s}$ converge a.s.~to  symmetric probability distributions 
	whose moments are as in (\ref{eqn:lim_mome_A1_TJc}) with $\e_1=\cdots = \e_{2k}=1$, and $\theta(r,s)=(\sigma_x^2 + \sigma_y^2)$ and $\theta(r,s)=\sigma_x^2$, respectively.
\end{corollary}

Now we proceed for the proof of Propositions \ref{thm:toeplitz}. First we start with  the following notation which will be used throughout the paper: 
\begin{align} \label{eqn:i_k in -n to n}
I_k&=\{(i_1,\ldots, i_k)\;:\; i_1,\ldots , i_k\in \{-(n-1),\ldots, -1, 0, 1,\ldots, (n-1)\}\}.
\end{align}
Liu and Wang \cite{liu_wang2011} studied the convergence of the ESD of real band Toeplitz matrices using backward and forward shift matrices. We follow this approach since it helps us to compute the trace of matrices in a systematic manner. 

\begin{lemma}\label{lem:tracetoeplitz}
(a) Let $M_n=((a_{i-j}))_{n\times n}$ be an $n\times n$ Toeplitz matrix (random or non-random) with \textit{complex} input entries. Then for $\e_1,\ldots, \e_k\in \{1,*\}$,
\begin{align*}
\Tr(M_n^{\e_1}\cdots M_n^{\e_k})=\sum_{j=1}^n\sum_{I_k} a^{\epsilon_1}_{i_1}\cdots a^{\epsilon_k}_{i_k}\prod_{\ell=1}^k
\chi_{[1,n]}(j+\sum_{t=\ell}^k\e_t'i_t)\d_{0,\sum_{t=1}^k\e_t'i_t},
\end{align*}
where $I_k$ is as in (\ref{eqn:i_k in -n to n}), $\e_i'=1$ if $\e_i=1$ and $\e_i'=-1$ if $\e_i=*$.
\vskip3pt
\noindent (b) 
 If  $M_n^{(\tau)}=((a_{i-j}^{(\tau)}))_{n\times n}$  for $\tau=1,\ldots, m$, then for $\tau_1,\ldots, \tau_k\in \{1,\ldots, m\}$,
 \begin{align*}
 \Tr(M_n^{(\tau_1)\e_1}\cdots M_n^{(\tau_{k})\e_{k}})
 =\sum_{j=1}^n\sum_{I_{k}} a_{i_1}^{(\tau_1) \e_1}\cdots a_{i_k}^{(\tau_k) \e_k}\prod_{\ell=1}^{k}
 \chi_{[1,n]}(j+\sum_{t=\ell}^{k}\e_t'i_t)\d_{0,\sum_{t=1}^k \e_t'i_t}.
 \end{align*} 
\end{lemma}

\begin{proof} 
For $i=1,\ldots, n$, let $e_i=(0,\ldots, 1,\ldots, 0)^t$, where $1$ is at the $i$-th place. Then 
\begin{align}\label{eqn:1}
\Tr(M_n^{\e_1}\cdots M_n^{\e_k}) =\sum_{j=1}^ne_j^t(M_n^{\e_1}\cdots M_n^{\e_k})e_j.
\end{align}
Suppose
$$B_n=((\d_{i+1,j }))_{n\times n}\ \ \text{and}\ \ F_n=((\d_{i,j+1}))_{n\times n},$$ are respectively the backward and the forward shift matrices. We abbreviate them as $B$ and $F$. Then  
\begin{align*}
M_n=\sum_{i=0}^{n-1}a_{-i}B^i+\sum_{i=1}^{n-1}a_iF^i \ \mbox{ and } \ M_n^*=\sum_{i=0}^{n-1}a^*_{-i}F^i+\sum_{i=1}^{n-1}a^*_iB^i.
\end{align*}
Therefore, for $j=1,\ldots, n,$ we have 
\begin{align}\label{eqn:2}
M_ne_j
=\sum_{i=0}^{n-1}a_{-i}B^ie_j+\sum_{i=1}^{n-1}a_iF^ie_j\notag
&=\sum_{i=0}^{n-1}a_{-i}\chi_{[1,n]}(j-i)e_{j-i}+\sum_{i=1}^{n-1}a_i\chi_{[1,n]}(j+i)e_{j+i}\notag
\\&=\sum_{i=-(n-1)}^{n-1}a_i\chi_{[1,n]}(j+i)e_{j+i}.
\end{align}
Similarly, for $j=1,\ldots, n,$ we have 
\begin{align}\label{eqn:3}
M_n^*e_j=\sum_{i=-(n-1)}^{n-1}a^*_i\chi_{[1,n]}(j-i)e_{j-i}.
\end{align}
Thus, for $\e\in \{1, *\}$, combining  \eqref{eqn:2} and \eqref{eqn:3}, we have
\begin{align*} 
M_n^{\epsilon}e_j=\sum_{i=-(n-1)}^{n-1}a^\epsilon_i\chi_{[1,n]}(j+\e' i)e_{j+\e'i},
\end{align*}
where $\e'=1$ if $\e=1$ and $\e'=-1$ if $\e=*$. 

 Using the above ideas, for $j=1,\ldots, n$, we have
\begin{align*}
M_n^{\epsilon_{k-1}}M_n^{\e_k}e_j&= M_n^{\e_{k-1}}\sum_{i_k=-(n-1)}^{n-1} a^{\epsilon_k}_{i_k}\chi_{[1,n]}(j+\e_{i_k}' i_k)e_{j+\e_{i_k}'i_k}
\\&=\sum_{i_k=-(n-1)}^{n-1} a^{\epsilon_k}_{i_k}\chi_{[1,n]}(j+\e_{i_k}' i_k) M_n^{\e_{k-1}}e_{j+\e_{i_k}'i_k}
\\&=\hspace{-5pt}\sum_{i_k, i_{k-1}}\hspace{-5pt} a^{\epsilon_{k-1}}_{i_{k-1}} a^{\epsilon_{k}}_{i_k} \chi_{[1,n]}(j+\e_{i_k}' i_k)\chi_{[1,n]}(j+\e_{i_{k-1}}'i_{k-1}+\e_{i_k}' i_k)e_{j+\e_{i_{k-1}}'i_{k-1}+\e_{i_k}' i_k}.
\end{align*}
Continuing this process, and using \eqref{eqn:1},  we get Part (a). 
The proof of Part (b) is similar and we skip the details. 
\end{proof}
Now using the above trace formula, we prove Proposition \ref{thm:toeplitz}.
\begin{proof}[Proof of Proposition \ref{thm:toeplitz}]
For transparency, we first consider the case of a single matrix, that is, we establish (\ref{eqn:lim_mome_A1_TJc}). Let $p$ be a positive integer. By Lemma \ref{lem:tracetoeplitz} we have
\begin{align}\label{eqn:thm1}
\vp_n(\dfrac{T_n^{\e_1}}{n^{1/2}}\cdots \dfrac{T_n^{\e_p}}{n^{1/2}})&=\frac{1}{n^{\frac{p}{2}+1}}\E\Tr[T_n^{\e_1}\cdots T_n^{\e_p}]\notag\\
&=\frac{1}{n^{\frac{p}{2}+1}}\sum_{j=1}^n\sum_{I_p}\E\big[a^{\e_1}_{i_1}\cdots a^{\e_p}_{i_p}\prod_{\ell=1}^p
\chi_{[1,n]}(j+\sum_{t=\ell}^{p}\e_t'i_t)\d_{0,\sum_{t=1}^p\e_t'i_t}\big]\notag
\\&=\frac{1}{n^{\frac{p}{2}}}\sum_{I_p'}\E[a^{\e_1}_{i_1}\cdots a^{\e_p}_{i_p}] \frac{1}{n}\sum_{j=1}^n\prod_{\ell=1}^p
\chi_{[1,n]}(j+\sum_{t=\ell}^p\e_t'i_t),
\end{align}
where  $I_p$ is as in (\ref{eqn:i_k in -n to n}) and
\begin{equation}\label{iprime}
I_p'=\{(i_1,\ldots, i_p)\in I_p\; :\; \sum_{t=1}^p\e_t'i_t=0\}.
\end{equation}
 Observe that, for all $n$, 
\begin{align*}
0 \leq \frac{1}{n}\sum_{j=1}^n\prod_{\ell=1}^p
\chi_{[1,n]}(j+\sum_{t=\ell}^{p}\e_t'i_t)\le 1.
\end{align*} 

Let $\pi$ be any partition of $\{i_1,\ldots, i_p\}$. We write it as $\pi=\{ \pi_1, \pi_2, \ldots, \pi_k\}$ such that the random variables in $\{a_i : i\in \pi_j\}$ are correlated for $1\le j\le k$, but $\{a_i : i\in \pi_u\}$ and $\{a_i : i\in \pi_v\}$ are uncorrelated if $1\le u\neq v\le k$. Let $I_p'(\pi)$ be the set of indices from $I_p'$ that obey this rule. Then, by Assumption \ref{assump:toeI}, we have
\begin{align*}
	\sum_{I_p'} | \E[a^{\e_1}_{i_1}\cdots a^{\e_p}_{i_p}] | =\sum_{\pi\in \mathcal P(p)}\sum_{I_p'(\pi)} \prod_{u=1}^{|\pi|} |\E_{\pi_u}[a^{\e_1}_{i_1}\cdots a^{\e_p}_{i_p}] |
	\leq \sum_{\pi\in \mathcal P(p)} n^{|\pi|} \prod_{u=1}^{|\pi|} 2^{|\pi_u|} \times c_{|\pi_u|},
\end{align*}
where $c_{|\pi_u|}$ are moments as in Assumption \ref{assump:toeI}, $|\pi|$ denotes the number of blocks in $\pi$ and $|\pi_u|$ denotes the cardinality of $\pi_u$. Here $\E_{\pi_u}[a^{\e_1}_{i_1}\cdots a^{\e_p}_{i_p}]=\E[a^{\e_{d_1}}_{i_{d_1}}\cdots a^{\e_{d_u}}_{i_{d_u}}]$ if $\pi_u$ looks like $\{i_{d_1}, \ldots, i_{d_u}\}$. Note that the constant $\prod_{u=1}^{|\pi|} c_{|\pi_u|}$ depends only on $\pi$.
Therefore
$$
\lim_{n \to \infty}\frac{1}{n^{\frac{p}{2}}}\sum_{I_p'}\E[a^{\e_1}_{i_1}\cdots a^{\e_p}_{i_p}]=0, \mbox{ if $|\pi|<\frac{p}{2}$}.$$
Thus, to obtain a non-zero limit of the expression in (\ref{eqn:thm1}), it is enough to consider only  those $\pi$ for which $|\pi|\ge p/2$.
On the other hand, as $\E[a_i]=0$,  $\E[a^{\e_1}_{i_1}\cdots a^{\e_p}_{i_p}]$ non-zero implies $|\pi_i|\ge 2$ for $1\le i\le k$. In that case 
$|\pi|\leq p/2$. 
Hence,  we are left to consider only those $\pi$ where $|\pi|=p/2$. In other words, $\pi$ is a \textit{pair-partition}. 

\noindent If $p$ is odd, then there 
are no pair-partitions and hence
\begin{equation} \label{eqn:phi_To(1)odd}
	\lim_{n\to \infty}n^{-\frac{p}{2}}\sum_{I_p'}\E[a^{\e_1}_{i_1}\cdots a^{\e_p}_{i_p}]=0.
\end{equation}

Now suppose $p$ is even, say $p=2k$. Then 
\begin{align*}
\vp_n(\dfrac{T_n^{\e_1}}{n^{1/2}}\cdots \dfrac{T_n^{\e_{2k}}}{n^{1/2}})&=\frac{1}{n^{k+1}}\sum_{j=1}^n\sum_{\pi\in \mathcal P_{2}(2k)}\sum_{I_{2k}'(\pi)}\prod_{(r,s)\in \pi}\hspace{-6pt} \E[a^{\e_{r}}_{i_r} a^{\e_{s}}_{i_{s}}]\prod_{\ell=1}^{2k}\hspace{-2pt}
\chi_{[1,n]}(j+\sum_{t=\ell}^{2k}\e_t'i_t)+o(1).
\end{align*}
Clearly, if the number of free indices among the indices $\{i_1,\ldots, i_{2k}\}$ is strictly less than $k$, then the limit contribution is zero. 
On the other hand, by the independence property of entries (Assumption \ref{assump:toeI}), we have that $\E[a^{\e_{r}}_{i_r} a^{\e_{s}}_{i_{s}}]$ is non-zero only if, $i_{r}=i_{s}$ or $i_{r}=-i_{s}$.
This implies that the number of free indices is at most $k$. Therefore we focus on the summands where the number of free indices is exactly equal to $k$. 


\noindent In other words, $\{i_{r_1}, i_{s_1}\},\ldots, \{i_{r_k}, i_{s_k}\}$ are disjoint, where
  $\pi=(r_1,s_1)\cdots (r_k,s_k)\in \mathcal P_2(2k)$. 
   Again, as
$(i_1,\ldots, i_{2k})\in I_{2k}'(\pi)$, we have   
\begin{align*}
0=\sum_{t=1}^{2k}\e_t'i_t=\sum_{t=1}^k(\e_{r_t}'i_{r_t}+\e_{s_t}'i_{s_t}).
\end{align*}
This implies that the number of free indices will be exactly $k$ only when
\begin{align}\label{eqn:reduction}
	(\e_{r_t}'i_{r_t}+\e_{s_t}'i_{s_t})=0, \mbox{ for $t=1,\ldots, k$}.
\end{align}
Otherwise, $\{i_{r_1}, i_{s_1}\},\ldots, \{i_{r_k}, i_{s_k}\}$ will satisfy a one dimensional equation which will reduce one degree of freedom. 
 Since  $\e_1',\ldots, \e_{2k}'\in \{1,-1\}$, the above holds if and only if
\begin{align*}
i_{r_t}=
\l\{\begin{array}{lll}
i_{s_t} & \mbox{ if } & \e_{r_t}'\e_{s_t}'=-1,\\ 
-i_{s_t} & \mbox{ if} & \e_{r_t}'\e_{s_t}'=1,
\end{array}
\r.
\end{align*}
for $t=1,\ldots, k$. In other words, we have a non-zero contribution in the limit iff
\begin{align}\label{eqn:E[ars_epsilon]}
	\E[a^{\e_{r_{t}}}_{i_{r_{t}}} a^{\e_{s_{t}}}_{i_{s_{t}}}] &= 
	\l\{\begin{array}{lll}
		\E|a_{i_{r_{t}}}|^2 & \mbox{ if}  & \e_{r_t}'\e_{s_t}'=-1,\\ 
		(\rho_2-\rho_5) + \mathrm{i} \e'_{r_t}(\rho_3+\rho_4) & \mbox{ if} & \e_{r_t}'\e_{s_t}'=1,
	\end{array}\r. \nonumber \\
&= \big[\sigma_x^2 + \sigma_y^2\big]^{1-\d_{\e'_{r_{t}}, \e'_{s_{t}} }} \big[(\rho_2-\rho_5) + \mathrm{i} \e'_{r_t}(\rho_3+\rho_4)\big]^{\d_{\e'_{r_{t}}, \e'_{s_{t}} }}   \nonumber \\
& = \theta(r_t,s_t), \mbox{ say}.
\end{align}

Recall the map $\pi'$ from (\ref{eqn:pi'}). As the $t$-th block is $(r_t, s_t)$, we have $\pi'(r_t)=\pi'(s_t)=t$.   
Thus $\pi'$ gives $k$ variables $z_1,\ldots,z_k$ from $k$ blocks. In other words, each block introduces one new variable. Define $\xi_\pi(r_t)=(-1)^{\d_{\e_{r_t}',\e_{s_t}'}}$ and $\xi_{\pi}(s_t)=1$ for $t=1,\ldots, k$. Then we can write 
\begin{align*} 
\prod_{\ell=1}^{2k}
\chi_{[1,n]}(j+\sum_{t=\ell}^{2k}\e_t'i_t)=\prod_{\ell=1}^{2k}
\chi_{[1,n]}(j+\sum_{t=\ell}^{2k}\e_t'\xi_{\pi}(t)i_{\pi'(t)}).
\end{align*}

Now from \eqref{eqn:E[ars_epsilon]} and the above equation, we have
\begin{align*}
\lim_{n\to \infty}\vp_n(\dfrac{T_n^{\e_1}}{n^{1/2}}\cdots \dfrac{T_n^{\e_{2k}}}{n^{1/2}})
&=\sum_{\pi \in {\mathcal P}_2(2k)}  \prod_{t=1}^k \theta(r_t,s_t) \lim_{n\to \infty}\frac{1}{n^{k+1}}\sum_{j=1}^n\sum_{I'_k}\prod_{\ell=1}^{2k}
\chi_{[1,n]}(j+\sum_{t=\ell}^{2k}\e_t'\xi_{\pi}(t)i_{\pi'(t)})
\\&=\sum_{\pi \in {\mathcal P}_2(2k)} \prod_{t=1}^k \theta(r_t,s_t) \lim_{n\to \infty}\frac{1}{n^{k+1}}\sum_{j=1}^n\sum_{I'_k}\prod_{\ell=1}^{2k}
\chi_{[\frac{1}{n},1]}\big(\frac{j}{n}+\sum_{t=\ell}^{2k}\e_t'\xi_{\pi}(t)\frac{i_{\pi'(t)}}{n}\big)
\\&=\sum_{\pi \in {\mathcal P}_2(2k)} \prod_{t=1}^k \theta(r_t,s_t) \int_{0}^1 \int_{[-1,1]^k}\prod_{\ell=1}^{2k}\chi_{[0,1]}\big(z_0+\sum_{t=\ell}^{2k}\e_t'\xi_{\pi}(t)z_{\pi'(t)}\big)\prod_{i=0}^kdz_i,
\end{align*} 
where the last equality follows by appealing to convergence of Riemann sums of nice functions to their Riemann integrals.

 Note that $\e'_{s_t}$ and $\e'_{r_t}\xi_{\pi}(r_t)$ have opposite signs.  
Consider the change of variables $\e_{s_t}'z_t$ to $z_t, 1\leq t\leq k$ and appeal to symmetry, 
 we get
 	\begin{align*}
	\lim_{n\to \infty}\vp_n(\dfrac{T_n^{\e_1}}{n^{1/2}}\cdots \dfrac{T_n^{\e_{2k}}}{n^{1/2}})
	=\sum_{\pi \in {\mathcal P}_2(2k)}  \prod_{(r,s)\in \pi} \theta(r,s) \int_{0}^1\int_{[-1,1]^k}\prod_{\ell=1}^{2k}\chi_{[0,1]}\big(z_0+\sum_{t=\ell}^{2k}\e_{\pi}(t) z_{\pi'(t)} \big)\prod_{i=0}^kdz_i,
		\end{align*}
where for any pair-partition $\pi=(r_1,s_1)\cdots (r_k,s_k) \ \text{with} \ \ r_d<s_d,$
\begin{align} \label{eqn:epsilon_pi}	
	\e_\pi(t):=\l\{\begin{array}{rll}
		1 & \mbox{ if } & t=s_d,\\
		-1& \mbox{ if } & t=r_d.
	\end{array} \r.
\end{align}
This completes the proof for a single matrix. 

 Now suppose we have $m$ independent matrices with complex entries. Then from Lemma \ref{lem:tracetoeplitz}, for  $\tau_i \in \{1,\ldots, m\}$, we have
\begin{align*}
\vp_n(\dfrac{T_n^{(\tau_1)\e_1}}{n^{1/2}}\cdots \dfrac{T_n^{(\tau_{p})\e_{p}}}{n^{1/2}})=\frac{1}{n^{\frac{p}{2}+1}}\sum_{j=1}^n\sum_{I_{p}'}\E\big[a_{i_1}^{(\tau_1)\e_1}\cdots a_{i_p}^{(\tau_p)\e_p}\prod_{\ell=1}^{p}
\chi_{[1,n]}(j+\sum_{t=\ell}^{p}\e_t'i_t)\big],
\end{align*}
where $I_p'$ is as in  (\ref{iprime}). Based on the arguments employed to derive (\ref{eqn:phi_To(1)odd}), for odd $p$
\begin{align*}
\lim_{n \to \infty}\vp_n(\dfrac{T_n^{(\tau_1)\e_1}}{n^{1/2}}\cdots \dfrac{T_n^{(\tau_p)\e_p}}{n^{1/2}})=0.
\end{align*}
Let $p=2k$. Then, using the previous arguments, we have
\begin{align*}
&\lim_{n\to\infty}\vp_n(\dfrac{T_n^{(\tau_1)\e_1}}{n^{1/2}}\cdots \dfrac{T_n^{(\tau_{2k})\e_{2k}}}{n^{1/2}})
\\&=\lim_{n\to\infty}\frac{1}{n^{k+1}}\sum_{j=1}^n\sum_{I_{2k}'}\sum_{\pi\in \mathcal P_2(2k)}\prod_{(r,s)\in \pi}\E[a_{i_r}^{(\tau_r)\e_r} a_{i_s}^{(\tau_s)\e_s}]\prod_{\ell=1}^{2k}
\chi_{[1,n]}(j+\sum_{t=\ell}^{2k}\e_t'i_t)
\\&=\lim_{n\to\infty}\frac{1}{n^{k+1}}\sum_{j=1}^n\sum_{\pi\in \mathcal P_2(2k)}\sum_{I_{2k}'(\pi)}\prod_{(r,s)\in \pi}\d_{\tau_r,\tau_s} \theta(r,s) \prod_{\ell=1}^{2k}
\chi_{[1,n]}(j+\sum_{t=\ell}^{2k}\e_t'\xi_{\pi}(t)i_{\pi'(t)})
\\&=\sum_{\pi \in {\mathcal P}_2(2k)}\prod_{(r,s)\in \pi}\d_{\tau_r,\tau_s} \theta(r,s) \int_{0}^1\int_{[-1,1]^k}\prod_{\ell=1}^{2k}
\chi_{[0,1]}(z_0+\sum_{t=\ell}^{2k}\e_{\pi}(t)z_{\pi'(t)}) \prod_{i=0}^kdz_i.
\end{align*}
This completes the proof for $m$ matrices. 

Now if the input entries of matrices are \textit{real}, then 
$\sigma_y^2=0$ and $\rho_3=\rho_4=\rho_5=0$. Thus from (\ref{eqn:E[ars_epsilon]}),
we have the limit $*$-moments are as in (\ref{eqn:lim_phi(T*tau)}) with $\theta(r,s)=(\sigma_x^2)^{1-\d_{\e'_{r}, \e'_{s} }} \rho_{2}^{\d_{\e'_r,\e'_s}}$.
This completes the proof of Proposition \ref{thm:toeplitz}.
\end{proof}

\begin{proof}[Proof of Corollary \ref{cor:A2+3_TJc}] 
	If the Toeplitz matrix is Hermitian, then $a_j^*=\bar{a}_j$, that is, $a_{-j}=\bar{a}_j$. Thus, $x_j=x_{-j}$ and $y_j=-y_{-j}$. So, in this case, we have $\rho_2=\sigma_x^2$, $\rho_3=-\rho_4$ and $\rho_5=-\sigma_y^2$. Hence from 
 (\ref{eqn:E[ars_epsilon]}), we have the limit $*$-moments are as in (\ref{eqn:lim_phi(T*tau)}) with $\theta(r,s)=(\sigma_x^2 + \sigma_y^2)$. This gives the first part of the corollary.	

Now if the matrices are real symmetric Toeplitz, then $\sigma_y^2$ will be zero, and thus the limit $*$-moments are as in (\ref{eqn:lim_phi(T*tau)}) with $\theta(r,s)=\sigma_x^2$. This gives the second part of the corollary, and hence the proof is complete.
\end{proof}

\begin{proof}[Proof of Corollary \ref{cor:poly_Lsd_T}] By Proposition \ref{thm:toeplitz}, moments of the expected ESD, that is, $n^{-1}\E [\Tr(Q_n^k)]$ converge for every $k\geq 1$. From the formula for the limiting moments, it is easy to see that the even limit moments are bounded by the even moments of an appropriate Gaussian distribution, and hence they satisfy Carleman's condition. So, the limit moments define a unique distribution, and thus weak convergence of the expected ESD holds.  
	
 One can show that for every $k\geq 1$, 
	\begin{equation}\label{eq:m4}
		\E\big[n^{-1}\Tr(Q_{n}^{k})-\E[n^{-1}\Tr(Q_n^{k})]\big]^4=O(n^{-2}).
	\end{equation}
This can be shown by following the arguments in Lemma 1.4.3 of Bose \cite{bose_patterned}, who proved (\ref{eq:m4}) 
	when we have real i.i.d.~input. 
 We omit the details. 
	Now using the Borel-Cantelli lemma, it follows that the ESD of $Q_n$ converges a.s.~to the same limit. 
\end{proof}

\subsection{Joint convergence of \texorpdfstring{$D_n$ and $P_{n}$}{dnpn}} \label{subsec:dp}
\begin{proposition} \label{thm:jc_D+P}
  Suppose $\{D_{n}^{(\tau)}; 1\leq \tau \leq m\}$ are $m$ deterministic Toeplitz matrices whose input entries satisfy Assumption \ref{assump:determin} and $P_n$ is the backward identity permutation matrix. Then $\{P_n, D_{n}^{(\tau)}; 1 \leq \tau \leq m\}$ converge jointly. 
	The limit $*$-moments are as given in (\ref{eqn:lim_phi_P,D*}).
\end{proposition}

We first derive a trace formula for an arbitrary monomial in $P_n$ and Toeplitz matrices. 
\begin{lemma}  \label{lem:hankel}
	Let $M^{(\tau)}_{n}$ be $m$ Toeplitz matrices (random or non-random) with input sequences $(a^{(\tau)}_i)_{i\in \Z}$ for $\tau=1,2, \ldots, m$ and $P_n$ be the backward identity permutation matrix. Then for $\e_i\in \{1,*\}$, $\tau_i \in \{1, \ldots,m\}$ and for integers $0 \leq k_1 \leq k_2 \leq \cdots \leq k_p$, we have
	\begin{align} \label{eqn:trace_Hs}
	&\Tr\big[(P_nM_{n}^{(\tau_1)\e_1} \cdots M_{n}^{(\tau_{k_1}) \e_{k_1}}) 
 \cdots (P_nM_{n}^{(\tau_{k_{p-1}+1})\e_{k_{p-1}+1}} \cdots M_{n}^{(\tau_{k_p})\e_{k_p}})\big] \nonumber \\
		& =
		\left\{ \begin{array}{ll}
			\displaystyle{ \sum_{j=1}^n \sum_{I_{k_p}} \prod_{t=1}^{k_p} a^{(\tau_t)\e_t}_{i_t} \prod_{e=1}^p m_{\underline k, e}(j,i)  \d_{0,\sum_{c=1}^p(-1)^{p-c} \sum_{t=k_{c-1}+1}^{k_{c}}  \e_t'i_t}} & \mbox{if $p$ is even},
			\\
			\displaystyle{\sum_{j=1}^n \sum_{I_{k_p}} \prod_{t=1}^{k_p} a^{(\tau_t)\e_t}_{i_t} \prod_{e=1}^p  m_{\underline k, e}(j,i)   \d_{2j-1-n,\sum_{c=1}^p(-1)^{p-c} \sum_{t=k_{c-1}+1}^{k_{c}}  \e'_t i_t}} & \mbox{if $p$ is odd},
		\end{array}
		\right.
	\end{align}
	where $I_{k_p}$ is as in (\ref{eqn:i_k in -n to n})  
	and for $e=1,2, \ldots,p$, 
	\begin{align} \label{eqn:m_chi,t_Hn}
		m_{\underline k, e}(j,i) = \prod_{\ell=k_{e-1}+1}^{k_e} \chi_{[1,n]}\big(j+ (-1)^{p-e} \sum_{t=\ell}^{k_e} \e'_t i_t + \sum_{c=e+1}^{p} (-1)^{p-c}  \sum_{t=k_{c-1}+1}^{k_c} \e'_t i_t \big), 
	\end{align}
with $k_0=0$ and $\sum_{c=p+1}^{p}(-1)^{p-c}  \sum_{t=k_{c-1}+1}^{k_c} \e'_t i_t =0$. Note that $m_{\underline k, e}(j,i)$ depends on $k_1, k_2, \ldots, k_p$.
\end{lemma}

\begin{proof}
	For simplicity of notation, we prove this lemma only for 
 a single matrix. The same argument will work if we have multiple collections. 
	Note from Lemma \ref{lem:tracetoeplitz} (a) that 
	\begin{align*}
		(M_n^{\e_{1}}\cdots M_n^{\e_{k_1}})e_j
		= \sum_{I_{k_1}} a^{\epsilon_1}_{i_1}\cdots a^{\epsilon_{k_1}}_{i_{k_1}}\prod_{\ell=1}^{k_1}
		\chi_{[1,n]}(j+\sum_{t=\ell}^{k_1}\e_t'i_t) e_{j+\sum_{t=1}^{k_1} \e_t'i_t},
	\end{align*}
	where $I_{k_1}$ is as in (\ref{eqn:i_k in -n to n}) and $\e'$ is as in (\ref{eqn:epsilon'}). Since $P_ne_j=e_{n+1-j}$, we have  
	\begin{align*}
		(P_n M_n^{\e_{1}}\cdots M_n^{\e_{k_1}})e_j &= \sum_{I_{k_1}} a^{\epsilon_1}_{i_1}\cdots a^{\epsilon_{k_1}}_{i_{k_1}}\prod_{\ell=1}^{k_1}
		\chi_{[1,n]}(j+\sum_{t=\ell}^{k_1}\e_t'i_t) e_{n+1 -(j+\sum_{t=1}^{k_1} \e_t'i_t)}.
	\end{align*}
	Now 
	\begin{align*}
		&(P_n M_n^{\e_{k_{p-2}+1}}\cdots M_n^{\e_{k_{p-1}}}) (P_n M_n^{\e_{k_{p-1}+1}}\cdots M_n^{\e_{k_p}}) e_j\\
		&  = \hskip-5pt \sum_{I_{k_p-k_{p-1}}}  \prod_{\ell=k_{p-1}+1}^{k_p} \hskip-5pt a^{\epsilon_\ell}_{i_\ell} \prod_{\ell=k_{p-1}+1}^{k_p}
		\hskip-5pt \chi_{[1,n]}(j+\sum_{t=\ell}^{k_p}\e_t'i_t) (P_n M_n^{\e_{k_{p-2}+1}}\cdots M_n^{\e_{k_{p-1}}}) e_{n+1 -(j+\sum_{t=k_{p-1}+1}^{k_p} \e_t'i_t)} \\
		&=\sum_{I_{k_p-k_{p-2}}}  \prod_{\ell=k_{p-2}+1}^{k_p} a^{\epsilon_\ell}_{i_\ell}  \prod_{\ell=k_{p-1}+1}^{k_p}
		\chi_{[1,n]}(j+\sum_{t=\ell}^{k_p}\e_t'i_t) \\
		& \qquad \times \prod_{\ell=k_{p-2}+1}^{k_{p-1}}
		\chi_{[1,n]}\big( j -  \sum_{t=\ell}^{k_{p-1}}\e_t'i_t  +\sum_{t=k_{p-1}+1}^{k_p}\e_t'i_t\big)  e_{j+\sum_{t=k_{p-1}+1}^{k_p} \e_t'i_t- \sum_{t=k_{p-2}+1}^{k_{p-1}} \e_t'i_t}.
	\end{align*}
	Continuing the process, for $\e_1,\ldots, \e_{k_p}\in \{1,*\}$, we get 
	\begin{align*}
		&	\big[(P_nM_{n}^{\e_1} \cdots M_{n}^{\e_{k_1}}) (P_nM_{n}^{\e_{k_1+1}} \cdots M_{n}^{\e_{k_2}})  \cdots (P_nM_{n}^{\e_{k_{p-1}+1}} \cdots M_{n}^{\e_{k_p}})\big] e_j\\
		& =	\l\{\begin{array}{lll}
			\displaystyle  \sum_{I_{k_p}} \prod_{t=1}^{k_p} a^{\e_t}_{i_t} \prod_{e=1}^p m_{\underline k, e}(j,i) e_{j+\sum_{c=1}^p(-1)^{p-c} \sum_{t=k_{c-1}+1}^{k_{c}}  \e_t'i_t}  & \mbox{ if $p$ is even}, \\
			\displaystyle	\sum_{I_{k_p}} \prod_{t=1}^{k_p} a^{\e_t}_{i_t} \prod_{e=1}^p m_{\underline k, e}(j,i) e_{n+1-j-\sum_{c=1}^p(-1)^{p-c} \sum_{t=k_{c-1}+1}^{k_{c}}  \e_t'i_t} &  \mbox{ if $p$ is odd},
		\end{array}\r.
	\end{align*}
	where $I_{k_p}$ and $m_{\underline k, e}(j,i)$ are as in (\ref{eqn:i_k in -n to n})  and (\ref{eqn:m_chi,t_Hn}), respectively.
	Now using the fact that $\Tr(A_n)=\sum_{j=1}^ne_j^t(A_n)e_j$, we get (\ref{eqn:trace_Hs}). This completes the proof.
\end{proof}
Dealing with the joint convergence of any monomial of a deterministic Toeplitz matrix ($D_n$) is easy when we consider its ``truncated version." The following lemma provides the limiting behavior of the tracial moment of a deterministic Toeplitz matrix in terms of its finite diagonal versions.
\begin{lemma} \label{lem:lim,D_Dn,m}
    Let $D_n=((d_{i-j}))$ be the deterministic Toeplitz matrix whose input entries satisfy Assumption \ref{assump:determin}. For a fixed positive integer $m$, we define a new matrix $D_{n,m}$ whose $(i,j)$-th entry is:
\begin{align} \label{def:D_n,m}
	D_{n,m}(i,j) &:= d_{i-j} \chi_{[0,m]} (|i-j|).
\end{align}
Then for $\alpha \in (0,1)$, we have
\begin{align} \label{eqn:lim+same+D=D_m}
    \lim_{n \to \infty}  \vp_n( D_{n}^p) = \lim_{n \to \infty}  \vp_n(( D_{n,n^\alpha})^p).
\end{align}
\end{lemma}

\begin{proof}
Consider the matrix $D^{(c)}_{n,m} := D_n - D_{n,m}$.
For any $\alpha \in (0,1)$, let $m=n^\alpha$, $A^{(0)}_{n}= D_{n,n^\alpha}$ and $A^{(1)}_{n}= D^{(c)}_{n,n^\alpha}$. Then for $\mu_i \in \{0,1\}$,
\begin{align*} 
	 \vp_n( D_{n})^p 
  &=  \frac{1}{n} \Tr(A^{(0)}_{n} + A^{(1)}_{n})^p\\
  &=\sum_{\mu_s \in \{0,1\} : \sum_{s=1}^p \mu_s = 0}   \frac{1}{n} \Tr[A^{(\mu_1)}_{n} A^{(\mu_2)}_{n} \cdots A^{(\mu_p)}_{n}] + \sum_{\mu_s \in \{0,1\} : \sum_{s=1}^p \mu_s \neq 0}   \frac{1}{n}\Tr[A^{(\mu_1)}_{n} A^{(\mu_2)}_{n} \cdots A^{(\mu_p)}_{n}]\\
  &=\frac{1}{n} \Tr( D_{n, n^\alpha})^p + \sum_{I_\mu(p)}   \frac{1}{n} \Tr[A^{(\mu_1)}_{n} A^{(\mu_2)}_{n} \cdots A^{(\mu_p)}_{n}],
\end{align*}
where $I_\mu(p)= \{(\mu_1, \mu_2, \ldots, \mu_p) : \mu_1, \ldots, \mu_p \in \{0, 1\} \mbox{ and } \mu_s = 1 \mbox{ for at least one } s \in \{1,2, \ldots, p\}\}.$
Thus in order to show (\ref{eqn:lim+same+D=D_m}), it is sufficient to establish 
\begin{align*}
 \lim_{n \to \infty} \frac{1}{n} \Tr[A^{(\mu_1)}_{n} A^{(\mu_2)}_{n} \cdots A^{(\mu_p)}_{n}] = 0, \ \ \forall \ (\mu_1, \mu_2, \ldots, \mu_p)  \in I_\mu(p).
\end{align*}

Now observe the fact that $\Tr[AB]= \Tr [BA]$ for any square matrices $A,B$. It is thus enough to deal with  $\Tr[A^{(\mu_1)}_{n} A^{(\mu_2)}_{n} \cdots A^{(\mu_p)}_{n}]$ for $\mu_1=1$.  
Now from Lemma \ref{lem:tracetoeplitz}, we have
\begin{align*}
\frac{1}{n}\Tr[ A^{(\mu_1)}_{n} A^{(\mu_2)}_{n} \cdots A^{(\mu_p)}_{n}] 
&=  \sum_{i_1 \in T_{\mu_1}, i_2 \in T_{\mu_2}, \ldots, i_p \in T_{\mu_p}} d_{i_1} d_{i_2} \cdots d_{i_p} \times  \frac{1}{n}\sum_{j=1}^n \prod_{\ell=1}^{p} \chi_{[1,n]}(j+\sum_{t=\ell}^{p}i_t)\d_{0,\sum_{t=1}^p i_t},
\end{align*}
where for $s=1,2, \ldots, p$,
\begin{align*}
  T_{\mu_s}
  & :=	\l\{\begin{array}{lll}
			\displaystyle  [-(n^\alpha-1), (n^\alpha -1)] & \mbox{ if }  \mu_s=0, \\ 
			\displaystyle	 [-(n-1), -(n^\alpha-1)] \cup [(n^\alpha -1), (n-1)]  & \mbox{ if }  \mu_s=1.
		\end{array}\r.
\end{align*}
Here note that $T_{\mu_1} =T_1$ and $T_{\mu_s} \subseteq [-(n-1), (n-1)]$ for all $s=2, 3,\ldots, p$. 
Also observe that for each $j,\ell$; $|\chi_{[1,n]}(j+\sum_{t=\ell}^p i_t)| \leq 1$, and $\d_{0,\sum_{t=1}^pi_t} \leq 1$. Thus we have 
\begin{align*} 
 |\frac{1}{n}\Tr[ A^{(\mu_1)}_{n} A^{(\mu_2)}_{n} \cdots A^{(\mu_p)}_{n}]| 
 & \leq \left(\sum_{i_1 = -(n-1)}^ {-(n^\alpha-1)} |d_{i_1}| + \sum_{i_1 = n^\alpha-1}^ {n-1} |d_{i_1}| \right) \sum_{i_2, i_{3}, \ldots, i_p = -(n-1)} ^ {n-1}|d_{i_{2}} d_{i_{3}} \cdots d_{i_p}|.
\end{align*}
Since $\alpha \in (0,1)$ and $\{d_i\}$ is absolutely summable, the second factor is summable and the first part is the tail part of the series $\sum_{i_1=- \infty}^\infty |d_{i_1}|$. Hence the whole quantity converges to zero as $n$ tends to infinity. 
This completes the proof.
\end{proof}
The following remark provide a generalization of Lemma \ref{lem:lim,D_Dn,m}.
\begin{remark} \label{rem:obse_D,D_m}
 Let $M_{n,1}=n^{-1/2}T_n$ where $T_n$ is a random Toeplitz matrix. For $i=1,2, \ldots, m$, let $\{M_{n,1}^{(i)}\}$ be independent copies of $M_{n,1}$ and  $M_{n,2}^{(i)}=D^{(i)}_n$ be $m$ deterministic Toeplitz matrices. Assume that the input entries of $T_n$ satisfy Assumption \ref{assump:toeI} and the input entries of $D^{(i)}_n$   satisfy Assumption \ref{assump:determin}.
Suppose $Q(T_n, P_n,D_n)$ is any monomial in $\{P_n, M^{(i)}_{n,j}; 1 \leq i \leq m, 1\leq j \leq 2\}$, and for $\alpha \in (0,1)$,  $Q(T_n, P_n,D_{n, n^\alpha})$ is the same monomial obtained by replacing $D_n$ by $D_{n, n^\alpha}$ in $Q(T_n, P_n, D_n)$.	Then $\lim_{n \to \infty}  \vp_n\big(Q(T_n, P_n, D_n)\big) = \lim_{n \to \infty}  \vp_n\big(Q(T_n, P_n, D_{n, n^\alpha})\big)$. The proof will follow from a similar argument that was used to establish (\ref{eqn:lim+same+D=D_m}) with some technical changes.
\end{remark}
Now using the above trace formula and Lemma \ref{eqn:lim+same+D=D_m}, we prove Proposition \ref{thm:jc_D+P}.
\begin{proof} [Proof of Proposition \ref{thm:jc_D+P}]
	We first consider 
 the monomial $D_n^p$. To find 
 $ \lim_{n\to \infty}\vp_n( D_{n}^p)$, we first consider the truncated matrix $D_{n,m}$, given in (\ref{def:D_n,m}) and look at its convergence. 
 From Lemma \ref{lem:tracetoeplitz}, we have
\begin{align*}
 \vp_n( (D_{n,m})^p) &= \sum_{i_1, \ldots, i_p = -(m-1)}^ {m-1} d_{i_1}\cdots d_{i_p} \times  \frac{1}{n}\sum_{j=1}^n\prod_{\ell=1}^p
\chi_{[1,n]}(j+\sum_{t=\ell}^p i_t) \d_{0,\sum_{t=1}^pi_t} \nonumber \\
& = \sum_{i_1, \ldots, i_p = -(m-1)}^ {m-1} f(i_1, \ldots, i_p) g(i_1, \ldots, i_p,n), \mbox{ say},
\end{align*}
where $f(i_1, \ldots, i_p)=d_{i_1}\cdots d_{i_p}$.
Note that 
$|g(i_1, \ldots, i_p,n)| \leq 1$ and
\begin{align} \label{eqn:lim_g_m}
\lim_{n \to \infty} g(i_1, \ldots, i_p,n)
	=\l\{\begin{array}{lll}
	1 & \mbox{if } \sum_{t=1}^pi_t=0, \\
		0 & \mbox{if } \sum_{t=1}^pi_t \neq 0.
	\end{array}\r.
\end{align}
Thus, from the above observations, for fixed $m$, we have 
\begin{align}\label{eqn:lim_Dm}
\lim_{n \to \infty}  \vp_n(( D_{n,m})^p) =  \sum_{i_1, \ldots, i_p=-(m-1)}^{(m-1)}d_{i_1}\cdots d_{i_{p-1}} d_{i_p} \d_{0,\sum_{t=1}^pi_t}.	
\end{align}

Now observe that if we take $m=n^\alpha$ with $\alpha \in (0,1)$, then (\ref{eqn:lim_g_m}) is still true. Let $\{d_i\}$ satisfy Assumption \ref{assump:determin}. Then using the Dominated convergence theorem, we have
\begin{align*} 
 \lim_{n \to \infty}  \vp_n(( D_{n,n^\alpha})^p )
	= \sum_{i_1, \ldots, i_{p}=-\infty }^{\infty} d_{i_1}\cdots d_{i_{p-1}} d_{i_p} \d_{0,\sum_{t=1}^pi_t}.
\end{align*}
Note that the existence of the above limit is guaranteed because 
$\sum_{t= -\infty}^{\infty} |d_t| < \infty$.

Now using Lemma \ref{lem:lim,D_Dn,m}, we have 
\begin{align} \label{eqn:lim_D_n^p}
\lim_{n \to \infty}  \vp_n( D_{n}^p) = \lim_{n \to \infty}  \vp_n(( D_{n,n^\alpha})^p) 
= \sum_{i_1, \ldots, i_{p}=-\infty }^{\infty} d_{i_1}\cdots d_{i_{p-1}} d_{i_p} \d_{0,\sum_{t=1}^pi_t}.
\end{align}

Similarly, we can also show that for $\e_i \in \{1,*\}$ and $\tau_i \in \{1, \ldots,m\}$,
\begin{align} \label{eqn:lim_D*_n^p}
	\lim_{n \to \infty}  \vp_n( D^{(\tau_1)\e_1}_{n} \cdots D^{(\tau_p)\e_p}_{n}) = \sum_{i_1, \ldots, i_{p}=-\infty }^{\infty} d^{(\tau_1)\e_1}_{i_1}\cdots d^{(\tau_p) \e_{p}}_{i_{p}}  \d_{0,\sum_{t=1}^p \e'_ti_t}.
\end{align}

Now we deal with an arbitrary monomial from the collection $\{P_n, D_{n}^{(\tau)}; 1 \leq \tau \leq m\}$. Note from the structure of $P_n$ that any even power and any odd power of $P_n$ are, respectively, $I_n$ and  $P_n$. Also observe that for any matrices $A_n, B_ n$; $\Tr(A_nB_n)= \Tr(B_n A_n)$. Therefore for $\e_i \in \{ 1,*\}$ and $\tau_i \in \{1, \ldots, m\}$,	it is sufficient to check the convergence of  the following monomial from the collection $\{P_n, D_{n}^{(\tau)}; 1 \leq \tau \leq m\}$:
$$\big[(P_nD_{n}^{(\tau_1)\e_1} \cdots D_{n}^{(\tau_{k_1})\e_{k_1}})  \cdots (P_nD_{n}^{(\tau_{k_{p-1}+1})\e_{k_{p-1}+1}} \cdots D_{n}^{(\tau_{k_p})\e_{k_p}})\big]=	q_{k_1, \ldots, k_p}(\e, \tau), \mbox{ say},$$
where $0\leq k_1 \leq k_2 \leq \cdots \leq k_p$.
First, let $p$ be an odd positive integer. Then observe from  Lemma \ref{lem:hankel} that
\begin{align*}
	\vp_n\big(	q_{k_1, \ldots, k_p}(\e, \tau)\big)  
	&=  \sum_{I_{k_p}} \prod_{t=1}^{k_p} d^{(\tau_{t})\e_t}_{i_t} \times \frac{1}{n} \sum_{j=1}^n \prod_{e=1}^p m_{\underline k, e}(j,i)   \d_{2j-1-n, \sum_{c=1}^p(-1)^{p-c} \sum_{t=k_{c-1}+1}^{k_{c}}  \e_t'i_t},
\end{align*}
where $I_{k_p}$ and $m_{\underline k, e}(j,i)$ are as in (\ref{eqn:i_k in -n to n})  and (\ref{eqn:m_chi,t_Hn}), respectively.
 Now for a fixed value of $i_1, \ldots, i_{k_p}$, note from (\ref{eqn:m_chi,t_Hn}) that $|m_{\underline k, e}(j,i) | \leq 1$ and thus
 \begin{align*}
 &	|\frac{1}{n} \sum_{j=1}^n \prod_{e=1}^p m_{\underline k, e}(j,i)   \d_{2j-1-n, \sum_{c=1}^p(-1)^{p-c} \sum_{t=k_{c-1}+1}^{k_{c}} \e_t'i_t}|  \\
 	&\leq \frac{1}{n}\sum_{j=1}^n  \d_{2j-1-n,\sum_{c=1}^p(-1)^{p-c} \sum_{t=k_{c-1}+1}^{k_{c}}  \e_t'i_t} = O(\frac{1}{n}).
 \end{align*}
Also observe that $\sum_{I_{k_p}} \prod_{t=1}^{k_p} d^{(\tau_{t})\e_t}_{i_t} < \infty$ from Assumption \ref{assump:determin}. Hence for odd values of $p$, 
\begin{equation}\label{eqn:phi_PD_o(1)odd}
	\vp_n\big(q_{k_1, \ldots, k_{p}}(\e, \tau)\big) =o(1).
\end{equation}

 Now suppose $p$ is even. Then from Lemma \ref{lem:hankel}, we have
 \begin{align*}
 	\vp_n\big(	q_{k_1, \ldots, k_{p}}(\e, \tau)\big)  
 	&=  \sum_{I_{k_p}} \prod_{t=1}^{k_p} d^{(\tau_{t})\e_t}_{i_t} \times \frac{1}{n} \sum_{j=1}^n \prod_{e=1}^p m_{\underline k, e}(j,i)   \d_{0, \sum_{c=1}^p(-1)^{p-c} \sum_{t=k_{c-1}+1}^{k_{c}}  \e_t'i_t}.
 \end{align*}
Let $q'_{k_1, \ldots, k_{p}}(\e, \tau)$ be a monomial which is obtained by replacing $D_n$ by $D_{n, n^\alpha}$ in $q_{k_1, \ldots, k_{p}}(\e, \tau)$, where $D_{n, n^\alpha}$ is as in (\ref{def:D_n,m}) with $\alpha \in (0,1)$. 
Using 
arguments similar to  those used while establishing (\ref{eqn:lim_D_n^p}), we have
\begin{align*} 
	\lim_{n \to \infty}  \vp_n\big(	q'_{k_1, \ldots, k_{p}}(\e, \tau)\big) 
	= \sum_{i_1, \ldots, i_{k_p}=-\infty }^{\infty}  \prod_{t=1}^{k_p} d^{(\tau_{t})\e_t}_{i_t} \d_{0, \sum_{c=1}^p(-1)^{p-c} \sum_{t=k_{c-1}+1}^{k_{c}}  \e_t'i_t}.
\end{align*}
But from Remark \ref{rem:obse_D,D_m}, $	\lim_{n \to \infty}  \vp_n\big(	q'_{k_1, \ldots, k_{p}}(\e, \tau)\big) = 	\lim_{n \to \infty}  \vp_n\big(	q_{k_1, \ldots, k_{p}}(\e, \tau)\big)$. Hence
	\begin{align}\label{eqn:lim_phi_P,D*}
	&	\lim_{n \to \infty}  \vp_n \big[(P_nD_{n}^{(\tau_1)\e_1} \cdots D_{n}^{(\tau_{k_1})\e_{k_1}})  \cdots (P_nD_{n}^{(\tau_{k_{p-1}+1})\e_{k_{p-1}+1}} \cdots D_{n}^{(\tau_{k_p})\e_{k_p}})\big] \nonumber \\
	&	= 
	\l\{\begin{array}{lll}
		\displaystyle \sum_{i_1, \ldots, i_{k_p}=-\infty }^{\infty}  \prod_{t=1}^{k_p} d^{(\tau_{t})\e_t}_{i_t} \d_{0, \sum_{c=1}^p(-1)^{p-c} \sum_{t=k_{c-1}+1}^{k_{c}}  \e_t'i_t} &  \mbox{ if $p$ is even}, \\
		0&   \mbox{ if $p$ is odd}.
	\end{array}\r.
\end{align}
This completes the proof of Proposition \ref{thm:jc_D+P}.
\end{proof}

\subsection{Joint convergence of \texorpdfstring{$T_n$ and  $D_{n}$}{tndn}} \label{subsec:td}
\begin{proposition} \label{thm:jc_T+D}
Let $A_{n,1}=n^{-1/2}T_n$ where $T_n$ is a random Toeplitz matrix.  Let $\{A_{n,1}^{(i)}\}$, $i =1,2,\ldots, m$, be independent copies of $A_{n,1}$. Let $A_{n,2}^{(i)}=D^{(i)}_n, 1 \leq i \leq m$ be $m$ deterministic Toeplitz matrices. Suppose the input entries of $T_n$  and $D^{(i)}_n$  satisfy Assumptions \ref{assump:toeI} and \ref{assump:determin}, respectively.
 Then $\{A_{n,j}^{(i)}; 1\leq i\leq m, 1\leq j \leq 2\}$ converge jointly.
	The limit $*$-moments are as given in (\ref{eqn:lim_ph(Dm*,T*)}).
\end{proposition}

\begin{proof}
First we deal with the monomial of the form $(D_{n}^p (n^{-1/2}T_n)^q)$. Recall the definition of $D_{n,m}$ from (\ref{def:D_n,m}).  Observe from Remark \ref{rem:obse_D,D_m} that to find the limit of $\vp_n(D_{n}^p (n^{-1/2}T_n)^q)$, it is sufficient to find the limit of  $\vp_n(D_{n, n^\alpha}^p (n^{-1/2}T_n)^q)$, where $\alpha \in (0,1)$, and both the limits will be the same. Now
from Lemma \ref{lem:tracetoeplitz}, we have
\begin{align} \label{eqn:ph(Dm,T)}
	&\vp_n(D_{n, n^\alpha}^p (n^{-1/2}T_n)^q)  \nonumber \\
	&= \hskip-5pt \sum_{i_1, \ldots, i_p=-(n^\alpha-1)}^{n^\alpha-1} \prod_{t=1}^{p} d_{i_t} \times  \frac{1}{n^{\frac{q}{2}}} \sum_{i_{p+1}, \ldots, i_{p+q}=-(n-1)}^{n-1} \hskip-5pt \E\big[ \prod_{t=p+1}^{p+q} a_{i_{t}}\big]  \frac{1}{n}\sum_{j=1}^n\prod_{\ell=1}^{p+q}
	\chi_{[1,n]}(j+\sum_{t=\ell}^{p+q} i_t) \d_{0,\sum_{t=1}^{p+q} i_t}.
\end{align}
Observe that 
$$ |\frac{1}{n}\sum_{j=1}^n\prod_{\ell=1}^{p+q}
\chi_{[1,n]}(j+\sum_{t=\ell}^{p+q} i_t)| \leq 1 \mbox{ and }  \sum_{i_1, \ldots, i_p=-\infty}^{\infty} \prod_{t=1}^{p} d_{i_t} < \infty.$$
Thus $\vp_n(D_{n, n^\alpha}^p (n^{-1/2}T_n)^q)$ is $O(1)$ iff  for fixed values of $i_1, \ldots, i_p$,
$$ \frac{1}{n^{\frac{q}{2}}} \sum_{i_{p+1}, \ldots, i_{p+q}=-(n-1)}^{n-1} \E \big[\prod_{t=p+1}^{p+q} a_{i_{t}}\big]   \d_{0,\sum_{t=1}^{p+q} i_t} =O(1).$$

Now from the previous arguments as used in the proof of Proposition \ref{thm:toeplitz}, the above expression is true iff $q$ is even and $(i_{p+1}, \ldots, i_{p+q})$ is pair-matched with $\sum_{t=p+1}^{p+q} i_t=0$. Thus (\ref{eqn:ph(Dm,T)}) becomes
\begin{align*} 
	&\vp_n(D_{n, n^\alpha}^p (n^{-1/2}T_n)^q) \\
	&= \sum_{i_1, \ldots, i_p=-(n^\alpha-1)}^{n^\alpha-1} \prod_{t=1}^{p} d_{i_t} \d_{0,\sum_{t=1}^{p} i_t} \nonumber \\
	& \quad \times  \frac{1}{n^{\frac{q}{2}}} \sum_{i_{p+1}, \ldots, i_{p+q}=-(n-1)}^{n-1} \hskip-5pt \E\big[ \prod_{t=p+1}^{p+q} a_{i_{t}}\big]  \frac{1}{n}\sum_{j=1}^n\prod_{\ell=1}^{p+q}
	\chi_{[1,n]}(j+\sum_{t=\ell}^{p+q} i_t) \d_{0,\sum_{t=p+1}^{p+q} i_t} +o(1).
\end{align*}
Since $i_1, \ldots, i_p \in (-(n^\alpha-1), (n^\alpha-1))$ with $\alpha \in (0,1)$, we have
\begin{equation} \label{eqn:chi_D_remove}
 \frac{1}{n}\sum_{j=1}^n\prod_{\ell=1}^{p+q}
 \chi_{[1,n]}(j+\sum_{t=\ell}^{p+q} i_t) = \frac{1}{n}\sum_{j=1}^n\prod_{\ell=p+1}^{p+q}
 \chi_{[1,n]}(j+\sum_{t=\ell}^{p+q} i_t) +o(1).
\end{equation}
Hence 
\begin{align} \label{eqn:lim_ph(Dm,T)}
	& \lim_{n \to \infty} \vp_n(D_{n, n^\alpha}^p (n^{-1/2}T_n)^q) \nonumber \\ 
	&= \sum_{i_1, \ldots, i_p=-\infty}^{\infty} \prod_{t=1}^{p} d_{i_t} \d_{0,\sum_{t=1}^{p} i_t} \nonumber \\ 
	& \qquad \times  \lim_{n \to \infty} \frac{1}{n^{\frac{q}{2}}} \sum_{i_{p+1}, \ldots, i_{p+q}=-(n-1)}^{n-1} \E\big[ \prod_{t=p+1}^{p+q} a_{i_{t}}\big]  \frac{1}{n}\sum_{j=1}^n\prod_{\ell=p+1}^{p+q}
	\chi_{[1,n]}(j+\sum_{t=\ell}^{p+q} i_t) \d_{0,\sum_{t=p+1}^{p+q} i_t} 
	\nonumber \\
	& = \lim_{n \to \infty}  \vp_n( D_{n}^p )  \times  \lim_{n \to \infty}  \vp_n( (n^{-1/2}T_n)^q),
\end{align}
where the existence of the first limit is given in (\ref{eqn:lim_D_n^p}) and the second limit in the above equation is guaranteed by  Proposition \ref{thm:toeplitz} with the limit value as in (\ref{eqn:lim_mome_A1_TJc}).

Now for $\e_i \in \{1, *\}$ and $\tau_i \in \{1,2, \ldots, m\}$, we calculate the limit for the monomial $q_{ k}(\e, \tau) = D_{n}^{(\tau_1) \e_1} \cdots D_{n}^{(\tau_p) \e_p} (n^{-1/2} T_{n}^{(\tau_{p+1}) \e_{p+1}}) \cdots( n^{-1/2} T_{n}^{(\tau_k) \e_k})$.
 Note from Remark \ref{rem:obse_D,D_m} that the limit of $\vp_n(q_{ k}(\e, \tau))$ will be the same as $\vp_n(q'_{ k}(\e, \tau))$, where $q'_{k}(\e, \tau)$ is the monomial obtained by replacing $D_{n}$ by $D_{n,n^\alpha}$ in $q_{ k}(\e, \tau)$. Now from Lemma \ref{lem:tracetoeplitz}, we have
\begin{align*} 
	\vp_n(q'_{k}(\e, \tau)) 
	&= \sum_{i_{1}, \ldots, i_{p}=-(n^\alpha-1)}^{n^\alpha-1} \prod_{t=1}^{p} d^{(\tau_t) \e_t}_{i_t}  \nonumber \\
	& \times  \frac{1}{n^{\frac{k-p}{2}}} \sum_{i_{p+1}, \ldots, i_{k}=-(n-1)}^{n-1} \E\big[ \prod_{t=p+1}^{k} a^{(\tau_t) \e_t}_{i_{t}}\big]  \frac{1}{n}\sum_{j=1}^n\prod_{\ell=1}^{k}
	\chi_{[1,n]}(j+\sum_{t=\ell}^{k} \e'_t i_t) \d_{0,\sum_{t=1}^{k} \e'_t i_t}.
\end{align*}
Following similar arguments to those used to establish (\ref{eqn:lim_ph(Dm,T)}), we get
\begin{align*} 
	& \lim_{n \to \infty} \vp_n(q_{k}(\e, \tau))  \nonumber \\
	& = \lim_{n \to \infty}  \vp_n(D_{n}^{(\tau_1)\e_1} \cdots D_{n}^{(\tau_{p})\e_{p}}) \times  \lim_{n \to \infty}  \vp_n \big( (n^{-1/2}T^{(\tau_{p+1}) \e_{p+1}}_n) \cdots (n^{-1/2}T^{(\tau_{k}) \e_{k}}_n) \big).
\end{align*}
Here the existence of the first limit is given in (\ref{eqn:lim_D*_n^p}) and the second limit is guaranteed by Proposition \ref{thm:toeplitz} with the limit value as in (\ref{eqn:lim_phi(T*tau)}).

Now for $\mu_i \in \{1,2\}$, an arbitrary monomial from collection $\{T_{n}^{(i)}, D_{n}^{(i)}; 1\leq i\leq m\}$ looks like $A_{n,\mu_1}^{(\tau_1) \e_1} A_{n,\mu_2} ^{(\tau_2) \e_2} \cdots A_{n,\mu_k}^{(\tau_k) \e_k}$, where $A^{(\tau_t) \e_t}_{n,1}=n^{-1/2}T^{(\tau_t) \e_t}_n$ and $A^{(\tau_t) \e_t}_{n,2}=D^{(\tau_t) \e_t}_n$. Using similar calculation as used above, we have
\begin{align} \label{eqn:lim_ph(Dm*,T*)}
	&\lim_{n \to \infty} \vp_n \big(A_{n,\mu_1}^{(\tau_1) \e_1} A_{n,\mu_2} ^{(\tau_2) \e_2} \cdots A_{n,\mu_k}^{(\tau_k) \e_k} \big)   \nonumber \\
	& =	\lim_{n \to \infty} \vp_n \big( \prod_{t \in \{v_1, v_2, \ldots, v_p\}} D_{n}^{(\tau_t) \e_t}  \big) \times 	\lim_{n \to \infty} \vp_n \big(\prod_{t \in \{[k] \setminus \{v_1, v_2, \ldots, v_p\} \}} \frac{T_{n}^{(\tau_t) \e_t}}{\sqrt{n}} \big),
\end{align}
where $ \{v_1, v_2, \ldots, v_p\}$ are the indices corresponding to the positions of $D_n$ in the monomial $ A_{n,\mu_1}^{(\tau_1) \e_1} A_{n,\mu_2} ^{(\tau_2) \e_2} \cdots A_{n,\mu_k}^{(\tau_k) \e_k}$. Here the  first limit is of the form (\ref{eqn:lim_D*_n^p}) and the second limit is of the form (\ref{eqn:lim_phi(T*tau)}).
This completes the proof of Proposition \ref{thm:jc_T+D}.
\end{proof}


\subsection{Joint convergence of \texorpdfstring{$T_n$ and  $P_{n}$}{tnpn}} \label{subsec:Tp}
We now focus on the joint convergences of $T_n$ and $P_n$, and 
of symmetric Hankel matrices.
\begin{proposition} \label{thm:jc_T+P}
	If $\{T_{n}^{(i)}; 1\leq i\leq m\}$ are $m$ independent copies of random Toeplitz matrices whose input entries satisfy Assumption \ref{assump:toeI}, then $\{P_n, n^{-1/2}T_{n}^{(i)}; 1 \leq i \leq m\}$ converge jointly. 
	The limit $*$-moments are as given in (\ref{eqn:lim_phi(Te*,P)}).
\end{proposition}
\begin{proof}
	[Proof of Proposition \ref{thm:jc_T+P}]
	The main idea of the proof is similar to the proof of Proposition \ref{thm:toeplitz}. For simplicity of notation, we first consider a single matrix. Note from the structure of $P_n$ and trace property, it is sufficient to deal with the following monomial from the collection $\{P_n,T_n\}$:
	\begin{equation*}
		(P_n \frac{T_{n}^{\e_1}}{\sqrt{n}} \cdots \frac{T_{n}^{\e_{k_1}}}{\sqrt{n}}) (P_n \frac{T_{n}^{\e_{k_1+1}}}{\sqrt{n}} \cdots \frac{T_{n}^{\e_{k_2}}}{\sqrt{n}} )  \cdots (P_n \frac{T_{n}^{\e_{k_{p-1}+1}}}{\sqrt{n}} \cdots \frac{T_{n}^{\e_{k_p}}}{\sqrt{n}} )= q_{k_1, \ldots, k_p}(\e), \mbox{  say},
	\end{equation*}
 where $0\leq k_1 \leq k_2 \leq \cdots \leq k_p$.
First suppose $p$ is odd. Then using arguments similar to those used while establishing (\ref{eqn:phi_PD_o(1)odd}), we get $\phi_n(q_{k_1, \ldots, k_p}(\e))= o(1)$. 

Now suppose $p$ is even. Again, from a similar argument that was used to establish (\ref{eqn:phi_To(1)odd}), we get $\phi_n(q_{k_1, \ldots, k_p}(\e))= o(1)$ when $k_p$ is odd. So let $k_p=2k$ and for $I_{2k}$ as in  (\ref{eqn:i_k in -n to n}), set
	\begin{equation*} 
		I_{2k}^{o}=\{(i_1,\ldots, i_{k_1}, i_{k_1+1}, \ldots, i_{k_2}, \ldots, i_{k_{p-1}+1}, \ldots, i_{2k})\in I_{2k}\; :\;\sum_{c=1}^{p}(-1)^c \sum_{t=k_{c-1}+1}^{k_c} \e'_t i_t=0  \}.
	\end{equation*}
From Lemma  \ref{lem:hankel}, we have
		\begin{align*} 
		\lim_{n \to \infty}	\vp_n\big(q_{k_1, \ldots, 2k}(\e)\big)
		&=\lim_{n \to \infty}\frac{1}{n^{k+1}}\sum_{j=1}^{n}\sum_{I_{2k}^{o}} \E[\prod_{t=1}^{2k} a^{\e_t}_{i_t}] \prod_{e=1}^{2k} m_{\underline k, e}(j,i),
	\end{align*}
	where $m_{\underline k, e}(j,i)$ is as in (\ref{eqn:m_chi,t_Hn}).
 As before, only the pair-partitions will contribute. 
	Hence 
	\begin{align}\label{eqn:EHs_2k}
		\lim_{n \to \infty}	\vp_n\big(q_{k_1, \ldots, 2k}(\e)\big)
		&=\sum_{\pi\in \mathcal P_{2}(2k)}\lim_{n\to \infty}\frac{1}{n^{k+1}}\sum_{j=1}^{n}\sum_{I_{2k}^{o}(\pi)}\prod_{(r,s)\in \pi}\E[a^{\e_r}_{i_r} a^{\e_s}_{i_s}] \prod_{e=1}^{2k} m_{\underline k, e}(j,i),
	\end{align}
where  $I_{2k}^{o}(\pi)=\{(i_1,\ldots, i_{2k})\in I_{2k}^{o} \; :\; \prod_{(r,s)\in \pi}\E[a^{\e_r}_{i_r} a^{\e_s}_{i_s}]\neq 0\}$.
	
	For a block $(r,s)\in \pi$,  let $r \in \{k_{c-1}, k_{c-1}+1, \ldots, k_{c}\}$ and $s \in \{k_{c'-1}, k_{c'-1}+1, \ldots, k_{c'}\}$ for some $c,c' \in \{1,2, \ldots,p\}$ with $k_0=1$ and $k_p=2k$. Observe that, the constraint on the indices, $\sum_{c=1}^{p}(-1)^c \sum_{t=k_{c-1}+1}^{k_c} \e'_t i_t$ can also be written as  $\sum_{c=1}^{p} \nu_c \sum_{t=k_{c-1}+1}^{k_c} \e'_t i_t$, where \begin{align}\label{eqn:nu_t_oddeven}
		\nu_c= 
		\l\{\begin{array}{lll}
			1 &  \mbox{ if $c$ is even}, \\
			-1&   \mbox{ if $c$ is odd}.
		\end{array}\r.
	\end{align}
Thus, in terms of a pair-partition, the constraint on the indices will be $\sum_{(r,s)\in \pi}(\nu_c \e'_ri_r+\nu_{c'} \e'_si_s)=0$. Here $i_r$ matches with $i_s$.  Observe that if $r,s \in \{k_d, k_d+1, \ldots, k_{d+1}\}$ for some $d=0,1, \ldots, (p-1)$, then the corresponding $\nu_c,\nu_{c'}$ have the same sign.

	Similar to \eqref{eqn:reduction}, the number of free indices for $ (i_1,\ldots, i_{k_1}, i_{k_1+1}, \ldots, i_{k_2}, \ldots, i_{k_{p-1}+1}, \ldots, i_{2k})\in I_{2k}^{o}(\pi)$ will be $k$, if and only if  for each block $(r,s)\in \pi$ 
	\begin{align*}
		\nu_c \e'_r i_r+\nu_{c'} \e'_s i_s=0 , \mbox{ 	equivalently } i_r=
		\l\{\begin{array}{rl}
			i_s & \mbox{ if } \nu_c\nu_{c'} \e'_r\e'_s =-1, \\
			-i_s & \mbox{ if } \nu_c \nu_{c'} \e'_r\e'_s=1.
		\end{array}
		\r.
	\end{align*}
	Since the input entries of the matrices satisfy Assumption \ref{assump:toeI}, following computations similar to those used in establishing (\ref{eqn:E[ars_epsilon]}), we have a non-zero contribution in the limit iff
	\begin{align}\label{eqn:EHs_ars_epsilon}
	&\E[a^{\e_{r}}_{i_{r}} a^{\e_{s}}_{i_{s}}] \nonumber \\
	&= 
	\begin{cases}
	(\sigma_x^2+\sigma_y^2) \d_{\nu_c, \nu_{c'}} (1-\d_{\e'_{r}, \e'_{s} })  + [\sigma_x^2-\sigma_y^2+\e'_{r} 2\mathrm{i} \rho_1 ] (1-\d_{\nu_c, \nu_{c'}}) \d_{\e'_{r}, \e'_{s} }  &\mbox{ if }   i_{r}=i_{s}>0, \\
	(\sigma_x^2+\sigma_y^2) \d_{\nu_c, \nu_{c'}} (1-\d_{\e'_{r}, \e'_{s} })  + [\sigma_x^2-\sigma_y^2+\e'_{r} 2\mathrm{i} \rho_6 ] (1-\d_{\nu_c, \nu_{c'}}) \d_{\e'_{r}, \e'_{s} }   &\mbox{ if }   i_{r}=i_{s}<0, \\
	[\rho_2+\rho_5 + \e'_{r} \mathrm{i}(\rho_4-\rho_3) ](1-\d_{\nu_c, \nu_{c'}}) (1-\d_{\e'_{r}, \e'_{s}}) \\ 
	\quad + [\rho_2-\rho_5 + \e'_{r} \mathrm{i}(\rho_4+\rho_3) ] \d_{\nu_c, \nu_{c'}} \d_{\e'_{r}, \e'_{s} } &\mbox{ if }   i_{r}=-i_{s}>0, \\
	[\rho_2+\rho_5 + \e'_{s} \mathrm{i}(\rho_4-\rho_3) ](1-\d_{\nu_c, \nu_{c'}})(1-\d_{\e'_{r}, \e'_{s} }) \\
	\quad	+ [\rho_2-\rho_5 + \e'_{s} \mathrm{i}(\rho_4+\rho_3) ] \d_{\nu_c, \nu_{c'}} \d_{\e'_{r}, \e'_{s}}  &\mbox{ if }   i_{r}=-i_{s}<0, 
	\end{cases}  \nonumber \\
		& = \theta^{(T,P)}_{r,s}(\e,\rho,i), \mbox{ say}.
	\end{align}
	Recall, for $\pi=(r_1,s_1)\cdots (r_k,s_k)$ with $r_t \leq s_t$, we have  $\pi'(r_t)=\pi'(s_t)=t$. Define $\eta_{\pi}(r_t)=1$,   and $\eta_{\pi}(s_t)=(-1)^{\d_{\e'_{r_t} \e'_{s_t}, \nu_{c} \nu_{c'}}}$. Also for a set of variables $z_0,z_1, \ldots, z_{2k}$, we define $\theta^{(T,P)}_{r,s}(\e,\rho,z)$ as $\theta^{(T,P)}_{r,s}(\e,\rho,i)$ where $i_r,i_s$ are respectively replaced by $z_r,z_s$,
	and 
	for $e=1,2, \ldots,p$,
	\begin{align} \label{eqn:m_chi,t,z_Hn}
		m_{\underline k, e}(z) & = \prod_{\ell=k_{e-1}+1}^{k_e} \chi_{[0,1]}\big(z_0+ (-1)^{e} \sum_{t=\ell}^{k_e} \e'_t \eta_{\pi}(t) z_{\pi'(t)} + \sum_{c=e+1}^{p} (-1)^{c}  \sum_{t=k_{c-1}+1}^{k_c} \e'_t \eta_{\pi}(t) z_{\pi'(t)} \big).
	\end{align}
	
	Using  \eqref{eqn:EHs_ars_epsilon} in (\ref{eqn:EHs_2k}) under the above notation, and convergence of Riemann sums, we get 
	\begin{align} \label{eqn:lim_phi(T*,P)}
		\lim_{n \to \infty} \vp_n \big(q_{k_1, \ldots, 2k}(\e)\big) 
		= \sum_{\pi\in \mathcal P_{2}(2k)} \int_0^1\int_{[-1,1]^k} \prod_{(r,s)\in \pi} \theta^{(T,P)}_{r,s}(\e,\rho,z) \prod_{e=1}^{p} m_{\underline k, e}(z) \prod_{i=0}^kdz_i.
	\end{align}
	This establishes the limit for a single matrix.
 
	Now for $\tau=(\tau_1,\ldots,\tau_{k_p})$ and $\e=(\e_1, \ldots, \e_{k_p})$ with $\tau_i\in \{1,\ldots, m\}$ and $\e_i \in \{1,*\}$, let 
	$$ q_{k_1, \ldots, k_p}(\tau, \e) :=\Big(  (P_n \frac{T_{n}^{(\tau_1)\e_1}}{\sqrt{n}} \cdots \frac{T_{n}^{(\tau_{k_1})\e_{k_1}}}{\sqrt{n}})   \cdots (P_n \frac{T_{n}^{(\tau_{k_{p-1}+1})\e_{k_{p-1}+1}}}{\sqrt{n}} \cdots \frac{T_{n}^{(\tau_{k_p})\e_{k_p}}}{\sqrt{n}} ) \Big).$$
Using   Lemma \ref{lem:hankel} and the above techniques, one can show that the limit of $ \vp_n \big( q_{k_1, \ldots, k_p}(\tau, \e) \big)$ is zero if $p$ is odd or $k_p$ is odd. Let $p$ be even and $k_p=2k$. Then using arguments similar  to those used while establishing (\ref{eqn:EHs_2k}), we have
	\begin{align*}
		&\lim_{n \to \infty}  \vp_n \big( q_{k_1, \ldots, k_p}(\tau, \e) \big) 
		& = \sum_{\pi\in \mathcal P_{2}(2k)} \lim_{n\to \infty} \frac{1}{n^{k+1}}\sum_{j=1}^{n}\sum_{I_{2k}^{o}(\pi)}\prod_{(r,s)\in \pi} \E[a^{(\tau_r)  \e_r}_{i_r} a^{(\tau_s) \e_s}_{i_s}] \prod_{e=1}^{p} m_{\underline k, e}(j,i),
	\end{align*}
	where $m_{\underline k, e}(j,i)$ is as in (\ref{eqn:EHs_2k}).

	Following the arguments which were  used to establish \eqref{eqn:EHs_ars_epsilon}, a contribution of the order $O(1)$ from a summand on the right side can only happen if   
	$$
	\E[a^{(\tau_r)  \e_r}_{i_r} a^{(\tau_s) \e_s}_{i_s}] =\d_{\tau_r,\tau_s}\theta^{(T,P)}_{r,s}(\e,\rho,i),
	$$
 where $\theta^{(T,P)}_{r,s}(\e,\rho,i)$ is as in \eqref{eqn:EHs_ars_epsilon}. Thus, using steps similar to those used to obtain (\ref{eqn:lim_phi(T*,P)}) for a single $T_{n}$, we get 
	\begin{align} \label{eqn:lim_phi(Te*,P)}
	&	\lim_{n \to \infty} \vp_n\big( q_{k_1, \ldots, k_p}(\tau, \e)\big) \nonumber \\
=&
	\l\{\begin{array}{lll}
	\displaystyle \hskip-5pt \sum_{\pi\in \mathcal P_{2}(2k)} \int_0^1 \hskip-3pt \int_{[-1,1]^k}  \prod_{(r,s)\in \pi} \hskip-7pt \d_{\tau_r,\tau_s} \theta^{(T,P)}_{r,s}(\e,\rho,z) \prod_{e=1}^{p} m_{\underline k, e}(z) \prod_{i=0}^kdz_i, &  \mbox{ if $k_p=2k$ and $p$ is even}, \\
			0&   \mbox{ otherwise},
		\end{array}\r.
	\end{align}
	where $\theta^{(T,P)}_{r,s}(\e,\rho,z)$, $m_{\underline k, e}(z)$ are as in \eqref{eqn:EHs_ars_epsilon}, (\ref{eqn:m_chi,t,z_Hn}), respectively.
	This completes the proof of Proposition \ref{thm:jc_T+P}.
\end{proof}

	\begin{remark}\label{cor:com_HJc}
Recall that any symmetric Hankel matrix is of the form $H_{n,s}=P_nT_n$. While the joint convergence of $\{H_{n,s}^{(i)}; 1\leq i \leq m\}$ follows from Proposition \ref{thm:jc_T+P}, their limit $*$-moments
cannot be easily deduced from (\ref{eqn:lim_phi(Te*,P)}). This is because 
$H^*_{n,s}=  T^*_nP_n$ and so, for the $*$-moment of a monomial, the positions of the $P_n$'s  depend on the particular monomial, and becomes crucial.
 That is, we have to know 
 the values of $\e_1, \e_2, \ldots, \e_p$. 
We now show how to calculate the moments directly.
		
 Suppose  $\{H_{n,s}^{(\tau)}; 1\leq \tau \leq m \}$ are $m$ independent copies of symmetric Hankel matrices 
		whose input entries satisfy Assumption \ref{assump:toeI}.
  Then for any $\e_i\in \{1,*\}$ and $\tau_i\in \{1,\ldots, m\}$,
	\begin{align} \label{eqn:lim_mome_A1_HJc}
	&\lim_{n\to \infty}\vp_n(\dfrac{H_{n,s}^{(\tau_1)\e_1}}{n^{1/2}}\cdots \dfrac{H_{n,s}^{(\tau_{p})\e_{p}}}{n^{1/2}}) \nonumber \\
		& =\l\{\begin{array}{lll}			\displaystyle \sum_{\pi\in \mathcal P_{2}(2k)} \int_0^1\int_{[-1,1]^k}  \prod_{(r,s)\in \pi} \d_{\tau_r,\tau_s} \theta^{(H)}_{r,s}(\e,\rho,z) \prod_{t=1}^{2k} m_{k,t}(z) \prod_{i=0}^kdz_i &  \mbox{ if } p=2k, \\
		0 & \mbox{ if } p=2k+1,
	\end{array}\r.
		\end{align}
		where for a set of variables $z_0, z_1, \ldots, z_{2k}$,
			\begin{align}\label{eqn:EHs_ars_epsilon_Hs}
			&	\theta^{(H)}_{r,s}(\e,\rho,z) \nonumber  \\
			&= 
			\begin{cases}
				(\sigma_x^2+\sigma_y^2) (1-\d_{\e'_{r}, \e'_{s} })   + [\sigma_x^2-\sigma_y^2+\e'_{r} 2\mathrm{i} \rho_1 ] \d_{\e'_{r}, \e'_{s} }  &\mbox{ if }   z_{r}=z_{s}>0, \\
				(\sigma_x^2+\sigma_y^2) (1-\d_{\e'_{r}, \e'_{s} })   + [\sigma_x^2-\sigma_y^2+\e'_{r} 2\mathrm{i} \rho_6 ] \d_{\e'_{r}, \e'_{s} }  &\mbox{ if }   z_{r}=z_{s}<0, \\
				[\rho_2+\rho_6 + \e'_{r} \mathrm{i}(\rho_4-\rho_3) ](1-\d_{\e'_{r}, \e'_{s} }) + [\rho_2-\rho_6 + \e'_{r} \mathrm{i}(\rho_4+\rho_3) ] \d_{\e'_{r}, \e'_{s} } &\mbox{ if }   z_{r}=-z_{s}>0, \\
				[\rho_2+\rho_6 + \e'_{s} \mathrm{i}(\rho_4-\rho_3) ](1-\d_{\e'_{r}, \e'_{s} }) + [\rho_2-\rho_6 + \e'_{s} \mathrm{i}(\rho_4+\rho_3) ] \d_{\e'_{r}, \e'_{s} }  &\mbox{ if }   z_{r}=-z_{s}<0, \\
			\end{cases} 
		\end{align}
		and $m_{k,t}(z)  = \chi_{[0,1]}\big(z_0+ \sum_{\ell=t}^{2k} (-1)^{\ell} \eta_{\pi}(\ell) z_{\pi'(\ell)}\big) $ with $\eta_{\pi}(\ell)$  as in (\ref{eqn:eta_pi_Hns}).

To prove (\ref{eqn:lim_mome_A1_HJc}), first note the following trace formula for symmetric Hankel matrices:
	Let $H_{n,s}^{(1)}, \ldots, H_{n,s}^{(m)}$ be $m$ symmetric Hankel matrices with input sequences $(a_i^{(1)})_{i\in \Z}, \ldots, (a_i^{(m)})_{i\in \Z}$, respectively. Then for $\tau_1,\ldots, \tau_p\in \{1,\ldots, m\}$ and $\e_1,\ldots, \e_p\in \{1,*\}$, we have
	\begin{align*} 
		& \Tr\big(H_{n,s}^{(\tau_1)\e_1}\cdots H_{n,s}^{(\tau_p)\e_p}\big) \nonumber \\
		& \quad =
		\left\{ \begin{array}{ll}
		\displaystyle{ \sum_{j=1}^n \sum_{I_p} \prod_{t=1}^p a^{(\tau_t)\e_t}_{i_t} \prod_{t=1}^p \chi_{[1,n]}\big(j+ \sum_{\ell=t}^p i_\ell (-1)^{p-\ell}\big) \d_{0,\sum_{t=1}^pi_t (-1)^{p-t}}} & \mbox{ if $p$ is even},
			\\
		\displaystyle{\sum_{j=1}^n \sum_{I_p} \prod_{t=1}^p a^{(\tau_t)\e_t}_{i_t} \prod_{t=1}^p \chi_{[1,n]} \big(j+ \sum_{\ell=t}^p i_\ell (-1)^{p-\ell}\big) \d_{2j-1-n,\sum_{t=1}^p i_t (-1)^{p-t}}} & \mbox{ if $p$ is odd},
		\end{array}
		\right.
	\end{align*}
	where $I_p$ is as in (\ref{eqn:i_k in -n to n}).

We skip the proof of the above trace formula and refer to the idea of the proofs of Lemmas \ref{lem:tracetoeplitz} and \ref{lem:hankel}.	

Now we prove (\ref{eqn:lim_mome_A1_HJc}), along the lines of 
 the proof of Proposition \ref{thm:jc_T+P}. For simplicity of notation, we only consider a single matrix. The same idea will also work when we have $m$ matrices. Note that the limiting odd moments are zero. 
	So let $p=2k$.
	 As before, only the pair-partitions will contribute. 
	Hence from the above trace formula, we have
	\begin{align}\label{eqn:EHs_2k_Hs}
		\lim_{n \to \infty} 	\vp_n( \frac{H_{n,s}^{\e_1}}{\sqrt{n}} \cdots \frac{H_{n,s}^{\e_{2k}}}{\sqrt{n}})
		&=\lim_{n \to \infty}\frac{1}{n^{k+1}}\sum_{j=1}^{n}\sum_{I_{2k}^{''}} \E[\prod_{t=1}^{2k} a^{\e_t}_{i_t}] \prod_{t=1}^{2k} \chi_{[1,n]} \big(j+ \sum_{\ell=t}^{2k} i_\ell (-1)^{\ell} \big) \nonumber \\
		&= \hskip-5pt \sum_{\pi\in \mathcal P_{2}(2k)}\lim_{n\to \infty}\frac{1}{n^{k+1}}\sum_{j=1}^{n}\sum_{I_{2k}^{''}(\pi)}\prod_{(r,s)\in \pi}\E[a^{\e_r}_{i_r} a^{\e_s}_{i_s}] \prod_{t=1}^{2k} \chi_{[1,n]}\big(j+ \sum_{\ell=t}^{2k} i_\ell (-1)^{\ell} \big),
	\end{align}
	where $I^{''}_{2k}(\pi)=\{(i_1,\ldots, i_{2k})\in I_{2k}^{''} \; :\; \prod_{(r,s)\in \pi}\E[a^{\e_r}_{i_r} a^{\e_s}_{i_s}]\neq 0\}$ with
	\begin{equation}\label{eqn:I'2k_Hs}
		I_{2k}^{''}=\{(i_1,\ldots, i_{2k})\in I_{2k}\; :\;\sum_{t=1}^{2k}(-1)^ti_t=0  \}.
	\end{equation}
	
	Note that, similar to \eqref{eqn:reduction}, the number of free indices for $ (i_1,\ldots, i_{2k})\in I_{2k}^{''}(\pi)$ will be $k$, if and only if  for each block $(r,s)\in \pi$ 
	\begin{align*}
	\nu_ri_r+\nu_si_s=0 , \mbox{ 	equivalently } i_r=
		\l\{\begin{array}{rl}
			i_s & \mbox{ if } \nu_r\nu_s=-1, \\
			-i_s & \mbox{ if } \nu_r\nu_s=1,
		\end{array}
		\r.
	\end{align*}
	where $\nu_t$ is as defined in (\ref{eqn:nu_t_oddeven}).
	Since the input entries of matrices satisfy Assumption \ref{assump:toeI}, we have a non-zero contribution in the limit if and only if  $\E[a^{\e_{r}}_{i_{r}} a^{\e_{s}}_{i_{s}}]= \theta^{(H)}_{r,s}(\e,\rho,i)$,	where 
		$\theta^{(H)}_{r,s}(\e,\rho,i)$ is as in (\ref{eqn:EHs_ars_epsilon_Hs}) with $z_r,z_s$ are respectively replaced by $i_r,i_s$.
	
	Recall, for $\pi=(r_1,s_1)\cdots (r_k,s_k)$, we have  $\pi'(r_t)=\pi'(s_t)=t$. Define 
 \begin{equation}\label{eqn:eta_pi_Hns}
\mbox{$\eta_{\pi}(r_t)=1$   and $\eta_{\pi}(s_t)=(-1)^{\d_{\nu_{r_t}, \nu_{s_t}}}$}.
 \end{equation} 
	Also for a set of variables $z_0, z_1, \ldots, z_{2k}$ and 
	for $t=1,2, \ldots,2k$, define
	$m_{k,t}(z) =
		 \chi_{[0,1]}\big(z_0+ \sum_{\ell=t}^{2k} (-1)^{\ell} \eta_{\pi}(\ell) z_{\pi'(\ell)}\big).$
	Now using  \eqref{eqn:EHs_ars_epsilon_Hs} in (\ref{eqn:EHs_2k_Hs}) under the above notations, we get 
	\begin{align*}
		\lim_{n \to \infty} \vp_n( \frac{H_{n,s}^{\e_1}}{\sqrt{n}} \cdots \frac{H_{n,s}^{\e_{2k}}}{\sqrt{n}}) 
		= \sum_{\pi\in \mathcal P_{2}(2k)} \int_0^1\int_{[-1,1]^k} \prod_{(r,s)\in \pi} \theta^{(H)}_{r,s}(\e,\rho,z) \prod_{t=1}^{2k} m_{k,t}(z) \prod_{i=0}^k dz_i.
	\end{align*}
	This establishes the limit for one sequence of matrices.
	\end{remark}
The following corollary can be concluded from Remark \ref{cor:com_HJc}.
\begin{corollary} \label{cor:A2+3_HJc}
	\noindent (a) \cite[Proposition 1]{bose_saha_patter_JC_annals} Let $\{H^{(i)}_{n,s}; 1 \leq i \leq m\}$ be independent copies of random symmetric Hankel matrices with \textit{real} input sequence $\{a^{(i)}_{j}\}$ which satisfies Assumption \ref{assump:toeI}. 
	Then $\{n^{-1/2}H^{(i)}_{n,s}; 1 \leq i \leq m\}$ converge in $*$-distribution and the limit $*$-moments are as in (\ref{eqn:lim_mome_A1_HJc}) with $\theta^{(H)}_{r,s}(\e,\rho,z)= \sigma_x^2 (1-\d_{\nu_{r}, \nu_{s}}) + \rho_{2}\d_{\nu_{r}, \nu_{s}}$.
	\vskip3pt
	\noindent (b) A result similar to Corollary \ref{cor:poly_Lsd_T} also holds. 
	In particular,
	if $\{a_{j}\}$ satisfy Assumption \ref{assump:toeI}, then the ESD of real symmetric $n^{-1/2}H_{n,s}$ converges a.s.~to a symmetric probability distribution whose moments are as  in (\ref{eqn:lim_mome_A1_HJc}) with $\theta^{(H)}_{r,s}(\e,\rho,z)= \sigma_x^2 (1-\d_{\nu_{r}, \nu_{s}}) + \rho_{2}\d_{\nu_{r}, \nu_{s}}$. The LSD results of \cite{bryc_lsd_06} and \cite{bose_sen_LSD_EJP} are then special cases via truncation arguments. 
\end{corollary}


\subsection{Final arguments in the proof of Theorem \ref{thm:JC_tdp}
}\label{subsec:jc_tdp}
Having gained an insight through the study of the above special cases, it will be enough to 
explain the main steps in the rest of the proof of Theorem \ref{thm:JC_tdp} when we consider all the matrices together.  
\begin{proof}[Proof of Theorem \ref{thm:JC_tdp}]
Let the input entries of $T^{(\tau)}_n$ and $D^{(\tau)}_n$ be $\{a^{(\tau)}_i\}$ and $\{d^{(\tau)}_i\}$, respectively. It is enough to examine the following monomial:
\begin{align*}
	&(P_n A_{n,\mu_1}^{(\tau_1) \e_1} \cdots A_{n,\mu_{k_1}}^{(\tau_{k_1}) \e_{k_1}} ) (P_n A_{n,\mu_{k_1+1}}^{(\tau_{k_1+1}) \e_{k_1+1}} \cdots A_{n,\mu_{k_2}}^{(\tau_{k_2}) \e_{k_2}} ) \cdots (P_n A_{n,\mu_{k_{p-1}+1}}^{(\tau_{k_{p-1}+1}) \e_{k_{p-1}+1}} \cdots A_{n,\mu_{k_p}}^{(\tau_{k_p}) \e_{k_p}} ) \nonumber \\
	& = q_{k_p}(P,D,T),  \mbox{ say},
\end{align*}
 where $\e_i \in \{1, *\}$, $\tau_i \in \{1,2, \ldots, m\}$ and for $\mu_i \in \{1,2\}$, $A^{(\tau_i) \e_i}_{n,1}=n^{-1/2}T^{(\tau_i) \e_i}_n$, $A^{(\tau_i) \e_i}_{n,2}=D^{(\tau_i) \e_i}_n$.

First note from the proof of Proposition \ref{thm:jc_T+P} that if $p$ is odd, then $ \vp_n \big(q_{k_p}(P,D,T)\big)=o(1)$.
So let $p$ be even. Then from Lemma \ref{lem:hankel}, we have
\begin{align*} 
 \vp_n \big(q_{k_p}(P,D,T)\big)	= \sum_{I_{k_p}} \E\big[\prod_{t=1}^{k_p} z^{(\tau_t)\e_t}_{i_t} \big]  \frac{1}{n^{1+\frac{w_p}{2} }} \sum_{j=1}^n \prod_{e=1}^p m_{\underline k, e}(j,i)   \d_{0,\sum_{c=1}^p(-1)^{p-c} \sum_{t=k_{c-1}+1}^{k_{c}}  \e_t'i_t},
\end{align*}
where $z^{(\tau_t)\e_t}_{i_t}$ is $ a^{(\tau_t)\e_t}_{i_t}$ or $d^{(\tau_t)\e_t}_{i_t}$ depending on whether the matrix $A_{n,\mu_t}^{(\tau_t)\e_t}$ is $T_n^{(\tau_t)\e_t}$ or $D_n^{(\tau_t)\e_t}$; $I_{k_p}$ and $ m_{\underline k, e}$ are as in (\ref{eqn:i_k in -n to n}) and (\ref{eqn:m_chi,t_Hn}), respectively; 
\begin{align}\label{eqn:no of T in Q}
	w_p= \# \{ \mu_t : A_{n,\mu_t}^{(\tau_t) \e_t} = A_{n,1}^{(\tau_t) \e_t} \mbox{ in }  q_{k_p}(P,D,T)\},
\end{align}
which is the number of random Toeplitz matrices in  $q_{k_p}(P,D,T)$.
Observe that if $w_p$ is odd, then 
by using arguments similar to those used in establishing (\ref{eqn:phi_To(1)odd}) and Proposition \ref{thm:jc_T+D}, we get $\vp_n \big(q_{k_p}(P,D,T)\big)=o(1)$.

 Now for $c=1,2, \ldots, p$, let $ u_{r_{c-1}}, u_{r_{c-1}+1}, \ldots, u_{r_{c}}$ be the positions of $D_n$ between $A_{n,\mu_{k_{c-1}+1}}^{(\tau_{k_{c-1}+1}) \e_{k_{c-1}+1}}$ and $A_{n,\mu_{k_{c}}}^{(\tau_{k_c}) \e_{k_c}}$ (including these two also)
 and let $ v_{w_{c-1}}, v_{w_{c-1}+1}, \ldots, v_{w_{c}}$ be the positions of $T_n$ between $A_{n,\mu_{k_{c-1}+1}}^{(\tau_{k_{c-1}+1}) \e_{k_{c-1}+1}}$ and $A_{n,\mu_{k_{c}}}^{(\tau_{k_c}) \e_{k_c}}$. 
 Here $r_0=k_0=w_0=1$ and $r_p+w_p=k_p$. 
 If $q'_{k_p}(P,D,T)$ is the monomial obtained by replacing the matrices $D_n$ by $D_{n, n^\alpha}$ in $q_{k_p}(P,D,T)$, where $D_{n, n^\alpha}$ is as in (\ref{def:D_n,m}), then for $q'_{k_p}(P,D,T)$, by applying arguments analogous to those used while establishing (\ref{eqn:chi_D_remove}), we have
\begin{align*}
	&\frac{1}{n} \sum_{j=1}^n  m_{\underline k, e}(j,i) \\ 
	&= 	\frac{1}{n} \sum_{j=1}^n  \prod_{\ell=k_{e-1}+1}^{k_e} \chi_{[1,n]}\big(j+ (-1)^{p-e} \sum_{t=\ell}^{k_e} \e_t' i_t + \sum_{c=e+1}^{p} (-1)^{p-c}  \sum_{t=k_{c-1}+1}^{k_c} \e_t' i_t \big)  \\
	&= 	\frac{1}{n} \sum_{j=1}^n  \prod_{\ell={w_{e-1}}+1}^{{w_e}} \chi_{[1,n]}\big(j+ (-1)^{p-e} \sum_{t=\ell}^{{w_e}} \e_{v_t}' i_{v_t} + \sum_{c=e+1}^{p} (-1)^{p-c}  \sum_{t={w_{c-1}}+1}^{{w_c}} \e_{v_t}' i_{v_t} \big) +o(1)\\
	&=  \frac{1}{n} \sum_{j=1}^n  m_{\underline w, e} +o(1).
\end{align*}
Therefore, using  arguments similar to those used in the proofs of  Propositions \ref{thm:jc_T+D} and \ref{thm:jc_T+P}, 
\begin{align*} 
& \lim_{n \to \infty} \vp_n \big(q'_{k_p}(P,D,T)\big)  \nonumber\\
 & = 
\sum_{i_{u_{r_0}}, \ldots, i_{u_{r_p}}=-\infty}^{\infty} \prod_{t=1}^{r_p} d^{(\tau_{u_t}) \e_{u_t}}_{i_{u_t}}  \d_{0,\sum_{c=1}^{p}(-1)^c \sum_{t=r_{c-1}+1}^{r_c} \e'_{u_t} i_{u_t}}  \nonumber\\
& \times \lim_{n \to \infty}  \frac{1}{n^{\frac{w_p}{2}}} \sum_{i_{v_{w_0}}, \ldots, i_{v_{w_p}}=-(n-1)}^{n-1} \E\big[ \prod_{t=1}^{w_p} a^{(\tau_{v_t}) \e_{v_t}}_{i_{v_t}} \big]   \frac{1}{n} \sum_{j=1}^n \prod_{e=1}^{p}  m_{\underline w, e}(j,i)  \d_{0,\sum_{c=1}^{p}(-1)^c \sum_{t=w_{c-1}+1}^{w_c} \e'_{v_t} i_{v_t}} \nonumber \\
& = 
\lim_{n \to \infty} \vp_n \big(q_{r_p}(P,D)\big) \times \lim_{n \to \infty} \vp_n \big(q_{w_p}(P,T)\big),
\end{align*}
where 
\begin{align*}
	q_{r_p}(P,D)  = 	\Big( (P_n D_{n}^{(\tau_{u_1})\e_{u_1}} \cdots D_{n}^{(\tau_{u_{r_1}})\e_{u_{r_1}}})   \cdots (P_n D_{n}^{(\tau_{u_{r_{p-1}}+1})\e_{u_{r_{p-1}}+1}} \cdots D_{n}^{(\tau_{u_{r_p}})\e_{u_{r_p}}}) \Big), \\
q_{w_p}(P,T) = 	\Big(  (P_n \frac{T_{n}^{(\tau_{v_1})\e_{v_1}}}{\sqrt{n}} \cdots \frac{T_{n}^{(\tau_{v_{w_1}})\e_{v_{w_1}}}}{\sqrt{n}})   \cdots (P_n \frac{T_{n}^{(\tau_{v_{w_{p-1}}+1})\e_{v_{w_{p-1}}+1}}} {\sqrt{n}} \cdots \frac{T_{n}^{(\tau_{v_{w_p}})\e_{v_{w_p}}}}{\sqrt{n}} ) \Big).
\end{align*}
Here $\lim_{n \to \infty} \vp_n \big(q_{r_p}(P,D)\big)$ and  $\lim_{n \to \infty} \vp_n \big(q_{w_p}(P,T)\big)$ are as in  (\ref{eqn:lim_phi_P,D*}) and
(\ref{eqn:lim_phi(Te*,P)}), respectively.
Finally, from Remark \ref{rem:obse_D,D_m}, we have $\lim_{n \to \infty} \vp_n \big(q_{k_p}(P,D,T)\big)= \lim_{n \to \infty} \vp_n \big(q'_{k_p}(P,D,T)\big)$, and hence 
	\begin{align*}
		\lim_{n \to \infty} \vp_n \big(q_{k_p}(P,D,T)\big) \nonumber 
	&	= 
	\l\{\begin{array}{lll}
		\displaystyle \lim_{n \to \infty} \vp_n \big(q_{r_p}(P,D)\big) \times \lim_{n \to \infty} \vp_n \big(q_{w_p}(P,T)\big) &  \mbox{ if $p, w_p$ are even}, \\
		0&   \mbox{ otherwise},
	\end{array}\r.
\end{align*}
where $r_p=k_p-w_p$ with $w_p$ is as in (\ref{eqn:no of T in Q}).
This completes the proof of Theorem \ref{thm:JC_tdp}.
\end{proof}

\section{Proof of Theorem \ref{thm:gen_tdp_com}}\label{sec:Jc_tdp_g}

As before, we break the proof into several steps. 
In Sections \ref{sec:T_n,g} to 
\ref{subsec:tp_g}, we establish the joint convergence of $T_{n,g}$, 
$\{D_{n,g},P_n\}$, 
$\{D_{n,g},T_{n,g}\}$ and 
$\{T_{n,g},P_n\}$, respectively. Finally, in Section \ref{subsec:Tdp,g} we show how to conclude 
the joint convergence of \texorpdfstring{$T_{n,g}, D_{n,g}$  and $P_n$}{joint} in Theorem \ref{thm:gen_tdp_com}.
In this section, for any given set $A$, $\one_{A}$ denotes the indicator function, that is, 
$\one_{A}(x) =1\;\;\mbox{if $x \in A$  and zero otherwise}$.
\subsection{Joint Convergence of \texorpdfstring{$T_{n,g}$}{tng}}\label{sec:T_n,g}
Recall the  generalized Toeplitz matrix  $T_{n,g}$ from (\ref{def:T_gen}). 
Then we have the following result.
\begin{proposition}\label{thm:generalT_com}
	Suppose $\{T^{(i)}_{n,g}; 1 \leq i \leq m \}$ are $m$ independent copies of generalized Toeplitz matrices and the input sequences $(a_i)_{i\in \Z}, (b_i)_{i\in \Z}$  satisfy Assumption \ref{assump:toe_gII}. Then, for $\e_1,\ldots, \e_{p}\in \{1,*\}$ and $\tau_1,\ldots, \tau_{p}\in \{1,\ldots, m\}$,
	\begin{align*} 
		\lim_{n\to \infty}\vp_n(\dfrac{T^{(\tau_1) \e_1}_{n,g}}{n^{1/2}}\cdots \dfrac{T^{(\tau_p)\e_p}_{n,g}}{n^{1/2}}) \nonumber 
		&  =\l\{\begin{array}{lll}
			\displaystyle \hskip-5pt \sum_{\pi\in \mathcal P_2(2k)}\int_{0}^1\int_{[-1,1]^k}\prod_{(r,s)\in \pi} \hskip-4pt \d_{\tau_r,\tau_s} \mathcal E_{r,s}'({\underline z_{2k}}) \prod_{i=0}^k dz_i & \mbox{ if } p=2k, \\
			0 &  \mbox{ if } p=2k+1,
		\end{array}\r.
	\end{align*}
	where $\mathcal E_{r,s}'({\underline z_{2k}})$ is as given in (\ref{eqn:mathE'(z2k)_T}).
\end{proposition}

Towards the proof of Proposition \ref{thm:generalT_com}, we first derive a trace formula for the product of generalized Toeplitz matrices. 
\begin{lemma}\label{lem:tracegeneralT} Suppose $\{T^{(\tau)}_{n,g}; 1 \leq \tau \leq m\}$ are $m$ copies of generalized Toeplitz matrices with input sequence $(a^{(\tau)}_i)_{i\in \Z}$ and $(b^{(\tau)}_i)_{i\in \Z}$. Then for $\e_1,\ldots, \e_{p}\in \{1,*\}$ and $\tau_1,\ldots, \tau_{p}\in \{1,\ldots, m\}$, 
	\begin{align*}
		\Tr[T_{n,g}^{(\tau_1) \e_1} \cdots T_{n,g}^{(\tau_p) \e_p}] =\sum_{j=1}^n\sum_{I_p'} \prod_{t=1}^p(a^{(\tau_t)\e_t}_{i_t}\one_{\AA_{\e_t'i_t}}+b^{(\tau_t)\e_t}_{i_t}\one_{\BB_{\e_t'i_t}}) (m_t),
	\end{align*}
	where $I'_p$ is as in (\ref{iprime}); $\e'_t$ is as in (\ref{eqn:epsilon'});
 for any $i \in \mathbb{Z}$, $\AA_i, \BB_i$ are  as in (\ref{eqn:AA_i,BB_i}),
and for $t=1,2, \ldots, (p-1),$
\begin{equation}\label{eqn:m_t_toetrace}
	m_t= j+\sum_{\ell=t}^p\e_\ell'i_\ell \mbox{ with } m_p=j.
\end{equation}
\end{lemma}

\begin{proof}
	We prove the lemma for a single matrix $T_{n,g}$. The same argument will also work for several 
 matrices.  
	Note from \eqref{eqn:2} that for a Toeplitz matrix $T_n$
	\begin{align*}
		T_{n}e_j=\sum_{i=-(n-1)}^{(n-1)}t_i\one_{[1,n]}{(j+i)}e_{j+i}=\sum_{i=-(n-1)}^{(n-1)}t_i\one_{A_i}{(j)}e_{j+i},
	\end{align*}
	where $A_i=[\max\{1, 1-i\}, \min\{n-i,n\}]$. Which means $t_i$, for  $-(n-1)\le i\le (n-1)$, appears in the $j$-th, $1\le j\le n$, column iff $i+j\in [1,n]$.  
	
Now for $T_{n,g}$, observe from (\ref{def:T_gen}) that $a_{i}$ appears in the $j$-th column if
	\begin{align*}
	\max\{1, 1-i\}\le j\le {\frac{1}{2}(n-i)},
	\end{align*}
	and $b_i$ appears in the $j$-th column if
$	{\frac{1}{2}(n-i)}< j\le \min\{n-i,n\}.$
 If we define intervals 
	\begin{equation}\label{eqn:AA_i,BB_i}
	 \mathcal A_i=[\max\{1,1-i\}, \ \frac{1}{2}(n-i)],  \ \mathcal B_i=(\frac{1}{2}(n-i),\  \min\{n-i,n\}],
\end{equation}
 then we have
	\begin{align*}
		T_{n,g}e_j=\sum_{i=-(n-1)}^{n-1}(a_i\one_{\AA_{i}}+b_i\one_{\BB_i})(j)e_{j+i}.
	\end{align*}
	Similarly, we have
	$	T_{n,g}^*e_j=\sum_{i}(a^*_i\one_{\AA_{-i}}+b^*_i\one_{\BB_{-i}})(j)e_{j-i}.$
	Recall that  $\e'=1$ if $\e=1$ and $\e'=-1$ if $\e=*$. Therefore, for $\e\in \{1,*\}$, we can write
	\begin{align*} 
		T_{n,g}^{\e}e_j=\sum_{i}(a^{\e}_i\one_{\AA_{\e'i}}+b^{\e}_i\one_{\BB_{\e'i}})(j)e_{j+\e'i}.
	\end{align*}
	Thus
	\begin{align*}
	(T_{n,g}^{\e_{p-1}}T_{n,g}^{\e_p})e_j &=\sum_{i_p}(a^{\e_p}_{i_p}\one_{\AA_{\e_p'i_p}}+b^{\e_p}_{i_p}\one_{\BB_{\e_p'i_p}})(j)T_{n,g}^{\e_{p-1}}e_{j+\e_p'i_p}
		\\&=\sum_{i_p, i_{p-1}}\prod_{t=p-1}^p (a^{\e_t}_{i_t}\one_{\AA_{\e_t'i_t}}+b^{\e_t}_{i_t}\one_{\BB_{\e_t'i_t}}) (j+\sum_{\ell=t}^p\e_\ell'i_\ell)e_{j+\sum_{\ell=p-1}^p\e_\ell'i_\ell},
	\end{align*}
	with  $j+\sum_{\ell=p}^p\e_\ell'i_\ell = j$. Continuing the process, for $\e_1,\ldots, \e_p\in \{1,*\}$, we get 
	\begin{align*}
	(T_{n,g}^{\e_1} \cdots T_{n,g}^{\e_p})e_j =\sum_{I_p} \prod_{t=1}^p(a^{\e_t}_{i_t}\one_{\AA_{\e_t'i_t}}+b^{\e_t}_{i_t}\one_{\BB_{\e_t'i_t}})(j+\sum_{\ell=t}^p\e_\ell'i_\ell)e_{j+\sum_{\ell=1}^p\e_\ell'i_\ell},
	\end{align*}
	where $I_p=\{(i_1,\ldots, i_p)\; :\; -(n-1)\le i_1,\ldots, i_p\le n-1\}$. Hence we have 
	\begin{align*}
		\Tr(T_{n,g}^{\e_1} \cdots T_{n,g}^{\e_p}) 
		=\sum_{j=1}^n e_j^t (T_{n,g}^{\e_1} \cdots T_{n,g}^{\e_p}) e_j =\sum_{j=1}^n\sum_{I_p'}\prod_{t=1}^p(a^{\e_t}_{i_t}\one_{\AA_{\e_t'i_t}}+b^{\e_t}_{i_t}\one_{\BB_{\e_t'i_t}})(m_t),
	\end{align*}
	where $I'_p$ is as in (\ref{iprime}) and $m_t$ is as in (\ref{eqn:m_t_toetrace}).	This completes the proof.
\end{proof}
Now using the above trace formula, we prove Proposition \ref{thm:generalT_com}.
\begin{proof}[Proof of Proposition \ref{thm:generalT_com}]
	For simplicity of notation, we first consider a single matrix.
	By Lemma \ref{lem:tracegeneralT} we have
	\begin{align*}
		\vp_n( \frac{T_{n,g}^{\e_1}}{\sqrt{n}} \cdots \frac{T_{n,g}^{\e_p}}{\sqrt{n}}) &=\frac{1}{n^{\frac{p}{2}+1}}\sum_{j=1}^n\sum_{I_{p}'}\E[\prod_{t=1}^{p}(a^{\e_t}_{i_t}\one_{\AA_{\e_t'i_t}}(m_t)+b^{\e_t}_{i_t}\one_{\BB_{\e_t'i_t}}(m_t))],
	\end{align*}
	where $I'_p$ is as in (\ref{iprime}) and $m_t$ is as in (\ref{eqn:m_t_toetrace}).
	By arguments similar to those used in the proof of Proposition \ref{thm:toeplitz}, we conclude that only pair-partitions will contribute to the limit. In particular, the limit is zero when $p$ is odd. 
	
	Let $p=2k$.   Then we have 
	\begin{align} \label{eqn:phi_T+n,g}
		\vp_n( \frac{T_{n,g}^{\e_1}}{\sqrt{n}} \cdots \frac{T_{n,g}^{\e_{2k}}}{\sqrt{n}})
		&=\sum_{\pi\in \mathcal P_2(2k)}\frac{1}{n^{k+1}}\sum_{j=1}^n\sum_{I_{2k}'(\pi)}\prod_{(r,s)\in \pi} \mathcal E_{r,s}(\underline i_{2k})+o(1),
	\end{align}
	where $\underline i_{2k}=(j,i_1,\ldots, i_{2k})$, $I_{2k}'(\pi)=\{(i_1,\ldots, i_{2k})\in I_{2k}' \; :\; \prod_{(r,s)\in \pi} \mathcal E_{r,s}(\underline i_{2k}) \neq 0\}$ and
	\begin{align*}
		\mathcal E_{r,s}(\underline i_{2k})=\E\l[\l(a^{\e_r}_{i_{r}}\one_{\AA_{\e_r'i_r}}(m_r)+b^{\e_r}_{i_r}\one_{\BB_{\e_r'i_r}}(m_r)\r)\l(a^{\e_s}_{i_{s}}\one_{\AA_{\e_s'i_s}}(m_s)+b^{\e_s}_{i_s}\one_{\BB_{\e_s'i_s}}(m_s)\r)\r].
	\end{align*}
	Note that, we have the constraint $\sum_{(r,s)\in \pi}(\e_r'i_r+\e_s'i_s)=0$ on the indices. Now using arguments similar to those used in the proof of Proposition \ref{thm:toeplitz}, we have a non-zero contribution if and only if $(\e_r'i_r+\e_s'i_s)=0$ for each pair $(r,s)\in \pi$, which is equivalent to the following:
	\begin{align*}
		i_{r}=
		\l\{\begin{array}{lll}
			i_{s} & \mbox{ if } & \e_{r}'\e_{s}'=-1,\\ 
			-i_{s} & \mbox{ if} & \e_{r}'\e'_{s}=1.
		\end{array}
		\r.
	\end{align*}
	Since the input entries $\{(a_{j}= x_{j} + \mathrm{i} y_{j}, b_{j}= x'_{j}+ \mathrm{i} y'_{j}); j \in \mathbb{Z}\}$  are independent, we have a loss of a degree of freedom if $i_r=-i_s$. Thus a non-zero contribution is only possible when $i_r=i_s$, that is, $\e_{r}'\e_{s}'=-1$. Also, the input entries satisfy Assumption \ref{assump:toe_gII}, we have
	\begin{align*}
	\E[a^{\e_r}_{i_r}a^{\e_s}_{i_s}] &= 1-\d_{\e_r,\e_s},\;  \E[b^{\e_r}_{i_r}b^{\e_s}_{i_s}]=1-\d_{\e_r,\e_s}, \ 		 \E[a^{\e_r}_{i_r}b^{\e_s}_{i_s}] = (1-\d_{\e_r,\e_s}) [\rho_2+\rho_6 - \mathrm{i} \e'_r(\rho_3-\rho_4)], \\
  \E[a^{\e_s}_{i_s}b^{\e_r}_{i_r}] &= (1-\d_{\e_r,\e_s}) [\rho_2+\rho_6 - \mathrm{i} \e'_s(\rho_3-\rho_4)].
	\end{align*}
	Therefore 
	\begin{align*}
		\mathcal E_{r,s}(\underline i_{2k}) =(1-\d_{\e_r,\e_s}) \big[(f_1+f_4) +  [\rho_2+\rho_6 - \mathrm{i}\e'_r(\rho_3-\rho_4)]f_2 +  [\rho_2+\rho_6 - \mathrm{i}\e'_s(\rho_3-\rho_4)]f_3 \big],
	\end{align*}
	where 
	\begin{align*}
		&f_1=\one_{\AA_{\e_r'i_r}}(m_r)\one_{\AA_{\e_s'i_s}}(m_s),\; f_2=\one_{\AA_{\e_r'i_r}}(m_r)\one_{\BB_{\e_s'i_s}}(m_s),\\&
		f_3=\one_{\AA_{\e_s'i_s}}(m_s)\one_{\BB_{\e_r'i_r}}(m_r),\; 
		f_4=\one_{\BB_{\e_r'i_r}}(m_r)\one_{\BB_{\e_s'i_s}}(m_s).
	\end{align*}
	Observe that, $\mathcal E_{r,s}(\underline i_{2k})$ implies that the number of free indices among $\{i_1,\ldots, i_{2k}\}$ is $k$, as the indices satisfy the relation $i_r=i_s$.
	Let $m_t'=\frac{1}{n}(j+\sum_{\ell=t}^{2k}\e_{\ell}'i_\ell)$. Then we have 
	\begin{align*}
		&f_1=\one_{\AA_{\e_r'\frac{i_r}{n}}}(m_r')\one_{\AA_{\e_s'\frac{i_s}{n}}}(m_s'),\; f_2=\one_{\AA_{\e_r'\frac{i_r}{n}}}(m_r')\one_{\BB_{\e_s'\frac{i_s}{n}}}(m_s'),\\&
		f_3=\one_{\AA_{\e_s'\frac{i_s}{n}}}(m_s')\one_{\BB_{\e_r'\frac{i_r}{n}}}(m_r'),\; 
		f_4=\one_{\BB_{\e_r'\frac{i_r}{n}}}(m_r')\one_{\BB_{\e_s'\frac{i_s}{n}}}(m_s').
	\end{align*}
	Also for a variable $z$, let 
	\begin{equation} \label{eqn:A_x,B_x}
		 \tilde \AA_z=[\max\{0,-z\}, \frac{1}{2}(1-z)], \ \tilde\BB_z=[\frac{1}{2}(1-z), \min\{1-z,1\}].
	\end{equation}
	 Recall, if $\pi=(r_1,s_1)\cdots (r_k,s_k)$, then we have $\pi'(r_t)=\pi'(s_t)=t$. For a set of variables $z_0, z_1, \ldots, z_{2k}$, we define $w_t=z_0+\sum_{\ell=t}^{2k}\e_\ell' z_{\pi'(\ell)}$ and
	\begin{align*} 
		&f_1'=\one_{\tilde\AA_{\e_r'z_{\pi'(r)}}}(w_r)\one_{\tilde\AA_{\e_s'z_{\pi'(s)}}}(w_s), \; f_2'=\one_{\tilde\AA_{\e_r'z_{\pi'(r)}}}(w_r)\one_{\tilde\BB_{\e_s'z_{\pi'(s)}}}(w_s),\\
		&	f_3'=\one_{\tilde\BB_{\e_r'z_{\pi'(r)}}}(w_r)\one_{\tilde\AA_{\e_s'z_{\pi'(s)}}}(w_s), \; 
		f_4'=\one_{\tilde\BB_{\e_r'z_{\pi'(r)}}}(w_r)\one_{\tilde\BB_{\e_s'z_{\pi'(s)}}}(w_s). \nonumber
	\end{align*}
	Then using all the above notations, and using convergence of  Riemann sums in (\ref{eqn:phi_T+n,g}), we get
	\begin{align*}
		\lim_{n \to \infty}\vp_n( \frac{T_{n,g}^{\e_1}}{\sqrt{n}} \cdots \frac{T_{n,g}^{\e_{2k}}}{\sqrt{n}}) &=\int_{z_0=0}^1 \sum_{\pi\in \mathcal P_2(2k)} \int_{[-1,1]^k}\prod_{(r,s)\in \pi} \mathcal E_{r,s}'({\underline z_{2k}})dz_0 dz_1 \cdots dz_k,
	\end{align*}
	where $\underline z_{2k}=(z_0,z_1,\ldots,z_{2k})$ and 
	\begin{align}\label{eqn:mathE'(z2k)_T}
		\mathcal E_{r,s}'(\underline z_{2k}) = (1-\d_{\e_r,\e_s}) \big[(f'_1+f'_4) +  [\rho_2+\rho_6 - \mathrm{i}\e'_r(\rho_3-\rho_4)]f'_2 +  [\rho_2+\rho_6 - \mathrm{i}\e'_s(\rho_3-\rho_4)]f'_3 \big].
	\end{align}

Now we specialize to the convergence for $m$ independent matrices. Note from Lemma \ref{lem:tracegeneralT} that
	\begin{align*}
\vp_n(\dfrac{T^{(\tau_1) \e_1}_{n,g}}{n^{1/2}}\cdots \dfrac{T^{(\tau_p)\e_p}_{n,g}}{n^{1/2}}) &=\frac{1}{n^{\frac{p}{2}+1}}\sum_{j=1}^n\sum_{I_{p}'}\E[\prod_{t=1}^{p}(a^{(\tau_t)\e_t}_{i_t}\one_{\AA_{\e_t'i_t}}(m_t) +b^{(\tau_t)\e_t}_{i_t}\one_{\BB_{\e_t'i_t}}(m_t))].
\end{align*}
Again, if $p$ is odd, then the limit will be zero, and for even $p$, say $2k$, we have
$$\lim_{n\to \infty}\vp_n(\dfrac{T^{(\tau_1) \e_1}_{n,g}}{n^{1/2}}\cdots \dfrac{T^{(\tau_{2k})\e_{2k}}_{n,g}}{n^{1/2}}) = \int_{z_0=0}^1 \sum_{\pi\in \mathcal P_2(2k)} \int_{[-1,1]^k}\prod_{(r,s)\in \pi} \d_{\tau_r,\tau_s} \mathcal E_{r,s}'({\underline z_{2k}}) \prod_{i=0}^k dz_i,$$
where $\mathcal E_{r,s}'({\underline z_{2k}})$ is as in (\ref{eqn:mathE'(z2k)_T}). This completes the proof.
\end{proof}

\subsection{Joint Convergence of \texorpdfstring{$D_{n,g}$ and  $P_{n}$}{dngpn}} \label{subsec:dp_g}
\begin{proposition} \label{thm:jc_D+P_g}
	Suppose $\{D_{n,g}^{(\tau)}; 1\leq \tau \leq m\}$ are $m$ deterministic generalized Toeplitz matrices with input sequences $(d^{'(\tau)}_i)_{i\in \Z}, (d^{''(\tau)}_i)_{i\in \Z}$ which satisfy Assumption \ref{assump:determin}. Then $\{P_n, D_{n,g}^{(\tau)}; 1 \leq \tau \leq m\}$ converge jointly. The limit $*$-moments are as given in (\ref{eqn:lim_phi_P,D*_g}).
\end{proposition}

First, we establish 
a trace formula for a monomial in $P_n$ and generalized Toeplitz matrices, similar to Lemma \ref{lem:hankel}. 
\begin{lemma}  \label{lem:hankel_g}
	For $\tau=1,2, \ldots, m$, let $M^{(\tau)}_{n,g}$ be generalized Toeplitz matrices (random or non-random) with input sequences $(a^{(\tau)}_i)_{i\in \Z}, (b^{(\tau)}_i)_{i\in \Z}$. Then for $\e_i\in \{1,*\}$, $\tau_i\in \{1, \ldots,m\}$ and $0\leq k_1 \leq k_2 \leq \cdots \leq k_p$,  we have
	\begin{align*}
		&\Tr\big[(P_nM_{n,g}^{(\tau_1)\e_1} \cdots M_{n,g}^{(\tau_{k_1}) \e_{k_1}}) 
		 \cdots (P_nM_{n,g}^{(\tau_{k_{p-1}+1})\e_{k_{p-1}+1}} \cdots M_{n,g}^{(\tau_{k_p})\e_{k_p}})\big] \nonumber \\
		& =
\l\{\begin{array}{l}
\displaystyle{\sum_{j=1}^n \sum_{I_{k_p}} \prod_{e=1}^p \prod_{t=k_{e-1}+1}^{k_e} (a^{(\tau_t)\e_t}_{i_t}\one_{\AA_{\e_t'i_t}}+b^{(\tau_t)\e_t}_{i_t}\one_{\BB_{\e_t'i_t}})  (m_{t,k_e})   \d_{0,\sum_{c=1}^p(-1)^{p-c} \sum_{\ell=k_{c-1}+1}^{k_{c}}  \e_\ell'i_\ell}} \\  \hfill{\mbox{ if $p$ is even,}}
			\\
	\displaystyle{\hskip-5pt \sum_{j=1}^n \sum_{I_{k_p}} \prod_{e=1}^p \prod_{t=k_{e-1}+1}^{k_e} (a^{(\tau_t)\e_t}_{i_t}\one_{\AA_{\e_t'i_t}}+b^{(\tau_t)\e_t}_{i_t}\one_{\BB_{\e_t'i_t}})   (m_{t,k_e})   \d_{2j-1-n,\sum_{c=1}^p(-1)^{p-c} \sum_{\ell=k_{c-1}+1}^{k_{c}}  \e_\ell'i_\ell}} \\
  \hfill\mbox{ if $p$ is odd,}
\end{array}	\r.
	\end{align*}
	where $I_{k_p}$ is as in (\ref{eqn:i_k in -n to n})  
	and for $e=1,2, \ldots,p$, 
	\begin{align} \label{eqn:m_chi,t_Hn_g}
		m_{t,k_e} =j+ (-1)^{p-e} \sum_{\ell=t}^{k_e} \e'_\ell i_\ell + \sum_{c=e+1}^{p} (-1)^{p-c}  \sum_{\ell=k_{c-1}+1}^{k_c} \e'_\ell i_\ell.
	\end{align}
\end{lemma}
We skip the proof of this lemma, but note that 
a proof can be fashioned out of the ideas 
used in the proofs of Lemmas \ref{lem:hankel} and \ref{lem:tracegeneralT}.
\begin{proof} [Proof of Proposition \ref{thm:jc_D+P_g}]
We skip the detailed proof. But we briefly justify the 
value of the 
limits. 
Note also the ideas used in the proof of Proposition \ref{thm:jc_D+P}. From arguments similar to those used in the proof of  (\ref{eqn:lim_D*_n^p}), using the trace formula from Lemma \ref{lem:tracegeneralT}, we can show that for $\e_i \in \{1,*\}$ and $\tau_i \in \{1, \ldots,m\}$,
	\begin{align} \label{eqn:lim_D*_n^p_g}
		&\lim_{n \to \infty}  \vp_n( D^{(\tau_1)\e_1}_{n,g} \cdots D^{(\tau_p)\e_p}_{n,g}) \nonumber\\
		&=  \sum_{i_1, \ldots, i_{p}=-\infty }^{\infty} \int_{z_0=0}^{1} \prod_{t=1}^p (d^{'(\tau_t)\e_t}_{i_t}\one_{[0,1/2]}(z_0)+d^{''(\tau_t)\e_t}_{i_t}\one_{(1/2,1]} (z_0))  \d_{0,\sum_{\ell=1}^p \e'_\ell i_\ell} dz_0 \nonumber\\
		&= \sum_{i_1, \ldots, i_{p}=-\infty }^{\infty} \prod_{t=1}^p (\frac{1}{2} d^{'(\tau_t)\e_t}_{i_t}+ \frac{1}{2} d^{''(\tau_t)\e_t}_{i_t})  \d_{0,\sum_{\ell=1}^p \e'_\ell i_\ell}.
	\end{align}	
	
 Now from arguments similar to those used in establishing (\ref{eqn:lim_phi_P,D*}), for the collection $\{P_n, D_{n,g}^{(\tau)}; 1 \leq \tau \leq m\}$, using Lemma \ref{lem:hankel_g}, we have
\begin{align}\label{eqn:lim_phi_P,D*_g}
		&	\lim_{n \to \infty}  \vp_n \big[(P_nD_{n,g}^{(\tau_1)\e_1} \cdots D_{n,g}^{(\tau_{k_1})\e_{k_1}})  \cdots (P_nD_{n,g}^{(\tau_{k_{p-1}+1})\e_{k_{p-1}+1}} \cdots D_{n,g}^{(\tau_{k_p})\e_{k_p}})\big] \nonumber \\
		&	= 
		\l\{\begin{array}{lll}
			\displaystyle \sum_{i_1, \ldots, i_{k_p}=-\infty }^{\infty}  \prod_{t=1}^{k_p} (\frac{1}{2} d^{'(\tau_t)\e_t}_{i_t}+ \frac{1}{2} d^{''(\tau_t)\e_t}_{i_t}) \d_{0, \sum_{c=1}^p(-1)^{p-c} \sum_{\ell=k_{c-1}+1}^{k_{c}}  \e'_\ell i_\ell} &  \mbox{ if $p$ is even}, \\
			0&   \mbox{ if $p$ is odd}.
		\end{array}\r.
	\end{align}
	This completes the proof of Proposition \ref{thm:jc_D+P_g}.
\end{proof}

\subsection{Joint Convergence of \texorpdfstring{$T_{n,g}$ and $D_{n,g}$}{tngdng}}\label{subsec:td_g}
\begin{proposition} \label{thm:jc_T+D_g}
Let $B_{n,1}=n^{-1/2}T_{n,g}$ where $T_{n,g}$  is a  random generalized Toeplitz matrix. For $i=1,2, \ldots, m$, let $\{B_{n,1}^{(i)}\}$ be independent copies of $B_{n,1}$ and  $B_{n,2}^{(i)}=D^{(i)}_{n,g}$ be $m$ deterministic generalized Toeplitz matrices. Suppose the input entries $(a_j)_{j\in \Z}, (b_j)_{j\in \Z}$  of $T_{n,g}$  satisfy Assumption \ref{assump:toe_gII} and the input entries $(d'_j)_{j\in \Z}, (d^{''}_j)_{j\in \Z}$ of $D_{n,g}$  satisfy Assumption \ref{assump:determin}. 
 Then $\{B_{n,j}^{(i)}; 1\leq i\leq m, 1\leq j \leq 2\}$ converge jointly.
	The limit $*$-moments are as given in (\ref{eqn:lim_ph(Dm*,T*)_g}).
\end{proposition}

\begin{proof}
Here again we outline the justification for the limit moments. For more details, see the proof of Proposition \ref{thm:jc_T+D}.	First note that the limit of $ \vp_n(D_{n, g}^p (n^{-1/2}T_{n,g})^q)$ will be zero if $q$ is odd. Let $q=2k$, then from the idea of the proof of (\ref{eqn:lim_ph(Dm,T)})  and  the proof of Proposition \ref{thm:jc_D+P_g}, we have 
	\begin{align*} 
		 \lim_{n \to \infty} \vp_n(D_{n, g}^p (n^{-1/2}T_{n,g})^{2k}) 
		&= \int_{z_0=0}^{1} \sum_{i_1, \ldots, i_{p}=-\infty }^{\infty} \prod_{t=1}^p (d^{'(\tau_t)\e_t}_{i_t}\one_{[0,1/2]}(z_0)+d^{''(\tau_t)\e_t}_{i_t}\one_{(1/2,1]} (z_0))   \nonumber \\ 
		&  \qquad \times \d_{0,\sum_{\ell=1}^p \e'_\ell i_\ell} \sum_{\pi\in \mathcal P_2(2k)} \int_{[-1,1]^k}\prod_{(r,s)\in \pi} \mathcal E_{r,s}'({\underline z_{2k}}) \prod_{i=0}^k dz_i,
	\end{align*}
	where the existence of the first limit is given in (\ref{eqn:lim_D*_n^p_g}) and $\mathcal E_{r,s}'({\underline z_{2k}})$ is as in (\ref{eqn:mathE'(z2k)_T}) with $\e_1 = \cdots= \e_{2k}=1$.
	
	For an arbitrary monomial $B_{n,\mu_1}^{(\tau_1) \e_1} B_{n,\mu_2}^{(\tau_2) \e_2} \cdots B_{n,\mu_q}^{(\tau_q) \e_q}$ with  $B^{(\tau_i) \e_i}_{n,1}=n^{-1/2}T^{(\tau_i) \e_i}_{n,g}$ and $B^{(\tau_i) \e_i}_{n,2}=D^{(\tau_i) \e_i}_{n,g}$, let $\mathcal{I}_p= (v_1, v_2, \ldots, v_p)$ be the indices corresponding to the positions of $D_{n,g}$ in the monomial and $R= [q] \setminus \mathcal{I}_p$. Note that if the cardinality of $R$ is odd, then the limit will be zero. Let $\# R$ be even, say $2k$, then similar to the above expression, we can show that
	\begin{align} \label{eqn:lim_ph(Dm*,T*)_g}
		\lim_{n \to \infty} \vp_n \big( B_{n,\mu_1}^{(\tau_1) \e_1} B_{n,\mu_2}^{(\tau_2) \e_2} \cdots B_{n,\mu_q}^{(\tau_q) \e_q} \big)   \nonumber 
		& = \hskip-5pt \int_{z_0=0}^{1} \sum_{i_{v_1}, \ldots, i_{v_{p}}=-\infty }^{\infty} \prod_{t\in \mathcal{I}_p}(d^{'(\tau_t)\e_t}_{i_t}\one_{[0,1/2]}(z_0)+d^{''(\tau_t)\e_t}_{i_t}\one_{(1/2,1]} (z_0))  \nonumber \\ 
		& \quad \times  \d_{0,\sum_{\ell=1}^p \e'_{v_\ell} i_{v_\ell}}  \sum_{\pi \in \mathcal P_2(R)} \int_{[-1,1]^k}\prod_{(r,s)\in \pi} \mathcal E_{r,s}'({\underline z_{2k}}) \prod_{i=0}^k dz_i,
	\end{align}
	where $\mathcal P_2(R)$ denotes the set of all pair-partitions of set $R$ and $\mathcal E_{r,s}'({\underline z_{2k}})$ is as in (\ref{eqn:mathE'(z2k)_T}) for the pair-partition of set $R$. This completes the proof of the proposition.
\end{proof}

\subsection{Joint Convergence of \texorpdfstring{$T_{n,g}$ and $P_n$}{tngpn}}\label{subsec:tp_g}
\begin{proposition} \label{thm:jc_T+P_g}
	Suppose $\{T_{n,g}^{(i)}; 1\leq i\leq m\}$ are $m$ independent copies of generalized Toeplitz matrices whose input entries satisfy Assumption \ref{assump:toe_gII}. Then $\{P_n, n^{-1/2}T_{n,g}^{(i)}; 1 \leq i \leq m\}$ converge jointly with the limit $*$-moments as given in (\ref{eqn:lim_phi(Te*,P)_g}).
\end{proposition}
	Note that  from the collection $\{P_n, T^{(i)}_{n,g}; 1 \leq i \leq  m \}$, it is enough to check the convergence for  monomial $(P_nT_{n,g}^{(\tau_1)\e_1} \cdots T_{n,g}^{(\tau_{k_1})\e_{k_1}})  \cdots (P_nT_{n,g}^{(\tau_{k_{p-1}+1})\e_{k_{p-1}+1}} \cdots T_{n,g}^{(\tau_{k_p})\e_{k_p}})$. 	Recall that $P_n T_{n,g}=H_n$. For simplicity, we first provide the arguments 
for the 
monomial
$H_{n}^{(\tau_1)\e_1} \cdots H_{n}^{(\tau_{p})\e_{p}}$.
For the general case, similar arguments will work. Also see the proof of Proposition \ref{thm:jc_T+P}.
Now under Assumption \ref{assump:toe_gII}, the following remark provides the joint convergence of Hankel matrices. 
\begin{remark}\label{cor:jc_H_g}
	Let $\{H^{(i)}_{n}; 1 \leq i \leq m\}$ be $m$ independent copies of Hankel matrices with the input sequences $(a^{(i)}_j)_{j\in \Z}, (b^{(i)}_j)_{j\in \Z}$ which  satisfy Assumption \ref{assump:toe_gII}. Then, for $\e_1,\ldots, \e_{p}\in \{1,*\}$, and $\tau_1,\ldots, \tau_{p}\in \{1,\ldots, m\}$,
	\begin{align} \label{eqn:lim_mome_A3_Hgen_com}
		 \lim_{n\to \infty}\vp_n(\dfrac{H^{(\tau_1) \e_1}_{n}}{n^{1/2}}\cdots \dfrac{H^{(\tau_p)\e_p}_{n}}{n^{1/2}}) 
		& =\l\{\begin{array}{lll}
			\displaystyle \hskip-5pt \sum_{\pi\in \mathcal P_2(2k)}\int_{0}^1\int_{[-1,1]^k}\prod_{(r,s)\in \pi} \hskip-5pt \d_{\tau_r,\tau_s} \mathcal E_{r,s}^{(H)}({\underline z_{2k}}) \prod_{i=0}^k dz_i &\mbox{ if } p=2k, \\
			0 & \hskip-10pt \mbox{ if } p=2k+1,
		\end{array}\r.
	\end{align}
	where $\mathcal E_{r,s}^{(H)}({\underline z_{2k}})$ is as in (\ref{eqn:mathE'(z2k)_H}).

 We need the following trace formula for the proof of (\ref{eqn:lim_mome_A3_Hgen_com}): Suppose $\{H^{(\tau)}_{n}; 1 \leq \tau \leq m \}$ are the Hankel matrices with input sequences $(a^{(\tau)}_i)_{i\in \Z}$ and $(b^{(\tau)}_i)_{i\in \Z}$. Then we have
	\begin{align*} 
		&\Tr[H_{n}^{(\tau_1) \e_1}\cdots H_{n}^{(\tau_p) \e_p}] \nonumber \\
	&= \l\{\begin{array}{lll}
	\displaystyle \sum_{j=1}^n \prod_{t=1}^p (a^{(\tau_t) \e_t}_{i_t}\one_{\AA_{\e_t'i_t}} +b^{(\tau_t) \e_t}_{i_t}\one_{\BB_{\e_t'i_t}})(m_t) \d_{0,\sum_{\ell=1}^p i_\ell (-1)^{p-\ell}}  & \mbox{ if $p$ is even}, \\
	\displaystyle	\sum_{j=1}^n \prod_{t=1}^p(a^{(\tau_t) \e_t}_{i_t} \one_{\AA_{\e_t'i_t}} +b^{(\tau_t) \e_t}_{i_t}\one_{\BB_{\e_t'i_t}})(m_t) \d_{2j-1-n,\sum_{\ell=1}^pi_\ell (-1)^{p-\ell}} & \mbox{ if $p$ is odd},
		\end{array}	\r.
	\end{align*}
	where for any $i \in \mathbb{Z}$, $\AA_i, \BB_i$ are as in (\ref{eqn:AA_i,BB_i});
	for $t=1,2, \ldots,(p-1)$,
	\begin{align} \label{eqn:m_t_Hn}
       m_t =
	 j+\sum_{\ell=t}^p  i_\ell (-1)^{p-\ell} \mbox{ with $m_p=j$}.
	\end{align}
	We skip the proof of this trace formula and refer to the proof of Lemma \ref{lem:tracegeneralT}. 
	
 Now we prove (\ref{eqn:lim_mome_A3_Hgen_com}).	For simplicity of notation, we first consider a single matrix. The same idea will work for $m$ matrices also.	By arguments similar to those used in the proof of Proposition \ref{thm:toeplitz}, we conclude that the limit of odd $*$-moments will be zero. Let $p=2k$.  Then using the above trace formula, we have
	\begin{align*}
		\vp_n( \frac{H_{n}^{\e_1}}{\sqrt{n}} \cdots \frac{H_{n}^{\e_{2k}}}{\sqrt{n}}) &=\frac{1}{n^{k+1}}\sum_{j=1}^n\sum_{I_{2k}^{''}}\E[\prod_{t=1}^{2k}(a^{\e_t}_{i_t}\one_{\AA_{\e_t'i_t}}(m_t)+b^{\e_t}_{i_t}\one_{\BB_{\e_t'i_t}}(m_t))],
	\end{align*}
where $I_{2k}^{''}$ is as in (\ref{eqn:I'2k_Hs}) and $m_t$ is as in (\ref{eqn:m_t_Hn}).
Note that, in the above expression, only the pair-partitions will contribute in the limit, and therefore
	\begin{align*}
	\vp_n( \frac{H_{n}^{\e_1}}{\sqrt{n}} \cdots \frac{H_{n}^{\e_{2k}}}{\sqrt{n}})
		&=\sum_{\pi\in \mathcal P_2(2k)}\frac{1}{n^{k+1}} \sum_{j=1}^n\sum_{I_{2k}^{''}(\pi)}\prod_{(r,s)\in \pi} \mathcal E_{r,s}(\underline i_{2k})+o(1),
	\end{align*}
	where $I_{2k}^{''}(\pi)=\{(i_1,\ldots, i_{2k})\in I_{2k}^{''} \; :\; \prod_{(r,s)\in \pi}\E[a^{\e_r}_{i_r} a^{\e_s}_{i_s}]\neq 0\}$ and for $\underline i_{2k}=(j,i_1,\ldots, i_{2k})$  
	\begin{align*}
		\mathcal E_{r,s}(\underline i_{2k})=\E\l[\l(a^{\e_r}_{i_{r}}\one_{\AA_{\e_r'i_r}}(m_r)+b^{\e_r}_{i_r}\one_{\BB_{\e_r'i_r}}(m_r)\r)\l(a^{\e_s}_{i_{s}}\one_{\AA_{\e_s'i_s}}(m_s)+b^{\e_s}_{i_s}\one_{\BB_{\e_s'i_s}}(m_s)\r)\r].
	\end{align*}
Observe that, the constraint on the indices, $\sum_{\ell=1}^{2k} i_\ell (-1)^{2k-\ell}$ can also be written as  $\sum_{\ell=1}^{2k}\nu_\ell i_\ell$, where $\nu_\ell$ is as in (\ref{eqn:nu_t_oddeven}).
Thus, in terms of a pair-partition, the constraint on the indices will be $\sum_{(r,s)\in \pi}(\nu_ri_r+\nu_si_s)=0$. Now 
using arguments similar to those used in the proof of Proposition \ref{thm:toeplitz} and Remark \ref{cor:com_HJc}, we have a non-zero contribution in the limit if and only if $(\nu_ri_r+\nu_si_s)=0$ for each pair $(r,s)\in \pi$, which is equivalent to the following
	\begin{align*}
		i_{r}=
		\l\{\begin{array}{lll}
			i_{s} & \mbox{ if } & \nu_{r} \nu_{s}=-1,\\ 
			-i_{s} & \mbox{ if} & \nu_{r} \nu_{s}=1.
		\end{array}
		\r.
	\end{align*}
Note that $\{(a_{j}= x_{j} + \mathrm{i} y_{j}, b_{j}= x'_{j}+ \mathrm{i} y'_{j}); j \in \mathbb{Z}\}$  are independent and satisfy Assumption \ref{assump:toe_gII}, therefore
	\begin{align*}
		\E[a^{\e_r}_{i_r}a^{\e_s}_{i_s}] 
		&=  	\l\{\begin{array}{lll}
			(1-\d_{\nu_{r}, \nu_{s}}) & \mbox{ if } & \e_r \neq \e_s,\\ 
			2i\rho_1(1-\d_{\nu_{r}, \nu_{s}}) & \mbox{ if} & \e_r = \e_s=1, \\
	    	-2i\rho_1(1-\d_{\nu_{r}, \nu_{s}}) & \mbox{ if} & \e_r = \e_s=*, \\
		\end{array}
		\r. \\
	&	= (1-\d_{\nu_{r}, \nu_{s}}) \big[\e'_{r} 2i\rho_1 \big]^{\d_{\e'_{r}, \e'_{s} }}, 
		\end{align*}
	where $\e'=1$ if $\e=1$ and $\e'=-1$ if $\e=*$. Similarly,
		\begin{align*}
		\E[b^{\e_r}_{i_r}b^{\e_s}_{i_s}] 
		&=  	\l\{\begin{array}{lll}
			(1-\d_{\nu_{r}, \nu_{s}}) & \mbox{ if } & \e_r \neq \e_s,\\ 
			2i\rho_5(1-\d_{\nu_{r}, \nu_{s}}) & \mbox{ if} & \e_r = \e_s=1, \\
			-2i\rho_5(1-\d_{\nu_{r}, \nu_{s}}) & \mbox{ if} & \e_r = \e_s=*, \\
		\end{array}
		\r. \\
		&	= (1-\d_{\nu_{r}, \nu_{s}}) \big[\e'_{r} 2i\rho_5 \big]^{\d_{\e'_{r}, \e'_{s} }}.
	\end{align*}
Also, from a similar calculation, we have
	\begin{align*}
\E[a^{\e_r}_{i_r}b^{\e_s}_{i_s}]
	&	= (1-\d_{\nu_{r}, \nu_{s}}) \big[\rho_2+\rho_6 + \e'_{r} \mathrm{i}(\rho_4-\rho_3) \big]^{(1-\d_{\e'_{r}, \e'_{s} })} \big[\rho_2-\rho_6 + \e'_{r} \mathrm{i}(\rho_4+\rho_3) \big]^{\d_{\e'_{r}, \e'_{s} }}, \\
	\E[a^{\e_s}_{i_r}b^{\e_r}_{i_s}]
	&	= (1-\d_{\nu_{r}, \nu_{s}}) \big[\rho_2+\rho_6 + \e'_{s} \mathrm{i}(\rho_4-\rho_3) \big]^{(1-\d_{\e'_{r}, \e'_{s} })} \big[\rho_2-\rho_6 + \e'_{s} \mathrm{i}(\rho_4+\rho_3) \big]^{\d_{\e'_{r}, \e'_{s} }}.
\end{align*}
Thus
\begin{align*}
 \mathcal E_{r,s}(\underline i_{2k}) 
		 =&\Big[ \big[\e'_{r} 2i\rho_1 \big]^{\d_{\e'_{r}, \e'_{s} }} f_1+ \big[\e'_{r} 2i\rho_5 \big]^{\d_{\e'_{r}, \e'_{s} }}f_4 +  \big[\rho_2+\rho_6 + \e'_{r} \mathrm{i}(\rho_4-\rho_3) \big]^{(1-\d_{\e'_{r}, \e'_{s} })} \\
		& \quad \times  \big[\rho_2-\rho_6 + \e'_{r} \mathrm{i}(\rho_4+\rho_3) \big]^{\d_{\e'_{r}, \e'_{s} }} f_2 +  \big[\rho_2+\rho_6 + \e'_{s} \mathrm{i}(\rho_4-\rho_3) \big]^{(1-\d_{\e'_{r}, \e'_{s} })} \\
		& \quad \times \big[\rho_2-\rho_6 + \e'_{s} \mathrm{i}(\rho_4+\rho_3) \big]^{\d_{\e'_{r}, \e'_{s} }} f_3 \Big] (1-\d_{\nu_{r}, \nu_{s}}),
	\end{align*}
	where 
	\begin{align*}
		&f_1=\one_{\AA_{\e_r'i_r}}(m_r)\one_{\AA_{\e_s'i_s}}(m_s),\; f_2=\one_{\AA_{\e_r'i_r}}(m_r)\one_{\BB_{\e_s'i_s}}(m_s),\\&
		f_3=\one_{\AA_{\e_s'i_s}}(m_s)\one_{\BB_{\e_r'i_r}}(m_r),\; 
		f_4=\one_{\BB_{\e_r'i_r}}(m_r)\one_{\BB_{\e_s'i_s}}(m_s),
	\end{align*}
with $m_t$ as in (\ref{eqn:m_t_Hn}). Note that for even $p$, say $2k$, $m_t$ can also be written as the following
	\begin{align} \label{eqn:m_t_nut_Hn}
	m_t = \displaystyle	j+\sum_{\ell=t}^{2k} \nu_\ell i_\ell,
\end{align}
where $\nu_\ell$ is as in (\ref{eqn:nu_t_oddeven}).	Observe that, $\mathcal E_{r,s}(\underline i_{2k})$ implies that the number of free indices among $\{i_1,\ldots, i_{2k}\}$ is $k$, as the indices have relations $i_r=i_s$.
	Let $m_t'=\frac{1}{n}(m_t)$, where $m_t$ is as in (\ref{eqn:m_t_nut_Hn}). Then we have 
	\begin{align*}
		&f_1=\one_{\AA_{\e_r'\frac{i_r}{n}}}(m_r')\one_{\AA_{\e_s'\frac{i_s}{n}}}(m_s'),\; f_2=\one_{\AA_{\e_r'\frac{i_r}{n}}}(m_r')\one_{\BB_{\e_s'\frac{i_s}{n}}}(m_s'),\\&
		f_3=\one_{\AA_{\e_s'\frac{i_s}{n}}}(m_s')\one_{\BB_{\e_r'\frac{i_r}{n}}}(m_r'),\; 
		f_4=\one_{\BB_{\e_r'\frac{i_r}{n}}}(m_r')\one_{\BB_{\e_s'\frac{i_s}{n}}}(m_s').
	\end{align*}
	Let $\tilde \AA_z$ and $\tilde \BB_z$ be as in (\ref{eqn:A_x,B_x}).
	Recall, if $\pi=(r_1,s_1)\cdots (r_k,s_k)$, then we have $\pi'(r_t)=\pi'(s_t)=t$.
	Now for a set of variables $z_0, z_1, \ldots, z_{2k}$, we define $w_d=(z_0+\sum_{\ell=d}^{2k} \nu_\ell z_{\pi'(\ell)})$,
	 also define
	\begin{align*} 
		&f_1'=\one_{\tilde\AA_{\e_r'z_{\pi'(r)}}}(w_r)\one_{\tilde\AA_{\e_s'z_{\pi'(s)}}}(w_s),\; f_2'=\one_{\tilde\AA_{\e_r'z_{\pi'(r)}}}(w_r)\one_{\tilde\BB_{\e_s'z_{\pi'(s)}}}(w_s),\\
		&	f_3'=\one_{\tilde\BB_{\e_r'z_{\pi'(r)}}}(w_r)\one_{\tilde\AA_{\e_s'z_{\pi'(s)}}}(w_s),\; 
		f_4'=\one_{\tilde\BB_{\e_r'z_{\pi'(r)}}}(w_r)\one_{\tilde\BB_{\e_s'z_{\pi'(s)}}}(w_s). \nonumber
	\end{align*}
	Then by the convergence of Riemann sums, we get
	\begin{align*}
		\lim_{n \to \infty}\vp_n( \frac{H_{n}^{\e_1}}{\sqrt{n}} \cdots \frac{H_{n}^{\e_{2k}}}{\sqrt{n}}) 
		&=\sum_{\pi\in \mathcal P_2(2k)} \int_{z_0=0}^1\int_{[-1,1]^k}\prod_{(r,s)\in \pi} \mathcal E_{r,s}^{(H)}({\underline z_{2k}})dz_0 dz_1 \cdots dz_k,
	\end{align*}
	where $\underline z_{2k}=(z_0,z_1,\ldots,z_{2k})$ and 
	\begin{align}\label{eqn:mathE'(z2k)_H}
		\mathcal E_{r,s}^{(H)}(\underline z_{2k})  
		=&\Big[ \big[\e'_{r} 2i\rho_1 \big]^{\d_{\e'_{r}, \e'_{s} }} f'_1+ \big[\e'_{r} 2i\rho_5 \big]^{\d_{\e'_{r}, \e'_{s} }}f'_4 +  \big[\rho_2+\rho_6 + \e'_{r} \mathrm{i}(\rho_4-\rho_3) \big]^{(1-\d_{\e'_{r}, \e'_{s} })} \nonumber\\
		& \quad \times  \big[\rho_2-\rho_6 + \e'_{r} \mathrm{i}(\rho_4+\rho_3) \big]^{\d_{\e'_{r}, \e'_{s} }} f'_2 +  \big[\rho_2+\rho_6 + \e'_{s} \mathrm{i}(\rho_4-\rho_3) \big]^{(1-\d_{\e'_{r}, \e'_{s} })}  \nonumber\\
		& \quad \times \big[\rho_2-\rho_6 + \e'_{s} \mathrm{i}(\rho_4+\rho_3) \big]^{\d_{\e'_{r}, \e'_{s} }} f'_3 \Big] (1-\d_{\nu_{r}, \nu_{s}}).
	\end{align}
\end{remark}


\begin{proof}[Proof of Proposition \ref{thm:jc_T+P_g}]
Here we skip the details. Using the idea of the proof of Proposition \ref{thm:generalT_com} and Remark \ref{cor:jc_H_g}, 
one can show that
	\begin{align} \label{eqn:lim_phi(Te*,P)_g}
	&	\lim_{n \to \infty} \vp_n\big((P_nT_{n,g}^{(\tau_1)\e_1} \cdots T_{n,g}^{(\tau_{k_1})\e_{k_1}})  \cdots (P_nT_{n,g}^{(\tau_{k_{p-1}+1})\e_{k_{p-1}+1}} \cdots T_{n,g}^{(\tau_{k_p})\e_{k_p}}) \big) \nonumber \\
	=&
	\l\{\begin{array}{lll}
		\displaystyle \sum_{\pi\in \mathcal P_{2}(2k)} \int_{z_0=0}^1\int_{[-1,1]^k}  \prod_{(r,s)\in \pi} \d_{\tau_r,\tau_s} \mathcal E_{r,s}^{(T,P)}({\underline z_{2k}}) \prod_{i=0}^kdz_i, &  \mbox{ if $k_p=2k$ and $p$ is even}, \\
		0&   \mbox{ otherwise},
	\end{array}\r.
\end{align}
where $\mathcal E_{r,s}^{(T,P)}({\underline z_{2k}})$ is some function of the form  (\ref{eqn:mathE'(z2k)_H}). 
\end{proof}
\subsection{Final arguments in the proof of Theorem \ref{thm:gen_tdp_com}} \label{subsec:Tdp,g}

We use the ideas from 
the proofs of Propositions \ref{thm:generalT_com} to 
\ref{thm:jc_T+P_g}.
We mention only the main steps. 

\begin{proof}[Proof of Theorem \ref{thm:gen_tdp_com}]
Let $(a^{(\tau)}_i)_{i\in \Z}, (b^{(\tau)}_i)_{i\in \Z}$ be the input sequences for $T^{(\tau)}_{n,g}$ and $(d^{'(\tau)}_i)_{i\in \Z}, (d^{''(\tau)}_i)_{i\in \Z}$ be the input sequences for $D^{(\tau)}_{n,g}$.	Same as earlier, it is sufficient to check the convergence for the following monomial from the collection $\{P_n, n^{-1/2}T_{n,g}^{(i)},  D_{n,g}^{(i)};  1 \leq i \leq m\}$:
\begin{align*}
	&(P_n B_{n,\mu_1}^{(\tau_1) \e_1} \cdots B_{n,\mu_{k_1}}^{(\tau_{k_1}) \e_{k_1}} ) (P_n B_{n,\mu_{k_1+1}}^{(\tau_{k_1+1}) \e_{k_1+1}} \cdots B_{n,\mu_{k_2}}^{(\tau_{k_2}) \e_{k_2}} ) \cdots (P_n B_{n,\mu_{k_{p-1}+1}}^{(\tau_{k_{p-1}+1}) \e_{k_{p-1}+1}} \cdots B_{n,\mu_{k_p}}^{(\tau_{k_p}) \e_{k_p}} ) \nonumber \\
		& = q_{k_p}(P,D_g,T_g),  \mbox{ say},
\end{align*}
	where $\e_i \in \{1, *\}$, $\tau_i \in \{1,2, \ldots, m\}$ and for $\mu_i \in \{1,2\}$, $B^{(\tau_i) \e_i}_{n,1}=n^{-1/2}T^{(\tau_i) \e_i}_{n,g}$, $B^{(\tau_i) \e_i}_{n,2}=D^{(\tau_i) \e_i}_{n,g}$.
	
	First note that if
	$p$ is odd, then $ \vp_n \big(q_{k_p}(P,D_g,T_g)\big)=o(1)$.
	Let $p$ be even.
	 Then from Lemma \ref{lem:hankel_g}, we have
	\begin{align*}
	& \vp_n \big(q_{k_p}(P,D_g,T_g)\big) \nonumber \\
 &	= \frac{1}{n^{1+\frac{w_p}{2} }} \sum_{j=1}^n \sum_{I_{k_p}} \prod_{e=1}^p \prod_{t=k_{e-1}+1}^{k_e} (z^{(\tau_t)\e_t}_{i_t}\one_{\AA_{\e_t'i_t}}+z^{'(\tau_t)\e_t}_{i_t}\one_{\BB_{\e_t'i_t}})  (m_{t,k_e})    \d_{0,\sum_{c=1}^p(-1)^{p-c} \sum_{\ell=k_{c-1}+1}^{k_{c}}  \e'_\ell i_\ell},
	\end{align*}
	where $z^{(\tau_t) \e_t}_{i_t} $ is $a^{(\tau_t)\e_t}_{i_t}$ or $d^{'(\tau_t)\e_t}_{i_t}$ depending on whether the matrix $B_{n,\mu_t} ^{(\tau_t)\e_t}$ is $T_{n,g}^{(\tau_t)\e_t}$ or $D_{n,g}^{(\tau_t)\e_t}$; $z^{'(\tau_t) \e_t}_{i_t}$ is $ b^{(\tau_t)\e_t}_{i_t}$ or $d^{''(\tau_t)\e_t}_{i_t}$ based on $B_{n,\mu_t}^{(\tau_t)\e_t}$ is $T_{n,g}^{(\tau_t)\e_t}$ or $D_{n,g}^{(\tau_t)\e_t}$;
	 $I_{k_p}$ and $ m_{t,k_e}$ are as in (\ref{eqn:i_k in -n to n}) and (\ref{eqn:m_chi,t_Hn_g}), respectively; and
	\begin{align}\label{eqn:no of T in Q_g}
	w_p= \# \{ \mu_t : B_{n,\mu_t}^{(\tau_t) \e_t} = B_{n,1}^{(\tau_t) \e_t} \mbox{ in }  q_{k_p}(P,D_g,T_g)\}.
	\end{align}
	Note that if $w_p$ is odd, then  
	$\vp_n \big(q_{k_p}(P,D_g,T_g)\big)=o(1)$. 
	
	Now  for $c=1,2, \ldots, p$, let $ u_{r_{c-1}}, u_{r_{c-1}+1}, \ldots, u_{r_{c}}$ be the positions of $D_{n,g}$ and $ v_{w_{c-1}}, v_{w_{c-1}+1}, \ldots, v_{w_{c}}$ be the positions of $T_{n,g}$ between $B_{n,\mu_{k_{c-1}+1}}^{(\tau_{k_{c-1}+1}) \e_{k_{c-1}+1}}$ and $B_{n,\mu_{k_{c}}}^{(\tau_{k_c}) \e_{k_c}}$. 
 Here $r_0=k_0=w_0=1$ and $r_p+w_p=k_p$. 
	Suppose 
	\begin{equation*} 
		R=([k_p] \setminus \cup_{c=1}^p \{ u_{r_{c-1}}, u_{r_{c-1}+1}, \ldots, u_{r_{c}}\}).
	\end{equation*}
Note from (\ref{eqn:no of T in Q_g}) that $\# R=w_p$. Let $w_p=2k$.
	Then using arguments similar to 
 those used while establishing (\ref{eqn:lim_phi_P,D*_g}) and
	(\ref{eqn:lim_phi(Te*,P)_g}), we have 
	\begin{align*}
		&	\lim_{n \to \infty} \vp_n \big(q_{k_p}(P,D_g,T_g)\big) \nonumber \\
		&	= 
	\l\{\begin{array}{lll}
	\displaystyle \hskip-5pt \int_{z_0=0}^1	\sum_{i_{u_{r_0}}, \ldots, i_{u_{r_p}}=-\infty}^{\infty} \prod_{t=1}^{r_p} (d^{'(\tau_{u_t})\e_{u_t}}_{i_{u_t}}\one_{[0,1/2]}(z_0) +d^{''(\tau_{u_t})\e_{u_t}}_{i_{u_t}}\one_{(1/2,1]} (z_0))   \\
	   \times \d_{0,\sum_{c=1}^{p}(-1)^c \sum_{\ell=r_{c-1}+1}^{r_c} \e'_{u_\ell} i_{u_\ell}} \displaystyle \hskip-3pt \sum_{\pi\in \mathcal P_{2}(R)} \int_{[-1,1]^k}  \prod_{(r,s)\in \pi} \hskip-5pt \d_{\tau_r,\tau_s} \mathcal E_{r,s}^{(T,P)}({\underline z_{2k}}) \prod_{i=0}^kdz_i  &  \mbox{ if  $w_p=2k$ and $p$ is even}, \\
			0&   \mbox{ otherwise},
		\end{array}\r.
	\end{align*}
	where $\mathcal E_{r,s}^{(T,P)}({\underline z_{2k}})$ is as in (\ref{eqn:lim_phi(Te*,P)_g}) for pair-partitions of the set $R$.
	This completes the proof of Theorem \ref{thm:gen_tdp_com}.
\end{proof}

\section{Conclusion}
The joint convergence of independent copies of real symmetric Toeplitz $T_{n,s}$ and Hankel $H_{n,s}$ matrices was already known \cite{bose_saha_patter_JC_annals}. No results were known for the non-symmetric versions.

We have considered independent random Toeplitz matrices $T_n$ with complex input entries that have a pair-correlation structure, along with deterministic Toeplitz matrices $D_n$ and the backward identity permutation matrix $P_n$. We have first 
established the joint convergence of $\{T_n, D_n, P_n\}$ (Theorem \ref{thm:JC_tdp}). This provides a generalization (in terms of input entries and correlation) of the earlier result \cite{bose_saha_patter_JC_annals}. In particular, the following known results follow from Theorem \ref{thm:JC_tdp}: 
LSD results on $T_{n,s}$ and any symmetric matrix polynomials \cite{bose_sen_LSD_EJP}; joint convergence of $H_{n,s}$ \cite{bose_saha_patter_JC_annals}; LSD results on $H_{n,s}$ \cite{bose_sen_LSD_EJP}.

Liu and Wang \cite{liu_wang2011} appear to be the first to provide a representation of $H_{n,s}$ in terms of  $T_n$ and $P_n$ ($H_{n,s}=P_nT_n$). This relation simplifies the study of $H_{n,s}$. In \cite{liu_wang2011} and \cite{liu2012fluctuations} this relation was used successfully. 

We have introduced a generalized Toeplitz matrix $T_{n,g}$, and have exploited the relation $H_n=P_nT_{n,g}$ to extend Theorem \ref{thm:JC_tdp} to $T_{n,g}$ and related matrices (Theorem \ref{thm:gen_tdp_com}). This in particular implies the joint convergence of independent asymmetric Hankel matrices $H_n$ (Remark \ref{cor:jc_H_g}). In all cases, the limits are universal, depending only on the correlation structure.

Finally, the relation $H_n=P_nT_{n.g}$ may play an important role in solving the open question of convergence of the ESD of the non-symmetric matrix $H_n$.

%


\end{document}